\documentclass[10pt]{amsart}
 
\usepackage{amsthm,amsmath,amssymb}
\usepackage{graphicx}
\usepackage[comma, sort&compress]{natbib}
\usepackage{tabularx}
\usepackage{multirow}
\usepackage[hyperindex,breaklinks]{hyperref}

\usepackage{paralist}

\usepackage[left=1.4in,right=1.4in,top=1.4in,bottom=1.4in]{geometry} 

\newtheorem{theorem}{Theorem}

\newtheorem{lemma}[theorem]{Lemma} 

\newtheorem{corollary}[theorem]{Corollary}

\newtheorem{assumption}[theorem]{Assumption}
\theoremstyle{definition} 

\theoremstyle{remark} 
\newtheorem{example}{Example}

\newtheorem*{claim*}{Claim}

\numberwithin{equation}{section}
\numberwithin{theorem}{section}
\numberwithin{example}{section}
\numberwithin{definition}{section}
\numberwithin{remark}{section}

\numberwithin{figure}{section}

\DeclareMathOperator{\diag}{diag}

\newcommand{\assumpref}[1]{Assumption~\ref{assump:#1}}

\newcommand{\secref}[1]{Section~\ref{sec:#1}}

\newcommand{\exref}[1]{Example~\ref{ex:#1}}

\newcommand{\exssref}[1]{\ref{ex:#1}}

\newcommand{\lemref}[1]{Lemma~\ref{lem:#1}}
\newcommand{\lemsref}[1]{Lemmas~\ref{lem:#1}}
\newcommand{\lemssref}[1]{\ref{lem:#1}}

\newcommand{\thmref}[1]{Theorem~\ref{thm:#1}}

\newcommand{\thmsref}[1]{Theorems~\ref{thm:#1}}
\newcommand{\thmssref}[1]{\ref{thm:#1}}

\newcommand{\tabref}[1]{Table~\ref{tab:#1}}

\title[]{Testing independence in high dimensions with sums of rank correlations}

\author[D.~Leung]{Dennis Leung} 
\address{Department of Statistics, University of Washington, Seattle,
  WA, U.S.A.}
\email{dmhleung@uw.edu}

\author[M.~Drton]{Mathias Drton} 
\address{Department of Statistics, University of Washington, Seattle,
  WA, U.S.A.}
\email{md5@uw.edu}

\begin{document}

\begin{abstract}
  We treat the problem of testing independence between $m$ continuous
  variables when $m$ can be larger than the available sample size
  $n$. We consider three types of test statistics that are constructed
  as sums or sums of squares of pairwise rank correlations. 
%
  In the asymptotic regime where  both $m$ and $n$ tend to infinity, a martingale central limit theorem
  is applied to show that the null distributions of these statistics
  converge to Gaussian limits, which are valid with no specific distributional or moment assumptions on the data.  Using the framework of
  U-statistics, our result covers a variety of rank correlations
  including Kendall's tau and a dominating term of Spearman's rank
  correlation coefficient (rho), but also degenerate U-statistics such
  as Hoeffding's
   $D$, or the $\tau^*$ of \citet{bergsma:2014}.
   As in the classical theory for U-statistics, the test statistics need to be scaled differently when the rank correlations used to construct them are degenerate U-statistics.    
%
  The power of the
  considered tests is explored in rate-optimality theory under a
  Gaussian equicorrelation alternative as well as in numerical
  experiments for specific cases of more general alternatives.
\end{abstract}



\maketitle

\section{Introduction} \label{sec:intro}

This paper is concerned with nonparametric tests of independence
between the coordinates of a continuous random vector
${\bf X}=(X^{(1)}, \dots, X^{(m)})$.  Let ${\bf X}_1,\dots,{\bf X}_n$ be an
i.i.d.~sample, with each ${\bf X}_i=(X_i^{(1)},\dots,X_i^{(m)})$
following the same distribution as ${\bf X}$.  We then wish to test the null
hypothesis
\begin{equation}\label{eq:H0}
H_0:\; X^{(1)},  \dots,  X^{(m)}
\text{ are independent}.
\end{equation}
The natural approach is to form a test statistic that measures the
dependence among the variables $X^{(1)},  \dots,  X^{(m)}$ based on the sample, and reject $H_0$ when its value is too large, where the critical value of rejection is calibrated by the asymptotic distribution of the test statistic under the null.  Our focus is on the use of rank correlations in problems where the dimension $m$ can be larger than the sample size $n$.  Specifically, our testing procedures will be studied under the asymptotic regime where $m = m(n)$ grows as a function
of $n$ such that $m$ also tends to infinity. This regime is denoted by $m, n \longrightarrow \infty$ throughout our paper. 


There is a vast literature on the problem of testing independence.  If
$\bf X$ is normal, then under the traditional asymptotic setup in
which $n$ goes to $\infty$ while $m$ is fixed, the likelihood ratio
test (LRT) statistic converges to a chi-square
distribution when $H_0$ is true
\citep{andersonTedMulti}.  This test is known to be unimplementable for
$m > n$ due to the singularity of the sample covariance matrix, but
recent work of \citet[Corollary 1]{JiangQi} shows asymptotic
normality for the LRT statistic under the regime where
$m, n \longrightarrow \infty$ while $n > m + 4$. When $m$ can
actually be larger than $n$, one line of work uses the \emph{maximum}
of many pairwise dependency measures to test for~(\ref{eq:H0}). For
$p=1,\dots, m$, let $\mathbf{X}^{(p)}=(X^{(p)}_1,\dots,X^{(p)}_n)$ be
the sample of observations for the $p$-th variable.  For
$1\le p\not= q\le m$, let $r^{(pq)}$ denote the sample Pearson
(product-moment) correlation of $\mathbf{X}^{(p)}$ and
$\mathbf{X}^{(q)}$.  \citet{jiang2004} proved that, under suitable
centering and scaling, the null distribution of the statistic
\begin{equation}\label{max_stat}
\max_{1\leq p < q \leq m}  \big(r^{(pq)}\big)^2
\end{equation}
converges to an extreme value distribution of type $1$ when $ m/n$
converges to a constant $\gamma \in (0, \infty)$ as
$m, n \longrightarrow \infty$.  We will abbreviate such convergence as
$m/n\longrightarrow \gamma \in (0, \infty)$.  He assumed
higher-order moment conditions that were weakened in subsequent work
\citep{zhou2007, liu2008, li2010, li2010followup}.  \citet{cai2011}
derived a similar asymptotic distribution for the statistic
from~(\ref{max_stat}), allowing for subexponential growth in the
dimension $m$.  Further weakening distributional assumptions, the
recent work of \cite{HanLiu2014} treated maxima of rank correlations,
that is, the sample Pearson correlation in \eqref{max_stat} is
replaced by a rank correlation measure such as Kendall's tau.  This
maximum was shown to have a similar extreme value type null
distribution.  Statistics such as~(\ref{max_stat}) are of obvious
appeal when strong dependence is expected between some variables.

This paper, however, aligns with a different approach that is
appealing when moderate dependence is expected between many variables.
In this approach, tests are based on estimates of the \emph{sum} of
many pairwise dependency signals.  Let $\Sigma = (\sigma^{(pq)})$ and
$R = (\rho^{(pq)})$ be, respectively, the population covariance and
Pearson correlation matrix of the random vector $\bf X$.  Under a Gaussian
assumption for ${\bf X}$, \cite{Schott2005} proposed the use of the
``plug-in" estimate
\begin{equation} \label{SchottStat}
 S_{r} := \sum_{1 \leq p < q \leq m} \big(r^{(pq)} \big)^2
 \end{equation}
 for the overall dependency signal $\sum_{p < q} (\rho^{(pq)})^2$.
 Subsequent work of \cite{ChenShao} obtained a Berry-Esseen bound for
 this statistic's weak convergence to normality under $H_0$ as
 $m, n \longrightarrow \infty$.  The statistic $S_r$ is in fact Rao's
 score statistic for the multivariate normal setting; see
 Appendix~\ref{sec:Rao}. \cite{mao2014} suggested a related statistic,
 namely, the sum of $f(r^{(pq)})$ for $f(x)=x^2/(1-x^2)$, and again
 the null distribution is shown to be asymptotically normal.  For the
 two related problems of testing the equality and the proportionality
 of $\Sigma$ to the identity matrix, similar statistics have been
 studied \citep{ledoit:2002, nagao:1973, john1972}.
Motivated by this approach, we construct our first class of test statistics by plugging in rank correlations to obtain nonparametric tests for~(\ref{eq:H0}).  We illustrate it here for Kendall's tau.
For $1\le p\not=q\le m$, let
\begin{equation}\label{tau_X_form}
\tau^{(pq)} = {n \choose 2}^{-1} \sum_{1 \leq i < j \leq n}
\text{sgn}\left( X^{(p)}_i -X^{(p)}_j\right) \text{sgn}
\left(X^{(q)}_i -X^{(q)}_j\right)
\end{equation}
be the sample Kendall's tau correlation coefficient for
${\bf X}^{(p)}$ and ${\bf X}^{(q)}$. A natural test is then to reject
$H_0$ for large values of the statistic 
\begin{equation}
  \label{eq:kendall-sample-sum}
 S_\tau := \sum_{1 \leq p < q \leq m}\big(\tau^{(pq)}\big)^2.
\end{equation}

As an estimator of the dependency signal
\begin{equation}
  \label{eq:kendall-signal-strength}
\sum_{1 \leq p < q \leq m} \left(\mathbb{E}\left[\tau^{(pq)}
\right]\right)^2,
\end{equation}
the ``plug-in" statistic $S_\tau$ from~(\ref{eq:kendall-sample-sum})
is biased and thus needs to be recentered to obtain a mean zero
asymptotic null distribution under our considered regime
$m, n \longrightarrow \infty$.
Alternatively, we may instead attempt to form an unbiased estimator
of~(\ref{eq:kendall-signal-strength}) to serve as a test statistic.  As shown in \secref{main},
such an unbiased estimator is given by
\begin{multline} \label{kendall_W_kern_stat}
  T_\tau := \frac{1}{4!\binom{n}{4}} \sum
  \text{sgn}\left(X^{(p)}_{i_{\pi(1)}} - X^{(p)}_{i_{\pi(2)}}\right)
  \text{sgn}\left(X^{(p)}_{i_{\pi(3)}} - X^{(p)}_{i_{\pi(4)}}\right)\\
  \times
  \text{sgn}\left(X^{(q)}_{i_{\pi(1)}} - X^{(q)}_{i_{\pi(2)}}\right)
  \text{sgn}\left(X^{(q)}_{i_{\pi(3)}} - X^{(q)}_{i_{\pi(4)}}\right),
 \end{multline}
 where the summation is over all variable pairs $1 \leq p < q \leq m$,
  ordered 4-tuples of indices $1 \leq i_1 < i_2 < i_3 < i_4 \leq n$,
 and permutations $\pi$ on four elements.
 This type of statistics is motivated by the work of \cite{Chen2010}
and \cite{CaiAndMa}, who tested the equality of $\Sigma$ to the identity based on unbiased estimates
of the squared Frobenius norm $\|\Sigma - I_m\|_F^2$, where $I_m$ is the $m$-by-$m$ identity matrix. Under a Gaussian assumption for ${\bf X}$,  \cite{CaiAndMa} showed their test to be asymptotically minimax rate optimal. 

As a last variant, when testing for \emph{positive} associations, it may be of interest to consider the statistic
\begin{equation} \label{sumoftaus}
 Z_\tau := \sum_{1 \leq p < q \leq m} \tau^{(pq)},
\end{equation}
which sums all pairwise sample correlations for a
``one-sided" test.  As we explain below, such a statistic also
provides a ``two-sided" test for $H_0$ when rank correlations such as the
$\tau^*$ of \citet{bergsma:2014} are used. In
Section~\ref{sec:asympt-null-distr}, we show that all the
statistics introduced above are asymptotically normal under suitable
recentering and rescaling.

 Kendall's tau is an example of a U-statistic whose values depend
 on the data  only via ranks \citep[Example 12.5]{VWbook}.  Indeed, the
 values of \eqref{tau_X_form}, \eqref{kendall_W_kern_stat} and \eqref{sumoftaus} remain
 unchanged if each observation $X^{(p)}_i$ is replaced with its rank
 $R^{(p)}_i$.  To be specific, 
 $R^{(p)}_i$ is the rank of $X^{(p)}_i$ among
 $X^{(p)}_1, \dots, X^{(p)}_n$.
 Other examples of measures of association that are both U-statistics
 and rank correlations are the $D$ of \citet{hoeffdingIndep} and the
 aforementioned 
 $\tau^*$ of \citet{bergsma:2014}.  We note that for a pair of continuous
 random variables both of these statistics lead to consistent tests of
 independence, that is, their expectations are zero if and only if the
 two random variables are independent.  Another classical example is
 Spearman's rho, which is not a U-statistic but can be approximated by
 a rank-based U-statistic.

 The above examples of U-statistics are reviewed in
 Section~\ref{sec:rank-cor}, which also introduces a general
 framework of rank-based U-statistics that we adopt for a unified
 theory.  In \secref{main} we construct our classes of test
 statistics for the null hypothesis $H_0$ from~(\ref{eq:H0}).  Their
 asymptotic null distributions when
 $m, n \longrightarrow \infty$  are derived in
 Section~\ref{sec:asympt-null-distr}.  Our arguments make use of a
 central limit theorem for martingale arrays and U-statistic theory.
 We emphasize that all our statistics admit a normal limit after appropriate rescaling, 
but just as in the classical theory for U-statistics,  the scaling factors have a different order when degenerate U-statistics are considered.
 In \secref{power}, we explore aspects of power of our tests from a minimax point of view.
 Simulation experiments are presented in \secref{simulations},
 which also discusses computational considerations in the
 implementation of the tests.  Throughout, for our null distributional theory, we make no
 distributional or moment assumption on $(X^{(1)}, \dots, X^{(m)})$ other than
 that it is a continuous random vector.  This assumption is needed to
 avoid ties in observations and ranks.  We conclude with a brief
 discussion in Section~\ref{sec:conclude}.

\subsection{Notational convention}
For $p\in\{1,\dots,m\}$, we let
$\mathbf{R}^{(p)} := (R_1^{(p)},\dots, R_n^{(p)})$ be the vector of
ranks of $\mathbf{X}^{(p)}= (X_1^{(p)},\dots, X_n^{(p)})$.  The
symmetric group of order $l$ is denoted by $\mathfrak{S}_l$.
Depending on the context, its elements are treated either as
permutation functions or as ordered tuples from the set $\{1, \dots, l\}$.
For $k \leq n$, $\mathcal{P}(n,k)$ denotes the set of $k$-tuples
$\mathbf{i} = (i_1, \dots, i_k)$ with
$1 \leq i_1 < \dots < i_k \leq n$, and we will also identify the tuple
$\mathbf{i}$ with its set of elements $\{i_1, \dots, i_k\}$.  Hence,
for any two elements ${\bf i}\in \mathcal{P}(n, k_1)$ and ${\bf j} \in \mathcal{P}(n, k_2)$ with $1 \leq k_1, k_2 \leq n$, the
operations ${\bf i}\cup {\bf j}$, ${\bf i}\cap {\bf j}$, and
${\bf i}\setminus {\bf j}$ give the tuples with increasing components
that, as sets, equal the union, intersection and difference of
${\bf i}$ and ${\bf j}$, respectively.  For
$\mathbf{i}\in\mathcal{P}(n,k)$, we let
${\bf X}^{(p)}_{\bf i}:= (X^{(p)}_{i_1}, \dots, X^{(p)}_{i_k})$, and
define the rank vector
\[
\mathbf{R}_{\mathbf{i}}^{(p)} := \left(R^{(p)}_{\mathbf{i}, 1}, \dots,
  R^{(p)}_{\mathbf{i}, k} \right),
\]
where $R^{(p)}_{\mathbf{i}, c}$ is the rank of $X^{(p)}_{i_{c}}$ among
$X^{(p)}_{i_1}, \dots, X^{(p)}_{i_k}$.  

Let $p\not=q$ index two distinct variables.  Then ${\bf X}^{(pq)}_c$
and ${\bf R}^{(pq)}_c$ denotes the pairs $(X^{(p)}_c, X^{(q)}_c)$ and
$(R^{(p)}_c, R^{(q)}_c)$, respectively, for $c = 1, \dots, n$.
Similarly, given
$\mathbf{i} = (i_1, \dots, i_k)\in \mathcal{P}(n, k)$, we let
$\mathbf{X}^{(pq)}_{\mathbf{i}, c} := (X^{(p)}_{i_c}, X^{(q)}_{i_c})$
and
$\mathbf{R}^{(pq)}_{\mathbf{i}, c} := (R^{(p)}_{{\mathbf i}, c},
R^{(q)}_{\mathbf{i}, c})$
for $c \in \{1, \dots, k\}$.  We then define the $k$-tuples that are observation
and rank vectors of pairs:
\[
 {\bf R}^{(pq)}_{\bf i}  := \left(\mathbf{R}^{(pq)}_{{\mathbf i}, 1}, \dots, \mathbf{R}^{(pq)}_{{\bf i}, k} \right) \text{ and }
 {\bf X}^{(pq)}_{\bf i}  := \left(\mathbf{X}^{(pq)}_{{\bf i}, 1}, \dots, \mathbf{X}^{(pq)}_{{\bf i}, k} \right).
\]

When taking expectations under the null hypothesis $H_0$, we write
$\mathbb{E}_0[\cdot]$, whereas $\mathbb{E}[\cdot]$ is the general
expectation operator, possibly under alternative hypotheses.
Similarly, we write $P_0[\cdot]$, $P[\cdot]$, $\text{Var}_0[\cdot]$,
$\text{Var}[\cdot]$, $\text{Cov}_0[\cdot]$ and $\text{Cov}[\cdot]$ for
the probability, variance and covariance operator under $H_0$ and
possibly alternatives respectively.  Finally, $\| \cdot \|_\infty$ and
$\|\cdot\|_2$ are the max norm and Euclidean norm for vectors,
respectively, and the Froebenius norm of a matrix is denoted by
$\|\cdot\|_F$. For two sequences $(a_n)$ and $(b_n)$, the symbol
$a_n \asymp b_n$ is used to indicate the existence of 
constants $c, C > 0$ such that $ c |a_n|\leq |b_n|\leq C|a_n|$ for all
indices $n$.

\section{Rank correlations as U-statistics} \label{sec:rank-cor}

This section lays out a rank-based U-statistic framework that
encompasses all rank correlations we will use when constructing specific
test statistics for $H_0$ in \secref{main}. Let
\[
h : \left(\mathbb{R}^2\right)^k \longrightarrow \mathbb{R}
\]
be a symmetric function of $k\ge 2$ arguments in $\mathbb{R}^2$, i.e., for
all choices of
$\mathbf{x}_i = (x^{(1)}_i, x^{(2)}_i)\in \mathbb{R}^2$,
$i = 1, \dots, k$, and any permutation $\pi\in\mathfrak{S}_k$, it
holds that
$ h\left(\mathbf{x}_1, \dots, \mathbf{x}_k \right)=
h\left(\mathbf{x}_{\pi(1)}, \dots, \mathbf{x}_{\pi(k)} \right)$.
For any pair of distinct variable indices $p,q\in\{1,\dots,m\}$, the function
$h$ yields a \emph{U-statistic}
\begin{equation}
   \label{U_stat1}
U^{(pq)}_h =  \frac{1}{{n \choose k}} \sum_{\mathbf{i}\in\mathcal{P}(n,k) } h\left(\mathbf{X}^{(pq)}_{\mathbf{i}, 1}, \dots, \mathbf{X}^{(pq)}_{\mathbf{i}, k}\right) = \frac{1}{{n \choose k}} \sum_{\mathbf{i}\in\mathcal{P}(n,k) } h\left(\mathbf{X}^{(pq)}_{\mathbf{i}}\right).
\end{equation}
In this context, $h$ is termed the \emph{kernel} of the U-statistics
and is said to be of \emph{degree} $k$.

Subsequently, we always assume that the kernel $h$ and the
induced U-statistics from~(\ref{U_stat1}) are \emph{rank-based}, that
is, the kernel has the property that
$h(\mathbf{x}_1, \dots, \mathbf{x}_k ) = h(\mathbf{r}_1, \dots,
\mathbf{r}_k ) $
for all arguments
$\mathbf{x}_1, \dots, \mathbf{x}_k \in \mathbb{R}^2$.  Here, for each
argument ${\bf x}_i= (x^{(1)}_i, x^{(2)}_i)\in \mathbb{R}^2$, we let
${\bf r}_i = (r^{(1)}_i, r^{(2)}_i)$ with $r^{(j)}_i$ being the rank
of $x^{(j)}_i$ among $x^{(j)}_1, \dots, x^{(j)}_k$ for $j = 1, 2$.  If
$U^{(pq)}_h $ from~(\ref{U_stat1}) is rank-based, then
\begin{equation} 
\label{U_stat}
\displaystyle U^{(pq)}_h 
=  \frac{1}{{n \choose k}} \sum_{\mathbf{i}\in\mathcal{P}(n,k) } h\left(\mathbf{R}^{(pq)}_{\mathbf{i}, 1}, \dots, \mathbf{R}^{(pq)}_{\mathbf{i}, k}\right) = \frac{1}{{n \choose k}} \sum_{\mathbf{i}\in\mathcal{P}(n,k) } h\left(\mathbf{R}^{(pq)}_{\mathbf{i}}\right).
\end{equation}
Note that $({\bf R}^{(pq)}_1, \dots,{\bf R}^{(pq)}_n)$ uniquely
determines all $k$-tuples
$(\mathbf{R}^{(pq)}_{\mathbf{i}, 1}, \dots,
\mathbf{R}^{(pq)}_{\mathbf{i}, k})$.

The following lemma lists elementary properties of $U^{(pq)}_h$ under
$H_0$.  It relies on the fact that under $H_0$ the distribution of
$h(\mathbf{R}^{(pq)}_{\mathbf{i}})$ does not depend on the choice of
$\mathbf{i}$, $p$ and $q$ because the rank vectors
$\mathbf{R}^{(1)},\dots,\mathbf{R}^{(m)}$ are i.i.d.~according to a
uniform distribution on the symmetric group $\mathfrak{S}_n$; recall
that we assume the original observations to be continuous random
vectors such that ties among the ranks have probability zero.  A proof
of the lemma is given in Appendix~\ref{sec:proof-rank-cor}.

\begin{lemma}
  \label{lem:general_properties}
 Suppose $g(\cdot)$ is a real-valued function defined on ${(\mathbb{R}^2)}^n$, and for $1 \leq p \not= q \leq m$,
\[
g^{(pq)} := g\left({\bf R}^{(pq)}_1, \dots,{\bf R}^{(pq)}_n\right)\qquad
\]
is symmetric in the $n$ arguments ${\bf R}^{(pq)}_1, \dots, {\bf R}^{(pq)}_n$. The random variables $g^{(pq)}$ satisfy the following
  properties under $H_0$: 
  \begin{enumerate}
  \item If $p\not=q$, then $g^{(pq)}$ has the same distribution as
    $g^{(12)}$.
  \item If $p\not=q$, then $g^{(pq)}$ is independent of
    $\mathbf{X}^{(p)}$ (and also independent of
    $\mathbf{X}^{(q)}$).
  \item For any fixed $1\le l\le m$, the $m-1$ random variables
    $g^{(pl)}$, $p\not=l$, are mutually independent.
  \item If $p\not=q$, $r\not=s$ and $\{p, q\} \not = \{r, s\}$, then 
    $g^{(pq)}$ and $g^{(rs)}$ are independent.
  \end{enumerate}
\end{lemma}

In this paper we assume all kernel functions $h$ to be
\emph{bounded}.  Since $h$ can be recentered if needed, without
loss of generality, we will further assume that
$\mathbb{E}_0[h(\mathbf{R}^{(pq)}_{\mathbf{i}})] = 0$, a property
exhibited by all the examples below.

\begin{example}[Kendall's tau] \label{ex:ken_ex} If we take $h$ in \eqref{U_stat} to
  be the kernel of degree $k = 2$ given by
  \begin{equation*}
    h_\tau(\mathbf{r}_1, \mathbf{r}_2)  = \text{sgn}  \left( \left(r_1^{(1)} -  r_2^{(1)}\right)\left(r_1^{(2)} - r_2^{(2)} \right)\right),
  \end{equation*}
  then $ \tau^{(pq)} := U^{(pq)}_{h_\tau}$ is Kendall's tau, which
  measures the association of ${\bf X}^{(p)}$ and ${\bf X}^{(q)}$ by
  counting concordant versus disconcordant pairs of points.
\end{example}

\begin{example}[Spearman's rho]\label{ex:rho_ex}
Let  
\begin{equation} \label{rho}
\rho^{(pq)}_s = 1 - \frac{6}{n (n^2 - 1)} \sum_{i = 1}^n \left(R_i^{(p)} - R_i^{(q)}\right)^2.
\end{equation}
be the Spearman's rank correlation coefficient (rho) between ${\bf X}^{(p)}$
and ${\bf X}^{(q)}$. Define
$\hat{\rho}^{(pq)}_s := U_{h_{\hat{\rho}_s}}$, where
$h_{\hat{\rho}_s}$ is the kernel function of degree $3$ given by
  \begin{equation} \label{kern_rho_hat}  h_{\hat{\rho}_s} \left(\mathbf{r}_1,
      \mathbf{r}_2, \mathbf{r}_3 \right)  =  \frac{1}{2} \sum_{\pi \in \mathfrak{S}_3} \text{sgn}\left( r^{(1)}_{\pi_1} -
      r^{(1)}_{\pi_2}  \right) \text{sgn}\left( r^{(2)}_{\pi_1}  -
      r^{(2)}_{\pi_3} \right).
  \end{equation} 
\citet[p.318]{hoeffdingU} showed that
  \begin{align} 
    \rho^{(pq)}_s 
    &= 
 \frac{n - 2}{n+1} \hat{\rho}^{(pq)}_s + \frac{3}{n + 1} \tau^{(pq)}  \label{rho_decomp}.
\end{align}
Hence, the dominating term $\hat{\rho}_s$ of Spearman's rho is a U-statistic.
\end{example} 

\begin{example}[Hoeffding's $D$ statistic]  \label{ex:D_ex}
Let
\[
h_D({\bf r}_1, \cdots, {\bf r}_5) = \frac{1}{5!} \sum_{\pi \in \mathfrak{S}_5} 
\frac{\phi \left(r^{(1)}_{\pi_1}, \dots, r^{(1)}_{\pi_5}\right) \phi\left(r^{(2)}_{\pi_1}, \dots, r^{(2)}_{\pi_5}\right) }{4},
\]
where 
\[
\phi \left(r_1, \dots, r_5\right) =\left(I \left(r_1 \geq r_2\right) - I\left(r_1 \geq  r_3\right) \right)\left(I(r_1 \geq r_4) - I(r_1 \geq  r_5)\right)
\]
and $I(\cdot)$ is the indicator function.  \citet{hoeffdingIndep}
suggested the statistic $D^{(pq)} := U^{(pq)}_{h_D}$ to measure
association between the vectors ${\bf X}^{(p)}$ and ${\bf X}^{(q)}$.
When the joint distribution of $(X^{(p)},X^{(q)})$ has continuous
joint and marginal densities, the expectation
\[ \mathbb{E}\left[h_D({\bf R}_{{\bf i}, 1}^{(pq)}, \cdots, {\bf
    R}_{{\bf i}, 5}^{(pq)}) \right]
\]
is zero if and only if $X^{(p)}$ and $X^{(q)}$ are independent \citep[Theorem 3.1]{hoeffdingIndep}. 

\end{example}

\begin{example}[Bergsma and Dassios' $t^*$] \label{ex:t^*_ex} In a
  recent paper, \citet{bergsma:2014} introduced
  ${t^*}^{(pq)}:= U^{(pq)}_{h_{t^*}}$, a U-statistic of degree $4$
  with the kernel
\[
h_{t^*}({\bf r}_1, \cdots, {\bf r}_4) = \frac{1}{4!} \sum_{\pi \in \mathfrak{S}_4} \phi\left(r^{(1)}_{\pi_1}, \dots, r^{(1)}_{\pi_4}\right)
\phi\left(r^{(2)}_{\pi_1}, \dots, r^{(2)}_{\pi_4}\right),
\] 
where now
 \[
\phi(r_1, \dots, r_4) = I(r_1, r_3 < r_2, r_4) + I(r_1, r_3 > r_2, r_4) - 
I(r_1, r_2 < r_3, r_4) - I(r_1, r_2 > r_3, r_4).
 \]
 According to Theorem 1 in \cite{bergsma:2014}, $t^*$ is an
 improvement over Hoeffding's $D$ in the sense that the vanishing of
 $\mathbb{E}[h_{t^*}({\bf R}_{{\bf i}, 1}^{(pq)}, \cdots, {\bf
   R}_{{\bf i}, 4}^{(pq)}) ]$
 characterizes the independence of $X^{(p)}$ and $X^{(q)}$ under the
 weaker assumption that $(X^{(p)},X^{(q)})$ has a bivariate
 distribution that is discrete or (absolutely) continuous, or a
 mixture of both. In fact, in their paper \cite{bergsma:2014} conjectured that even this assumption is not necessary. 
\end{example}

Returning to our general setup, the variance and also the large-sample
behavior of the statistic $U^{(pq)}_h$ is determined by the
covariance quantities
\begin{equation} \label{cov_quan_0}
\zeta_c^h  \;:=\; \mathrm{Cov}\left[h\left(\mathbf{R}^{(pq)}_{\mathbf{i}} \right) 
  h\left(\mathbf{R}^{(pq)}_{\mathbf{j}} \right) \right], \quad c=0,\dots,k,
\end{equation}
where ${\bf i}, {\bf j} \in \mathcal{P}(n,k)$ are such that
$|{\bf i}\cap{\bf j}| = c$.  When $H_0$ is true,
\begin{equation} \label{cov_quan_00}
\zeta_c^h 
\; =\; 
\mathbb{E}_0\left[h\left(\mathbf{R}^{(pq)}_{\mathbf{i}} \right)
  h\left(\mathbf{R}^{(pq)}_{\mathbf{j}} \right) \right] \text{ } 
\end{equation}
as we are assuming that
$\mathbb{E}_0[h(\mathbf{R}^{(pq)}_{\mathbf{i}})] =
0$.  Furthermore, the value of
$\zeta_c^h$ does not depend on the choice of $({\bf i}, p,
q)$ under
$H_0$.  In the sequel, it will be clear from the context whether
$\zeta^h_c$ is defined under
$H_0$ or an alternative hypothesis. 

It is well known that
$0 = \zeta_0^h \leq \zeta_1^h , \dots, \leq \zeta_k^h$, and the kernel
$h$ is said to have order of degeneracy $d$ if
$\zeta_0^h = \zeta_1^h = \dots = \zeta^h_{d - 1} = 0$ and
$\zeta_{d}^h > 0$ \citep[chapter 5]{Serfling1980}.  If $d\ge 2$, the
kernel and the U-statistic it defines are referred to as degenerate.
For any $c = 1, \dots, k$, it holds under $H_0$ that
\begin{equation} \label{gen_prop} \zeta_c^h = 0 \iff
  \mathbb{E}_0\left[h\left(\mathbf{R}^{(pq)}_{\mathbf{i}}\right) \middle| {\bf
      X}^{(pq)}_{{\bf i}'} \right] = 0, \quad \text{almost surely},
\end{equation}
where
${\bf i}' \subset {\bf i}$ may be any subset with $|{\bf i}'| = c$. In
particular, for the kernels $h_D$ and $h_{t^*}$, the right-hand side
of \eqref{gen_prop} holds with $c \leq 1$.

As in the classical theory of U-statistics,
$\zeta_{d}^h$ will play a role in our asymptotic results for the test statistics
we construct from rank-based
U-statistics, for which the kernels have order of degeneracy $d=1$ or $d=2$
under $H_0$.  However, when $d=2$,  an additional
quantity is needed to describe our asymptotic results.
For a symmetric kernel
$h: (\mathbb{R}^2)^k \longrightarrow \mathbb{R}$ with order of
degeneracy $d = 2$ under $H_0$, we define
\begin{equation}
  \label{eq:eta}
\eta^h \;:=\; \mathbb{E}_0\left[ 
h\left({\bf R}^{(pq)}_{{\bf i}^1}\right)
h\left({\bf R}^{(pq)}_{{\bf i}^2}\right)
h\left({\bf R}^{(pq)}_{{\bf i}^3}\right)
h\left({\bf R}^{(pq)}_{{\bf i}^4}\right)  \right],
\end{equation}
where ${\bf i}^1, \dots {\bf i}^4 \in \mathcal{P}(n, k)$ are any four tuples such that 
\begin{enumerate}
\item $|\cup_{\omega = 1}^4 {\bf i}^\omega| = 4k - 4$,
\item $|{\bf i}^1  \cap {\bf i}^2| =  |{\bf i}^2  \cap {\bf i}^3| = |{\bf i}^3  \cap {\bf i}^4| = |{\bf i}^4  \cap {\bf i}^1| =1$, and 
\item no index $i\in \cup_{\omega = 1}^4 {\bf i}^\omega$ is an element
  of more than two of the sets ${\bf i}^1, \dots {\bf i}^4 $.
\end{enumerate} 
For our purpose we only need to define $\eta^h$ under $H_0$, and it is
also easy to see that the choice of $p$, $q$, ${\bf i}^\omega$,
$\omega = 1, \dots, 4$, does not matter in its
definition. \tabref{degen_tab} collects the order of degeneracy $d$
under $H_0$, and the quantities $\zeta_d^h$ and $\eta^h$ for the
kernels in Example~\ref{ex:ken_ex}--\ref{ex:t^*_ex}.  The latter are found in \citet{hoeffdingU, hoeffdingIndep}, and by our
own calculations.

Finally, it is easy to check that all the kernels in
\exref{ken_ex}--\exssref{t^*_ex} satisfy the following property that will
be assumed for our null asymptotic results.

\begin{assumption} \label{assump:sp_assume}
Let $h: (\mathbb{R}^2)^k \longrightarrow \mathbb{R}$ be a symmetric kernel with order of degeneracy $d \geq 1$ under $H_0$. Then given ${\bf i} = (i_1, \dots, i_k) \in \mathcal{P}(n, k)$ and $1 \leq p \not= q \leq m$,
\[
\mathbb{E}_0 \left[h\left({\mathbf R}_{\bf i}^{(pq)}\right)\middle|  {\bf X}^{(p)}_{\bf j}, {\bf X}^{(q)}_{\bf j'}\right] = 0
\]
for all ${\bf j}, {\bf j}' \subset {\bf i}$ such that $\min(|{\bf j}|, |{\bf j}'|) < d$.
\end{assumption}

\begin{table}
\centering
\caption{Degree $k$, order of degeneracy $d$, covariance $\zeta_d^h$ and fourth
  moment $\eta^h$ for the kernel functions in
  \exref{ken_ex}--\exssref{t^*_ex} when independence holds. }
\begin{tabular}{ c| c |c| c| c }
\hline
Kernel &$h_\tau$& $h_{\hat{\rho}_s}$&  $h_D$ & $h_{t^*}$\\
\hline
\hline
$k$ & 2 & 3&5 &4 \\
\hline
$d$ & 1 & 1&2 &2 \\
\hline
 $\zeta_d^h$& $1/9$ &$1/9$& $1/810000$ & $1/225$ \\
\hline
 $\eta^h$& -- & --& $(7/864000)^2$ & $(2/525)^2$ \\
\hline
\end{tabular}
\label{tab:degen_tab}
\end{table}

\section{Test statistics} \label{sec:main} 

We now proceed to construct test statistics for the independence hypothesis $H_0$ from~(\ref{eq:H0}).  Building on the
pairwise rank correlations from \secref{rank-cor}, we introduce general classes of statistics and derive their
respective asymptotic null distributions when
$m, n \longrightarrow \infty$.

\subsection{Sum of squared sample rank
  correlations} \label{sec:Schott_Stat} 

Let $U^{(pq)}_h$ be a rank-based U-statistic as defined in
\eqref{U_stat}, with mean zero when $X^{(p)}$ and $X^{(q)}$ are
independent.  Suppose further that large absolute values of
$U^{(pq)}_h$ indicate strong association (positive or negative) between
${\bf X}^{(p)}$ and ${\bf X}^{(q)}$.  
It is then natural to reject $H_0$ for large values
of the centered quantity
\begin{equation} \label{centered_sum_of_sq_U}
S_h := \sum_{1 \leq p < q \leq m} \left(U^{(pq)}_h\right)^2 - {m \choose 2} \mu_h.
\end{equation}
Here, $\mu_h := \mathbb{E}_0[(U^{(pq)}_h)^2]$.  Note that, as
indicated by our notation, this expectation does not depend on the
choice of $p$ and $q$ by \lemref{general_properties}$(i)$. The
following lemma specifies $\mu_h$ and gives a result on other moments
of $U_h^{(pq)}$ that will be used later.

\begin{lemma}
  \label{lem:moment_lem} 
  Let $n \ge 2k \geq 2$, and suppose that $U^{(pq)}_h$
  from~\eqref{U_stat} has a kernel $h$ with order of degeneracy $d$
  under $H_0$.  Then the following three facts hold under $H_0$:
\begin{enumerate}
\item   
\[
\mu_h  
    = 
      {n \choose k}^{-1} \sum_{c = 1}^k {k \choose c} {n-k
      \choose k - c } \zeta_c = {k \choose d}^2\frac{ d!\zeta_{d}}{n^{d}} + O\left(n^{-d - 1}\right).
\]   
 \item For any $r>2$,
    \[
    \mathbb{E}_0\left[\big(U^{(pq)}_h\big)^r\right] =
    O\left(n^{-\lfloor(rd +1 )/2\rfloor} \right),
    \]
where $\lfloor\cdot \rfloor$ denotes the floor function.

\item 
\[
 \mathbb{E}_0\left[\big(U^{(pq)}_h\big)^4 \right]  = \begin{cases}
    \displaystyle\frac{3k^{4}(\zeta^h_1)^2}{n^2} + O\left(n^{-3}\right) &\text{if}
    \quad d = 1,\\[0.5cm] 
    \displaystyle
      {k \choose 2}^4\frac{12}{n^{4}}\left((\zeta_2^h)^2+ 
 4\eta^h\right)+
 O\left(n^{-5}\right) &\text{if} \quad d = 2.
  \end{cases}
\]
\end{enumerate}
\end{lemma}

For \lemref{moment_lem}$(i)$ and $(ii)$, see~Lemma 5.2.1A and 5.2.2B
in \cite{Serfling1980}. The last claim about the leading term of the
fourth moment is proven in Appendix~\ref{sec:PfSec3}.  Let
$\mu_{\tau}$, $\mu_{\hat{\rho}_s}$, $\mu_{D}$ and $\mu_{t^*}$ be the
values of $\mu_h$ when $h$ is equal to $h_{\tau}$, $h_{\hat{\rho}_s}$,
$h_{D}$ and $h_{t^*}$ respectively. Then
\begin{align*}
\mu_{\tau} &= \frac{2(2n+5)}{9n (n-1)},
 &\mu_{\hat{\rho}_s} &= \frac{ (n^2 - 3)}{n(n-1)(n-2)},\\
\mu_{D} &= \frac{2(n^2 + 5n - 32)}{9n (n-1)(n-3)(n-4)},
&\mu_{t^*} &= \frac{8}{75} \frac{3n^2 + 5n - 18}{n(n-1)(n-2)(n-3)}.
\end{align*}
The first three quantities can be found in \citet{hoeffdingU,
  hoeffdingIndep}.  The stated value of $\mu_{t^*}$ is based on our
own calculations.

\subsection{Unbiased estimator of the sum of squared population correlations}

The kernel function $h$ is central to the role of $U^{(pq)}_h$ as a
measure of association between the vectors of observations
${\bf X}^{(p)}$ and ${\bf X}^{(q)}$.   At the population level, the
association (positive or negative) is captured by the expectation of  $U^{(pq)}_h$, which is
also equal to 
\begin{equation} \label{pop_cor}
\theta_h^{(pq)} := \mathbb{E}\left[h\left({\bf R}^{(pq)}_{{\bf j}}\right)\right],
\end{equation} 
where ${\bf j}$ may be any element in $\mathcal{P}(n, k)$.  Hence,
\begin{equation} \label{kernel_signal}
\sum_{1\leq p < q \leq m} (\theta_h^{(pq)})^2
\end{equation}
is a population measure of overall dependency in the joint
distribution of $X^{(1)}, \dots, X^{(m)}$.  
As an alternative
approach to Section~\ref{sec:Schott_Stat}, we now construct an unbiased estimator of
\eqref{kernel_signal}, targeting more directly the problem of global
(in-)dependence.

Recall that given ${\bf i} \in \mathcal{P}(n,2k)$ and
${\bf j}\in \mathcal{P}(n, k)$ such that ${\bf j} \subset {\bf i}$ as
sets, ${\bf i}\setminus {\bf j}$ is the $k$-tuple in
$\mathcal{P}(n, k)$ that is given by their set difference.  The
function 
\begin{equation} \label{hw_kern}
h^W\left({\bf R}^{(pq)}_{{\bf i}}\right) \;:=\;
{2k \choose k}^{-1}\sum_{\substack{{\bf j} \subset {\bf i} \\  {\bf j} \in \mathcal{P}(n, k)}} h\left({\bf R}^{(pq)}_{{\bf j}}\right) h\left({\bf R}^{(pq)}_{{\bf i}\setminus {\bf j}}\right)
\end{equation}
defined on the domain $(\mathbb{R}^2)^{2k}$ is symmetric in its $2k$ arguments
${\bf R}^{(pq)}_{{\bf i}, 1}, \dots, {\bf R}^{(pq)}_{{\bf i}, 2k}$, due to the
symmetry of $h$ and the summation over all possible tuples
${\bf j} \in \mathcal{P}(n, k)$ contained in ${\bf i}$ on the right hand side of
\eqref{hw_kern}.  Moreover, $h^W$ is an unbiased estimator of the
square of the expectation in~\eqref{pop_cor}, since each summand on
the right hand side of \eqref{hw_kern} is a product of two independent
unbiased estimators of
$\theta_h^{(pq)}$.
Therefore, defining the U-statistic
\begin{equation} \label{W_stat_unscaled}
W_h^{(pq)} = W_h^{(pq)}\left({\bf R}_1^{(pq)}, \dots, {\bf R}_n^{(pq)}\right) =  {n \choose 2k}^{-1} \sum_{{\bf i}\in \mathcal{P}(n, 2k)} h^W\left({\bf R}^{(pq)}_{{\bf i}}\right),
\end{equation}
we have that the sum 
\begin{equation} \label{sum_of_w_kern_stat}
T_h := \sum_{1 \leq p < q \leq m} W_h^{(pq)}
\end{equation}
is an unbiased estimator of \eqref{kernel_signal}.  The statistic
$T_h$ is a U-statistic itself and serves as a natural test statistic
for $H_0$.  Large values of $T_h$ indicate departures from $H_0$.
When $h = h_\tau$, i.e., the case of Kendall's tau,
$T_h$ equals the statistic displayed in \eqref{kendall_W_kern_stat} in
the introduction.

Clearly, $W_h^{(pq)}$ is a rank-based U-statistic with the kernel
$h^W$ of degree $2k$. 
The following lemma summarizes the degeneracy properties of $h^W$
under $H_0$.

\begin{lemma}\label{lem:degen_relate}
Suppose $h : (\mathbb{R}^2)^k \longrightarrow \mathbb{R}$ is a
symmetric kernel function of degree $k$ with order of degeneracy
$d\in\{1,2\}$  under $H_0$.  So, $\zeta_d^h > 0$. Then, under $H_0$, the induced symmetric kernel function $h^W$ defined in \eqref{hw_kern} has order of degeneracy $2d$ and 
\begin{align*}
\zeta^{h^W}_{2d} &= \mathbb{E}_0\left[h^W\left(\mathbf{ R}^{(pq)}_{{\bf i}}
                   \right)h^W \left(
                   \mathbf{R}^{(pq)}_{{\bf j}}\right)\right]\\
&\phantom{:}= 
  \begin{cases}
    \displaystyle
   4 {2k - 2 \choose k -1}^2 {2k \choose k}^{-2}
   \big(\zeta_d^h\big)^2&\text{if} \quad d = 1,\\[0.2cm]
   \displaystyle
    12 {2k - 4 \choose k -2}^2 {2k \choose k }^{-2}\left\{\big(\zeta_d^h\big)^2 + 2 \eta^h\right\} &\text{if} \quad d = 2,
  \end{cases}
\end{align*}
where ${\bf i}, {\bf j} \in \mathcal{P}(n, 2k)$ and $|{\bf i}\cap {\bf j}| = 2d$.
\end{lemma}

The proof of the lemma is deferred to Appendix~\ref{sec:PfSec3}.

\subsection{Sum of sample rank correlations}
For testing $H_0$ it is also interesting to consider the simple sum
\begin{equation} \label{simple_sum_of_U}
Z_h := \sum_{1 \leq p < q \leq m} U^{(pq)}_h,
\end{equation}
which unbiasedly estimates the signal
\begin{equation} \label{kernel_signal_one_side}
\sum_{1\leq p < q \leq m}\theta_h^{(pq)};
\end{equation}
compare with \eqref{kernel_signal}. When the kernel $h$ is
$h_{\hat{\rho}_s}$ or $h_{\tau}$, without the squaring as in
\eqref{kernel_signal}, \eqref{kernel_signal_one_side} may not be an
effective measure for the overall dependency of $X^{(1)}, \dots,
X^{(m)}$ since any pairwise signal $\theta_h^{(pq)}$ can be either
negative or positive depending on the direction of association
\citep{kruskal1958}.
Hence, the rejection of $H_0$ for large value of $Z_h$ is only good
for testing against the ``one-sided" alternative
\[
\sum_{1 \leq p< q \leq m} \theta_h^{(pq)} > 0 , \quad \theta^{(pq)}_h \geq 0 \text{ for all } p < q.
\]

However, when $h=h_{t^*}$ or $h=h_{D}$, \eqref{kernel_signal_one_side}
is an effective measure of the overall dependency of
$X^{(1)}, \dots, X^{(m)}$, since any pairwise signal $\theta_h^{(pq)}$
is non-negative and equals zero if and only if $X^{(p)}$ and $X^{(q)}$
are independent under the weak assumptions in the work of
\citet{hoeffdingIndep} and \citet{bergsma:2014}. In this case, large
values of $Z_h$ detect dependency among $X^{(1)}, \dots, X^{(m)}$,
without any restrictions to the direction of the pairwise associations.

%

\section{Asymptotic null distributions}
\label{sec:asympt-null-distr}

We are now ready to state our results on the asymptotic distributions
for the test statistics introduced in
Section~\ref{sec:main}.  As mentioned in \secref{rank-cor}, we
focus on rank-based U-statistics with a kernel $h$ satisfying
\assumpref{sp_assume} and order of degeneracy $d\in\{1,2\}$ under
$H_0$.


\begin{theorem}\label{thm:main}
  Suppose the null hypothesis $H_0$ from~(\ref{eq:H0}) is
  true.  Let $h$ be a symmetric bounded kernel function of degree $k$
  satisfying \assumpref{sp_assume}, and consider the asymptotic regime $m, n \longrightarrow  \infty$.  If $d= 1$, after suitable rescaling, 
$S_h$, $T_h$ and $Z_h$ are asymptotically normal, namely,
\[
\frac{nS_h}{k^2 m\zeta^h_1 },   \;\frac{nT_h}{k^2m\zeta^h_1 }, \; \frac{\sqrt{2n}Z_h}{k m \sqrt{\zeta_1^h}}\;\Longrightarrow\; \mathcal{N}(0, 1). 
\]
If $d= 2$, then
\[
\frac{n^2 {k \choose 2}^{-2}S_h}{2 m\sqrt{(\zeta^h_2)^2 + 6 \eta^h }}   , \;\frac{ n^2 {k \choose 2}^{-2} T_h}{2  m\sqrt{(\zeta^h_2)^2 + 2 \eta^h }}, \; \frac{n {k \choose 2}^{-1} Z_h}{m \sqrt{\zeta_2^h}} \;\Longrightarrow\;
 \mathcal{N}(0, 1).
\]
\end{theorem} 

The theorem covers in particular the rank correlations from Examples
\ref{ex:ken_ex}--\exssref{t^*_ex}.  
A critical value for an approximate $\alpha$-size test can thus be
calibrated based on normal quantiles.  As in the classical theory for
U-statistics, the rescaling factors for the non-degenerate and
degenerate cases differ in order; for instance, we have to multiply
$S_h$ with a factor of order $O(n/m)$ when $h$ has order of degeneracy
$d = 1$, and with a factor of order $O(n^2/m)$ when $h$ has order of
degeneracy $d = 2$. The ingredients needed to compute the rescaling
factors were given in \tabref{degen_tab}.  In slight abbreviation, we
write $S_\tau$, $S_{\hat{\rho}_s}$, $S_D$ and $S_{t^*}$ for the four
versions of the statistic $S_h$ from \eqref{centered_sum_of_sq_U} with
the different kernels reviewed in \secref{rank-cor}, and analogously,
$T_\tau$, $T_{\hat{\rho}_s}$, $T_D$, $T_{t^*}$ and $Z_\tau$,
$Z_{\hat{\rho}_s}$, $Z_D$, $Z_{t^*}$ for the versions of $T_h$ and
$Z_h$ from \eqref{sum_of_w_kern_stat} and \eqref{simple_sum_of_U}.
This matches the notation used in ~(\ref{eq:kendall-sample-sum}),
\eqref{kendall_W_kern_stat} and \eqref{sumoftaus}.

We remark that while the classical Spearman's rho is not a U-statistic
one may of course consider the centered test statistic 
\begin{equation} \label{rho_test_stat}
S_{\rho_s} := \sum_{1 \leq p < q \leq m} \left(\rho^{(pq)}_s\right)^2 - {m \choose 2} \mu_{\rho_s},
\end{equation}
where $\mu_{\rho_s} := \mathbb{E}_0[(\rho^{(pq)}_s)^2] = 1/(n -1)$; see
\citet[p.321]{hoeffdingU}.  The convergence of $\frac{n}{m}S_{\hat{\rho}_s}$ to a
standard normal distribution, as suggested by \thmref{main} and \tabref{degen_tab}, 
implies the
following distributional convergence for $S_{\rho_s} $. Its proof, given in Appendix~\ref{sec:PfSec4},
makes use of the decomposition from~(\ref{rho_decomp}). The same
result has been obtained by \citet{zhou2007} and \citet{SpearmanTest} 
via different methods.

\begin{corollary} \label{spearnull}
Under $H_0$, $\frac{n}{m}S_{\rho_s} \Longrightarrow N(0, 1)$ as $m, n \longrightarrow \infty$. 
\end{corollary}

%
%
%

Our proof of Theorem~\ref{thm:main} is based on a central limit
theorem for martingale arrays \citep[Corollary 3.1]{HeydeAndHall} that
was also applied by \cite{Schott2005}.  We outline the approach here,
postponing computations verifying the conditions of the martingale CLT
to Appendix~\ref{sec:PfSec4}.

\begin{proof}[Proof of \thmref{main}]
  Fix a sample size $n$.  For $q=1,\dots,m$, let $\mathcal{F}_{nq}$ be
  the $\sigma$-algebra generated by
  $\mathbf{X}^{(1)}, \dots, \mathbf{X}^{(q)}$ (or for our purposes,
  equivalently, $\mathbf{R}^{(1)},\dots,\mathbf{R}^{(q)}$) under
  $H_0$.  For convenience we will use the shorthand
  $\bar{U}^{(pq)}_h := \left(U_h^{(pq)}\right)^2 - \mu_h$ for
  $1 \leq p < q \leq m$. Let
  \begin{equation}
    \label{eq:mtg-diffs}
    D^S_{nq} :=   \sum_{p
      = 1}^{q-1}\bar{U}^{(pq)}_h, 
    \quad
    D^T_{nq} := \sum_{p
      = 1}^{q-1} W_h^{(pq)},\quad \text{and}
    \quad
 D^Z_{nq} := \sum_{p
      = 1}^{q-1} U_h^{(pq)} 
  \end{equation}
  and set $D^S_{n1}= D^T_{n1} = D^Z_{n1} = 0$.  Writing
  $S_{nq}=\sum_{l=1}^q D^S_{nl}$, $T_{nq}=\sum_{l=1}^q D^T_{nl}$ and  $Z_{nq}=\sum_{l=1}^q D^Z_{nl}$, 
  we have that $S_h=S_{nm}$, $T_h=T_{nm}$ and $Z_h=Z_{nm}$.

  We claim that, for each $n$, the  sequences 
\begin{equation} \label{3mtg}
\left\{S_{nq},
    \mathcal{F}_{nq}, 1 \leq q \leq m\right\}, \quad \left\{T_{nq},
    \mathcal{F}_{nq}, 1 \leq q \leq m\right\} \quad \text{and} \quad \left\{Z_{nq},
    \mathcal{F}_{nq}, 1 \leq q \leq m\right\}
\end{equation}
are martingales,
  i.e., $\mathbb{E}_0\left[S_{nq}\middle|\mathcal{F}_{n , q-1}\right] =
  S_{n, q-1}$, $\mathbb{E}_0\left[T_{nq}\middle|\mathcal{F}_{n , q-1}\right] =
  T_{n, q-1}$ and $\mathbb{E}_0\left[Z_{nq}\middle|\mathcal{F}_{n , q-1}\right] =
  Z_{n, q-1}$ for $q=2,\dots,m$. Given the way $S_{nq}$, $T_{nq}$ and $Z_{nq}$ are defined
  as sums, it suffices to show that
  \begin{equation}
    \label{eq:martingale-prop}
    \mathbb{E}_0\left[\bar{U}^{(pq)}_h\middle| \mathcal{F}_{n,
      q-1} \right] = 
  \mathbb{E}_0\left[W_h^{(pq)}\middle| \mathcal{F}_{n,
      q-1} \right] = 
 \mathbb{E}_0\left[U_h^{(pq)}\middle| \mathcal{F}_{n,
      q-1} \right] = 0
  \end{equation}
  for all $1\le p< q\le m$.  Since
  $\mathbf{X}^{(1)}, \dots, \mathbf{X}^{(m)}$ are independent under
  $H_0$, conditioning on $\mathcal{F}_{n,q-1}$ is the same as
  conditioning on ${\bf X}^{(p)}$ alone in~(\ref{eq:martingale-prop}). As
  $\bar{U}^{(pq)}_h$, $W_h^{(pq)}$ and $U_h^{(pq)}$are all symmetric functions of
  the $n$ arguments ${\bf R}^{(pq)}_1, \dots, {\bf R}^{(pq)}_n$,
  (\ref{eq:martingale-prop}) follows from
  \lemref{general_properties}(i) and (ii).

  By the boundedness of our kernel $h$, each of the martingales in \eqref{3mtg} is trivially square-integrable. As such, the central limit theorem for martingale arrays from Corollary 3.1 in \cite{HeydeAndHall}
implies the assertion of \thmref{main} if we can show that the
  squares of the martingale differences $D^S_{nl}$, $D^T_{nl}$ and $D^Z_{nl}$ each
  satisfy the following two conditions.  The first condition requires that as $m, n \longrightarrow \infty$,
  \begin{equation}
     \label{cond11}
   \frac{n^2}{  m^2} \sum_{l = 2}^m \mathbb{E}_0\left[(D^S_{nl})^2\middle| \mathcal{F}_{n, l
        -1}\right]  ,  \quad  \frac{n^2}{  m^2} \sum_{l = 2}^m \mathbb{E}_0\left[(D^T_{nl})^2\middle| \mathcal{F}_{n, l
        -1}\right]  \quad \underset{p}{\longrightarrow } \quad   k^4 (\zeta_1^h)^2   ,
  \end{equation} 
  \begin{equation} \label{cond111}
   \frac{n}{  m^2} \sum_{l = 2}^m \mathbb{E}_0\left[(D^Z_{nl})^2\middle| \mathcal{F}_{n, l
        -1}\right]  \quad  \underset{p}{\longrightarrow } \quad   2^{-1}k^2 \zeta_1^h   ,
  \end{equation} 
 for $d =1 $, and  
\begin{align}
     \label{cond12}
    \frac{n^4}{m^2}\sum_{l = 2}^m \mathbb{E}_0\left[(D^S_{nl})^2\middle| \mathcal{F}_{n, l
        -1}\right]   \quad &\underset{p}{\longrightarrow } \quad  4 {k \choose 2}^4 \left\{ (\zeta^h_2)^2 + 6\eta^h  \right\}\;,\\
 \label{cond13}
\quad \frac{n^4}{m^2} \sum_{l = 2}^m\mathbb{E}_0\left[(D^T_{nl})^2\middle| \mathcal{F}_{n, l
        -1}\right]  \quad &\underset{p}{\longrightarrow }  \quad   4 {k \choose 2}^4 \left\{ (\zeta^h_2)^2 + 2\eta^h  \right\},\\
 \label{cond14}
\quad \frac{n^2}{m^2}\sum_{l = 2}^m \mathbb{E}_0\left[(D^Z_{nl})^2\middle| \mathcal{F}_{n, l
        -1}\right]  \quad &\underset{p}{\longrightarrow }  \quad    {k \choose 2}^2 \zeta^h_2,
  \end{align} 
for $d  =2$,
where the convergence symbol stands for convergence in probability.
  The second condition is a Lindeberg condition.  
  In Lemma~\ref{lem:ell2-convergence} in the Appendix~\ref{sec:PfSec4},
  we show that, in fact, \eqref{cond11}-\eqref{cond14}  also hold in the
  stronger sense of $L^2$ (or quadratic mean).
  Lemma~\ref{lem:lyapunov} proves a Lyapunov condition that implies
  the Lindeberg condition, which completes the proof of \thmref{main}.  
\end{proof}

\section{Aspects of power} \label{sec:power}

In order to investigate the power of our tests we adopt an
asymptotic minimax perspective. While our null distributional results
in \secref{main} are valid under the more general asymptotic regime
$m, n \longrightarrow \infty$, we treat here the
particular regime $\frac{m}{n} \longrightarrow \gamma \in (0, \infty)$.
Recall the definition in \eqref{pop_cor}, 
and let $\Theta=( \theta_h^{(pq)})_{ 1 \leq p < q \leq m}$ be the
$m \choose 2$-vector comprising all these pairwise measures of
association.  In our exploration of power, it is at times convenient
to have U-statistics with a kernel $h$ of degree 2.  For
instance, we apply results for U-statistics of degree 2 from
\cite{chen2016sup}.  Consequently, our power analysis focuses on the
two classes of statistics $S_h$ and $T_h$ for the kernel $h = h_\tau$
of Kendall's tau.
To indicate
this restriction, we write
$\theta^{(pq)}_\tau := \mathbb{E}[ h_\tau({\bf R}^{(pq)}_{\bf i}
)]$ for ${\bf i} \in \mathcal{P}(n, 2)$ and
$\Theta_\tau = (\theta^{(pq)}_\tau)_{1 \leq p <q \leq m} $.  

Let $\mathcal{D}_m$ be a family of continuous joint distributions on
$\mathbb{R}^m$ containing all $m$-variate Gaussian distributions, to
be considered as joint distributions for $(X^{(1)}, \dots,
X^{(m)})$. For a given significance level $\alpha \in (0, 1)$, we
study which sequences of lower bounds $\epsilon_n$ on the dependency
signal $\|\Theta_\tau\|_2$ allow tests to uniformly achieve a
fixed power $\beta > \alpha$ over the set of alternative
distributions
\begin{equation} \label{sq_alt}
\mathcal{D}_m(\| \Theta_\tau\|_2 \geq \epsilon_n) := \bigg\{D \in \mathcal{D}_m : \| \Theta_\tau\|_2 \geq \epsilon_n   \bigg\}.
\end{equation}
As usual, we take a test
$\phi$ to be a function mapping the data into the unit interval
$[0,1]$.
Given a test statistic $S = S({\bf X}_1, \dots, {\bf X}_n)$, we write
$\phi_\alpha (S)$ for the test that rejects for large values of $S$
and has (asymptotic) size $\alpha$.

The statistics $S_\tau$ and $T_\tau$ estimate the squared Euclidean
norm of the signal $\|\Theta_\tau\|_2^2$.  They are thus natural when the
interest is in detecting the alternatives in \eqref{sq_alt}.  The
following theorem gives a rough lower bound on the signal size
$\|\Theta_\tau\|_2$ that is needed for detectability.

\begin{theorem} \label{thm:beatMax}
Let $0 < \alpha < \beta < 1$. Under the asymptotic regime $m/n \longrightarrow \gamma \in (0, \infty)$, there exist constants $C_i = C_i(\alpha, \beta, \gamma) > 0$ for $i = 1, 2$,  such that
\begin{enumerate}
\item 
\[
\liminf_{n \longrightarrow \infty} 
\inf_{\mathcal{D}_m(\| \Theta_\tau\|_2 \geq \epsilon_n) }
\mathbb{E}[\phi_\alpha (S_\tau)] > \beta \quad \text{for} \quad \epsilon_n  = C_1 \sqrt{n}, \quad \text{and}
\]
\item
\[
\liminf_{n \longrightarrow \infty}
 \inf_{\mathcal{D}_m(\| \Theta_\tau\|_2 \geq \epsilon_n) }\mathbb{E}[\phi_\alpha (T_\tau)] > \beta \quad \text{for} \quad \epsilon_n  = C_2 \sqrt{n} .
\]
\end{enumerate} 
\end{theorem}

Our proof of \thmref{beatMax} uses rather general concentration
bounds and it should be possible to sharpen the analysis to show
asymptotic power for $\phi_\alpha{ (S_\tau)}$ and
$\phi_\alpha{ (T_\tau)}$ under smaller signal strength.
Indeed, 
we {conjecture} that a
test based on $T_\tau$ can asymptotically attain uniform power $\beta$
when the signal size $\|\Theta_\tau\|_2$ is of constant order
$O(1)$. This conjecture is partially supported by
\thmref{conjectureThm} below.

\subsection{Rate-optimality under equicorrelation} \label{equi}

When the joint distribution of $X^{(1)}, \dots, X^{(m)}$ is a regular
Gaussian distribution, then $H_0$ is equivalent to $R - I_m = 0$,
where $I_m$ is the $m$-by-$m$ identity
matrix; recall that $R$ is the population Pearson correlation matrix.  For any $\epsilon>0$, define the alternative
\begin{equation} \label{GaussAlt}
\mathcal{N}_m(\| R - I_m\|_F \geq \epsilon)
\end{equation}
as the family of regular $m$-variate Gaussian distributions whose
correlation matrix $R$ satisfies $\| R - I_m\|_F \geq \epsilon$.  Fix
any $\alpha,\beta\in(0,1)$ with $\alpha < \beta$.  A result of
\citet*[Remark 1(a)]{CaiAndMa} implies that in the regime
$m/n \longrightarrow \gamma$, 
there exists a sufficiently small constant
$c = c(\alpha, \beta, \gamma) > 0$ such that
\[
\limsup_{n \rightarrow \infty} \inf_{\mathcal{N}_m(\| R - I_m\|_F \geq c) } \mathbb{E} [\phi] < \beta
\]
for any $\alpha$-level test $\phi$. In other words,
asymptotically, no $\alpha$-level test can uniformly achieve the
desired power against the alternative \eqref{GaussAlt} when the signal
size $\|R - I_m\|_F$ is allowed to be as small as $c$.  It follows
immediately that in our nonparametric setup there also
exists a constant $\tilde{c}=\tilde{c}(\alpha, \beta, \gamma) > 0$
such that
\[
\limsup_{n \rightarrow \infty} \inf_{\mathcal{D}_m(\|\Theta_\tau\|_2 > \tilde{c})} \mathbb{E}[\phi] < \beta
\]
for any $\alpha$-level test $\phi$.  This is true because the
nonparametric class $\mathcal{D}_m$ contains all $m$-variate Gaussian
distributions, and because $\theta_\tau^{(pq)} \asymp \rho^{(pq)}$
when $X^{(p)}$ and $X^{(q)}$ are jointly Gaussian.  The latter fact
follows from
$\rho^{(pq)} = \sin\big(\frac{\pi}{2}\theta_\tau^{(pq)}\big)$ for non-degenerate elliptical distributions; see \cite{ellipKendall}.

Given the observation just made, an $\alpha$-level test $\phi$
that satisfies
\begin{equation} \label{rate_opt}
\liminf_{n \rightarrow \infty}  \inf_{\mathcal{D}_m(\|\Theta_\tau\|_2 \geq \tilde{C})}\mathbb{E}[\phi] > \beta
\end{equation}
for a large enough constant
$\tilde{C}=\tilde{C}(\alpha, \beta, \gamma) > 0$ would be
rate-optimal.  
If the signal $\|\Theta_\tau\|_2$ is large,  being an unbiased estimator of  $\|\Theta_\tau\|_2^2$ our statistic $T_\tau$ always centers around the same large value regardless of the true underlying distribution of ${\bf X}$. It is hence
natural to \emph{conjecture} that the optimality condition
\eqref{rate_opt} is satisfied by the test $\phi_\alpha(T_\tau)$, for a
reasonable class of elliptical distributions $\mathcal{D}_m$ that extends
beyond the Gaussians. 
 Our
next result supports the 
conjecture.

Let $\mathcal{N}_m^{\text{equi}}(\|\Theta_\tau\|_2 \geq \tilde{C})$ be
the set of $m$-variate Gaussian distributions that have all pairwise
(Pearson and thus also Kendall) correlations equal to a common value
such that $\|\Theta_\tau\|_2 \geq \tilde{C}$.  If
$\theta^{(pq)}_\tau = \theta$ for all $1 \leq p \not = q \leq m$, then
$\|\Theta_\tau\|_2^2 =\theta^2{m \choose 2}$.

\begin{theorem} \label{thm:conjectureThm}
As $\frac{m}{n} \longrightarrow \gamma$, there exists a constant $\tilde{C} = \tilde{C}(\alpha, \beta, \gamma)>0$ such that 
\[
\liminf_{n \longrightarrow \infty}
\inf_{\mathcal{N}_m^{\text{equi}}(\|\Theta_\tau\|_2 \geq
  \tilde{C})}\mathbb{E}[\phi_\alpha(T_\tau)] > \beta.
\]
\end{theorem}

The theorem is proven in Appendix~\ref{sec:proofs-sec5}.  
Our simulation experiments on power in Section~\ref{sec:simulations}
corroborate the conjecture made above.

\subsection{Comparison with the ``max'' statistic}

The work of \cite{HanLiu2014} considered testing the
independence hypothesis $H_0$ from~(\ref{eq:H0}) using maxima of
rank correlations and, in particular, the statistic 
\begin{equation} \label{max_U}
S^{\max}_{\tau} := \max_{ 1\leq p < q \leq m} | \tau^{(pq)}|
\end{equation}
that is based on Kendall's tau.  
\cite{HanLiu2014} derived the
asymptotic null distribution under the regime $\log m = o(n^{1/3})$.
Let $\phi_\alpha(S^{\max}_\tau)$ be the level $\alpha$ test that
rejects for large values of $S^{\max}_\tau$. Naturally, this test is
powerful against alternatives belonging to the set
\begin{equation} \label{max_alter}
\mathcal{D}_m(\| \Theta_\tau \|_\infty \geq \epsilon_n) := \left\{ D \in \mathcal{D}_m : \|\Theta_\tau\|_\infty \geq \epsilon_n\right\},
\end{equation}
which is characterized by the max norm of $\Theta_\tau$.  Indeed, when
$\log m = o(n^{1/3})$, for a given significance level $\alpha$ and
targeted power $\beta \in (\alpha, 1)$, it was shown that there exists
a constant $c_1 = c_1(\alpha, \beta)$ such that
\[
\liminf_{n \longrightarrow \infty} \inf_{\mathcal{D}_m(\| \Theta_\tau \|_\infty \geq  c_1 \sqrt{(\log m) /n})  }  \mathbb{E} [\phi_\alpha(S^{\max}_\tau) ] >\beta.
\]
\cite{HanLiu2014} also showed
rate-optimality of this test, i.e., there exists a constant
$c_2 = c_2(\alpha, \beta) < c_1$ such that for \emph{any}
$\alpha$-level test $\phi$,
\begin{equation} \label{lbddmax}
 \limsup_{n \longrightarrow \infty}\inf_{
\mathcal{D}_m(\| \Theta_\tau \|_\infty \geq c_2 \sqrt{(\log m) /n})
}\mathbb{E}[\phi] < \beta. 
\end{equation}
Note that in the regime $m/n \longrightarrow \gamma$ that we consider in
this section we have $\log m = o(n^{1/3})$.

While a test based on $S_\tau^{\max}$ is rate-optimal in detecting
alternatives of the form \eqref{max_alter} characterized by the max
norm signal, it is---as intuition suggests---not powerful in detecting
alternatives with small but non-zero dependence among many pairs of
random variables.  The latter scenario is best described via the
Euclidean norm as in \eqref{sq_alt}.  This is demonstrated by the
following theorem about equicorrelation alternatives; recall the
positive result in Theorem~\ref{thm:conjectureThm}.

\begin{theorem} \label{thm:HLnotWork} As
  $\frac{m}{n} \longrightarrow \gamma$, there does not exist any
  constant $C = C(\alpha, \beta, \gamma)>0$ such that
\[
\liminf_{n \longrightarrow \infty}
\inf_{\mathcal{N}_m^{\text{equi}}(\|\Theta_\tau \|_2 \geq
  C)}\mathbb{E}[\phi_\alpha(S_\tau^{\max})] > \beta.
\]
\end{theorem}

The proof of the theorem is deferred to
Appendix~\ref{sec:proofs-sec5}.  It relies on a comparison lemma of
\citet{MR3161448} and a recent result on Gaussian approximation for
high-dimensional U-statistics in
\cite{chen2016sup}. \thmref{HLnotWork} says that a signal size
$\|\Theta_\tau\|_2$ of constant order is not enough to guarantee a
preset asymptotic power for a test based on $S^{\max}_\tau$
under the regime $\frac{m}{n} \longrightarrow \gamma$.  We demonstrate
this in our simulations in the next section.

\section{Implementation and simulation experiments}
\label{sec:simulations}

We now compare several tests of the independence hypothesis $H_0$
based on specific versions of the statistics introduced in this paper.
Our simulations first explore the size of the tests when critical
values are set using asymptotic normal approximations.  We then compare their power.  Before turning to the
simulations, however, we discuss the computation of the 
test statistics.

\subsection{Implementation}

In order to compute the statistics $S_h$ from
\eqref{centered_sum_of_sq_U} and $Z_h$ from \eqref{simple_sum_of_U}
for $m$ variables, one has to make $m \choose 2$ evaluations of the
U-statistics $U^{(pq)}_h$.  In general, for a U-statistic of degree
$k$, a na\"{i}ve calculation following the definition
in~\eqref{U_stat} requires $O(n^k)$ operations.  Fortunately, more
efficient algorithms are available for the specific examples covered here.  For instance, 
Spearman's $\rho_s^{(pq)}$ from
Example~\ref{ex:rho_ex} can be computed in $O(n \log n)$ operations.
The same is true for Kendall's $\tau^{(pq)}$ from Example~\ref{ex:ken_ex} \citep{nlogn}.
Similarly, \citet{n2logn} showed how to compute the Bergsma-Dassios sign covariance
${t^*}^{(pq)}$ in $O(n^2 \log n)$ operations despite the fact that its
kernel has degree $k=4$, as reviewed in Example~\ref{ex:t^*_ex}.  An
improvement to $O(n^2)$ was given by \cite{Heller2016}.
Finally, \cite{hoeffdingIndep} gives formulas for efficient
computation of his statistic
$D$ in Section 5 of his paper.

The situation with the class of statistics $T_h$ from
\eqref{sum_of_w_kern_stat} is more complicated. Since a kernel
$h$ of degree $k$ gives rise to an induced kernel $h^W$ of degree
$2k$, the number of operations equals $O(n^{2k})$ if we compute
$W_h^{(pq)}$ by na\"{i}vely following its definition.  This would lead
to a total of ${m \choose 2} O(n^{2k})$ operations to find all
$W_h^{(pq)}$, $1 \leq p < q \leq m$.  A more efficient way to compute
each $W_h^{(pq)}$ in $O(n^k)$ time proceeds as follows.  Using
\eqref{hw_kern} and \eqref{W_stat_unscaled}, we see that
\begin{equation} \label{W_altern}
W^{(pq)}_h=\frac{1}{ {n \choose k}{ n - k \choose k}}\sum_{{\bf i } \in \mathcal{P}(n, k)} h_{\bf i } \bar h_{\bf i},
\end{equation}
where for each ${\bf i} \in \mathcal{P}(n, k)$, and suppressing the
dependence on the pair $(p,q)$, we define
\[
h_{\bf i} := h\left({\bf R}^{(pq)}_{\bf i}\right)  \quad \text{and} \quad
\bar h_{\bf i} := \sum_{{\bf j } \in \mathcal{P}(n, k): {\bf j}\cap {\bf i} = \emptyset} h_{\bf j}.
\]
Hence, it suffices to calculate \begin{inparaenum} \item $h_{\bf i}$
  for all ${\bf i} \in \mathcal{P}(n, k)$, \item $\bar h_{\bf i}$ for
  all ${\bf i} \in \mathcal{P}(n, k)$ and \item the summation in
  \eqref{W_altern}\end{inparaenum}, in that order.   Evidently, step (i)
involves $O(n^k)$ operations. By the inclusion-exclusion principle,
\begin{equation} \label{hibar}
\bar h_{\bf i} = \sum_{{\bf j } \in \mathcal{P}(n, k)} h_{\bf j} + \sum_{ 1 \leq \ell \leq k} (-1)^{\ell} \sum_{\substack{{\bf j'} \in \mathcal{P}(n, \ell) : \\ {\bf j'} \subset {\bf i}}} h_{\bf j'},
\end{equation}
where
$ h_{\bf j'} := \sum_{{\bf j} \subset \mathcal{P}(n, k): {\bf j'}
  \subset {\bf j}} h_{\bf j} $
for each $1 \leq \ell < k$ and
${\bf j'} \subset \mathcal{P}(n, \ell)$.  Note that there are
$O(n^{\ell})$ many ${\bf j'} \in \mathcal{P}(n, \ell)$, and each
$h_{\bf j' }$ is a sum of $O(n^{k - \ell})$ many terms.  Finding
$h_{\bf j'}$ for all ${\bf j'} \in \mathcal{P}(n, \ell)$ and
$1\leq \ell < k$ thus requires $O(n^k)$ operations, and with these as
ingredients, by \eqref{hibar}, one can compute each $\bar h_{\bf i}$
in $O(1)$ operations if
$ \sum_{{\bf j } \in \mathcal{P}(n, k)} h_{\bf j}$ is already known.
But the quantity $ \sum_{{\bf j } \in \mathcal{P}(n, k)} h_{\bf j}$
only has to be computed once, with another $O(n^k)$ computations.
Consequently, step (ii) involves $O(n^k)$ operations, and so does
the final summation in step (iii).

\subsection{Simulations}

We first consider the sizes of tests based on our statistics $S_\tau$,
$S_{\rho_s}$, $S_{t^*}$, $T_\tau$, $T_{\hat{\rho}_s}$ and $Z_{t^*}$
that we introduced in \secref{asympt-null-distr}.
For comparison, we also consider the sum of squared Pearson
correlations $S_r$ from \cite{Schott2005}; recall \eqref{SchottStat}.
Each test compares a rescaled test statistic to the limiting standard
normal distribution from \thmref{main} and Corollary~\ref{spearnull}.
Targeting a size of 0.05, the null hypothesis $H_0$ is rejected if the
value of the rescaled statistic exceeds the 95th percentile of the
standard normal distribution.  \tabref{tncsize} gives Monte-Carlo
estimates of finite-sample sizes for different combinations of $n$ and
$m$.
The data underlying the table are i.i.d.~noncentral $t$ with $\nu = 3$
degrees of freedom and noncentrality parameter $\mu = 2$.  For each
combination of $m$ and $n$, the sizes of the tests are calculated from
$5,000$ independently generated data sets.  As expected, the tests
that use rank-based statistics all have their sizes get closer to the
nominal $0.05$ when $m$ and $n$ increase, but the test based on $S_r$
is not valid as it rejects too often. Recall that Schott's limit
theorem is derived under a Gaussian assumption. For certain new
non-parametric tests introduced in this paper, the test sizes are not
very satisfactory when $n$ is small, but they all get close to the
nominal $0.05$ level once $n$ becomes $128$, indicating that the
asymptotics described by \thmref{main} kicks in. Surprisingly, the
test given by $S_\rho$ has good size even for very small $n$.  It would
be of interest to explore more refined results, such as a Berry-Ess\'{e}en
bound or an Edgeworth expansion
for the normal convergences of
\thmref{main} in future research.

Next, we consider the power of the tests, as studied in
\secref{power}.  For different combinations of $(m, n)$, we generate
data as $n$ independent draws from three different $m$-variate
elliptical distributions.  These are
\begin{enumerate}
\item the $m$-variate normal distribution:  $N_m(0,  \Sigma)$, 
\item the $m$-variate $t$ distribution: $t_{\nu = 20, m}(\mu = 2\cdot{\bf 1}_m, \Sigma)$, and
\item the $m$-variate power exponential distribution: $\mathit{PE}(\mu = 0, \Sigma, \nu = 20)$. 
\end{enumerate}
Here, ${\bf 1}_m$ is the $m$-vector with all entries equal to $1$, and
the parametrizations of these distributions are in accordance with
\citet[pp.~8--10]{Oja}.  For each distribution, the scatter matrix
$\Sigma = (\sigma_{ij})$ is taken to be a matrix with $1$'s on the
diagonal and equal values for the off-diagonal entries, which are set
to obtain the signal strengths $\|\Theta_\tau\|_2^2 = 0.1$, $0.3$, and
$0.7$ based on Kendall's $\tau$.  We refer again to
\cite{ellipKendall} for the relationship between $\Sigma$ and
$\|\Theta_\tau\|_2^2$. 
The power, computed based on $500$ repetitions of experiments, for tests
based on $S_\tau$, $T_\tau$,
and the statistic $S^{\max}_\tau$ of \cite{HanLiu2014} are compared in
\tabref{empEvidence}.
As expected,  $S^{\max}_\tau$ is not well-adapted for detecting
the alternatives we generated.
For each $(m, n)$
combination and a given value of $\|\Theta_\tau\|_2^2$, the power of
the test based on $T_\tau$ is similar across different
data-generating distributions.  In contrast,
$S_\tau$ tends to yield more power for t-distributed data, and less
power for data with power exponential distribution. 
The stability of
the power rendered by $T_\tau$ points to our conjecture in
\secref{power} on the minimax optimality of $T_\tau$
over a wider class of distributions.

When the data are generated from multivariate normal distributions,
\tabref{empEvidence} includes a comparison to three further tests.
First, Schott's $S_r$ from~\eqref{SchottStat} yields a valid
(asymptotic) test in
this case.
As seen in \tabref{empEvidence}, the three
statistics, $S_\tau$, $T_\tau$ and $S_r$ give comparable power for
different combinations of $(m, n)$ and signal strength
$\|\Theta_\tau\|_2^2$.  Second, we tried the likelihood ratio test (LRT) with critical rejection region calibrated based on Corollary $1$ in \cite{JiangQi} whenever it is implementable, i.e. when $m < n$ in the table. It is generally less powerful than our new tests and $S_r$ in detecting the alternatives we consider. Lastly, we experimented with the statistic
proposed in \cite{CaiAndMa}, which again demonstrates similar power.
The test of \cite{CaiAndMa} is minimax rate optimal in detecting the
Frobenius norm signal $\|\Sigma - I_m\|_2$, but only for testing the
different hypothesis $\tilde H_0:\Sigma=I_m$
and under a Gaussian assumption on ${\bf X}$.  Under Gaussianity, our
hypothesis of independence $H_0$ from~(\ref{eq:H0}) is of course
equivalent to the $R=I_m$ instead.
Despite this mismatch, the comparable power of the test of
\cite{CaiAndMa} indicates that the three statistics $S_\tau$, $T_\tau$
and $S_r$ are all powerful in detecting the signal
$\|R - I_m\|_2 \asymp \|\Theta_\tau\|_2$; recall that our experiment
has $\Sigma$ with 1's on the diagonal so that
$\Sigma=R$.  Lastly, we speculate that $S_r$
is minimax optimal in detecting the signal $\|R - I_m\|_2$ for the
null hypothesis $H_0$ \emph{under a Gaussian assumption} on
${\bf X}$, although to our knowledge this has not yet been
demonstrated theoretically in the literature; see also the last
section of \cite{CaiAndMa} for other related open problems.

To provide further evidence for the conjectures we have made, we
repeated the above simulation study in a case without equicorrelation.
Specifically, we generated data from elliptical distributions with
scatter matrices $\Sigma$ that are pentadiagonal.  The precise setup has
$\Sigma$ with $1$'s on the diagonal, equal values for the entries
$\sigma_{ij}$, $1 \leq |i - j| \leq 2$, and zeros elsewhere.
The results are reported in \tabref{empConj} and lead to similar
conclusions as \tabref{empEvidence}.


Finally, in \tabref{dirtpowernew}, we report Monte Carlo estimates of
power in a setting of data contamination and without
restricting solely to Kendall's tau.
We generate data as $n$ independent random vectors
${\bf X}_1,\dots,{\bf X}_n$ whose $m$ coordinates are dependent.  Each
${\bf X}_i$ is multivariate normal, with mean vector zero and
pentadiagonal covariance matrix.  Precisely,
${\bf X}_i \sim N_m \left(0, \Sigma_{\text{band}2} \right)$, where
$\Sigma_{\text{band}2}=(\sigma_{ij})$ has diagonal entries
$\sigma_{ii}=1$ and entry $\sigma_{ij}=0.1$ if $1\le |i - j| \leq 2$
and $\sigma_{ij}=0$ if $|i-j| \ge 3$.  For each combination of
$(n, m)$, we randomly select $5 \%$ of the $nm$ values of the data
matrix to be contaminated. Each selected value is replaced by an
independent draw from $N(2.5, 0.2)$ multiplied with a random sign.
Such outliers tend to decrease observed correlations, but the rank
correlations are affected less than Pearson correlations.  The
empirical power of these tests is computed based on $500$ repetitions
of experiments.  As the results in \tabref{dirtpowernew} show,
Schott's $S_r$ tends to give smaller power than the other statistics.
At the larger sample sizes, when the test have approximately nominal
size (recall Table~\ref{tab:tncsize}), the `Kendall statistics'
$S_\tau$ and $T_\tau$ show rather similar power, and the same happens
for $S_{\rho_s}$ and $T_{\hat\rho_s}$.  For the Bergsma-Dassios
statistics, there is some evidence that $Z_{t^*}$ has greater power
than $S_{t^*}$ in this setting.

\subsection{Comparison of the statistics}

When data are approximately Gaussian, the statistic $S_r$ of
\cite{Schott2005} yields a powerful test.  Since the computation of a
Pearson correlation is linear in the sample size $n$, it is
inexpensive to compute, and its distribution is
well-approximated by a normal limit at surprisingly small sample
sizes (see Table $1$ in the original paper of \cite{Schott2005}).  However, as one would expect, our simulations show that the
size of the test may be far from nominal in non-Gaussian settings.

The Kendall and Spearman `sum of squares' $S_\tau$ and $S_{\rho_s}$
are attractive alternatives that are nearly as efficient to compute as
$S_r$.  The use of rank correlations guards against effects of
non-Gaussianity all the while leading to rather little loss in power
when data are indeed Gaussian. Compared to $S_{\rho_s}$, $S_\tau$ requires somewhat larger samples for the normal approximation to the null distribution to be useful.


The statistics $T_h$ similarly guard against non-Gaussianity but are
computationally more costly to use.  However, as we explored in the
case of $T_\tau$, their unbiasedness as an estimator of the signal
strength leads to power that is similarly large across very different
alternatives.  We consider this an attractive feature and conjecture
that these statistics are minimax optimal, at least for a wide class of elliptical
distributions.


Another interesting assessment of independence is obtained by using
the statistics $Z_{t^*}$ and $Z_D$ based on Bergsma and Dassios' sign
covariance $t^*$ and Hoeffding's $D$, respectively.  Both $t^*$ and
$D$ have the intriguing property of providing a consistent assessment
of pairwise independence.  For continuous observation, their
expectations are zero if and only if the considered pair of random
variables is independent.  In the case of $t^*$, this also holds for
discrete variables.  Under independence, $t^*$ and $D$ are degenerate
U-statistics \citep{Nandy2016}.  The computational cost of their use
in $Z_{t^*}$ and $Z_D$ is comparable to that of $T_\tau$.  However,
determining the signal strength relevant for $Z_{t^*}$ is more
complicated than for $T_{\tau}$.  We are not aware of any literature
that would offer a simple relationship between the expectation of
$t^*$ or $D$ and the scatter matrix of an elliptical distribution.

\section{Discussion}\label{sec:conclude}

This paper treats nonparametric tests of independence using pairwise
rank correlations or, more precisely, rank correlations that are also
U-statistics.  As reviewed in Section~\ref{sec:rank-cor}, the
motivating examples are Kendall's tau and Spearman's rho but also
Hoeffding's $D$ and Bergsma and Dassios' sign covariance $t^*$.  The
latter two correlations allow for consistent assessment of pairwise
independence but form degenerate U-statistics.  With a view towards
alternatives in which dependence is ``spread out over many
coordinates'', we proposed statistics that are formed as
sums of many pairwise dependency signals
as explained in
Section~\ref{sec:main}.  In a high-dimensional regime in which both 
the number of variables $m$ and the sample size $n$ tend to infinity, 
we derived normal limits for the null distributions
of these statistics
(Section~\ref{sec:asympt-null-distr}).  Our general framework gives
results for U-statistic degeneracy of order up to two.  Finally, we
explored aspects of power theoretically and in simulations
(Sections~\ref{sec:power} and~\ref{sec:simulations}).

Under the null hypothesis of independence, the $m$ rank vectors are
independent, each following a uniform distribution on the symmetric
group $\mathfrak{S}_n$.  In small to moderate size problems, we may
thus implement exact tests using Monte Carlo simulation to compute
critical values.  However, for large-scale problems and/or when using
the computationally more involved $t^*$ or $D$, the asymptotic normal
distributions we derived furnish accurate approximations and allow for
great computational savings.


Our study of power has focused on the case of Kendall's tau.  In a
minimax paradigm and for Gaussian equicorrelation alternatives we
showed rate-optimality for the test based on $T_\tau$, the unbiased
estimator of the signal strength defined
via~(\ref{sum_of_w_kern_stat}) with kernel $h=h_\tau$.  It would be an
interesting problem for future work to prove such rate-optimality more
broadly, for more general alternatives as well as other kernels.  In
particular, for the kernel associated to Kendall's tau, we conjectured
in Section~\ref{equi} that rate-optimality holds for
alternatives from a wide class of elliptical distributions.

\section{Acknowledgments}

The authors would like to thank the Editor and the referees for their very helpful comments which lead to many improvements in this paper. 

\appendix

\section{Motivation of Schott's statistic as a Rao score} \label{sec:Rao}

We show that, up to a rescaling by the squared sample size, the
statistic $S_r$ from \eqref{SchottStat} is Rao's score statistic in
the multivariate normal setting. Let ${\bf X}_1, \dots,{\bf X}_n$ be
i.i.d.\ $m$-variate normal random vectors with mean vector
$\boldsymbol{\mu}$ and \emph{precision} matrix $K = (\kappa^{(pq)})$.  Let
$\bar{\bf X} := \tfrac{1}{n} \sum_i{\bf X}_i$ be the sample mean
vector, and let
$W = (w^{(pq)}) := \tfrac{1}{n} \sum_i ({\bf X}_i - \bar{\bf X} )(
{\bf X}_i - \bar{\bf X})^T $ be the sample covariance matrix.  The
score test considers the gradient $\nabla l_n$ of the multivariate normal
log-likelihood function
\[
l_n(\boldsymbol{\mu},K) \;=\; \frac{1}{2} n\log |K| - \frac{1}{2}\sum_i({\bf X}_i - \boldsymbol{\mu})^T K ({\bf X}_i - \boldsymbol{\mu})
\]
at the maximum likelihood estimate
$ (\hat{\boldsymbol{\mu}}_{0}, \hat{K}_{0})$ under the null hypothesis
$H_0$ from~(\ref{eq:H0}).  
Specifically, the score test rejects $H_0$ for large values of
\begin{equation} \label{quadform}
\nabla l_n(\hat{\boldsymbol{\mu}}_0,\hat K_0)^T  I(
\hat{\boldsymbol{\mu}}_0,\hat K_0)^{-1}  \nabla
l_n(\hat{\boldsymbol{\mu}}_0,\hat K_0),
\end{equation}
where $I(\boldsymbol{\mu},K)$ is the Fisher-information matrix,
$\hat{\boldsymbol{\mu}}_{0}=\bar{\bf X}$ and
$\hat{K}_{0}= \diag(w^{(11)}, \dots,w^{(mm)})^{-1}$.

Routine calculations show that
\begin{equation} \label{partials}
 \frac{\partial l_n}{\partial \mu}\Bigg|_{\boldsymbol{\mu}
   =\hat{\boldsymbol{\mu}}_{0},K= \hat{K}_{0}} = 0  \text{, } \quad
 \frac{\partial l_n}{\partial \kappa^{(pq)}}\Bigg|_{\boldsymbol{\mu} =\hat{\boldsymbol{\mu}}_{0},K= \hat{K}_{0}}   = \begin{cases}
   0 &\text{if}
    \quad p = q,\\
 - nw^{(pq)} &\text{if} \quad p< q.
  \end{cases}
\end{equation}
Moreover, for $p < q$ and $p' < q'$,
\begin{align} \label{fisher}
[I({\hat{\boldsymbol{\mu}}_{0}, \hat{K}_{0}})]_{\kappa^{(pq)}, \kappa^{(p'q')}} 
&= \mathbb{E}_{\hat{\boldsymbol{\mu}}_{0}, \hat{K}_{0}}[ (X^{(p)} - \mu^{(p)})(X^{(q)} - \mu^{(q)})(X^{(p')} - \mu^{(p')})(X^{(q')} - \mu^{(q')})]\notag\\
&= \begin{cases}
  ( [\hat{K}_{0}]_{pp} [\hat{K}_{0}]_{qq} )^{-1}&\text{if}
    \quad (p, q) =(p',  q'),\\
 0&\text{if}
    \quad (p, q) \not=(p',  q'),
  \end{cases}
\end{align}
where $\mathbb{E}_{\hat{\boldsymbol{\mu}}_{0}, \hat{K}_{0}}$ means
taking expectation under a multivariate normal distribution with mean
$\hat{\boldsymbol{\mu}}_{0}$ and precision matrix $\hat{K}_{0}$. In
light of \eqref{partials} and \eqref{fisher}, one obtains that the
statistic from \eqref{quadform} is equal to $n^2$ times Schott's
statistic $S_r$ from \eqref{SchottStat}.

\section{Technical lemmas}\label{sec:techproof}

The following lemma will be used to prove both \lemsref{4prod_lem1}
and \lemssref{4prod_lem2} below, as well as \lemsref{moment_lem} and
\lemssref{degen_relate}.  We make use of the following notion of
multisets.  For $1 \leq k \leq n$, if ${\bf i}^1, \dots, {\bf i}^r$
are tuples in $\mathcal{P}(n, k)$, let the pair
$(\cup_{\omega = 1}^r {\bf i}^\omega, f_m)$ be the multiset associated
with $\cup_{\omega = 1}^r {\bf i}^\omega$, where
$f_m : \cup_{\omega = 1}^r {\bf i}^\omega \longrightarrow \mathbb{N}$
is the multiplicity function such that $f_m (i)$ is the number of
occurrences of index $i$ in the sets ${\bf i}^1, \dots, {\bf i}^r$.

\begin{lemma} \label{lem:tech_lem}
Let $h : (\mathbb{N}^2)^k \longrightarrow \mathbb{R}$ be a kernel that
is symmetric in its $k$ arguments and has order of degeneracy $d$ under $H_0$.
\begin{enumerate}
\item Suppose ${\bf i}^1, \dots, {\bf i}^4 \in \mathcal{P}(n, k)$.  If
  $|\cup_{\omega = 1}^4 {\bf i}^\omega| > 4k - 2d$, then
\[
\mathbb{E}_0 \left[  
\prod_{\omega = 1}^4 h\left({\bf R}^{(p^\omega q^\omega)}_{{\bf i}^\omega}\right) 
\right] = 0
\]
for all $1 \leq p^\omega \not= q^\omega \leq m$, $\omega = 1, \dots,
4$. If $|\cup_{\omega = 1}^4 {\bf i}^\omega| = 4k - 2d$, then
$\mathbb{E}_0 [ \prod_{\omega = 1}^4 h({\bf R}^{(p^\omega
  q^\omega)}_{{\bf i}^\omega}) ] $ is nonzero only if $|{\bf i}^\omega
\cap (\cup_{\omega' \not= \omega} {\bf i}^{\omega'})| = d$ for all
$\omega = 1, \dots, 4$, and in this case the multiplicity function $f_m$
of the multiset $(\cup_{\omega = 1}^4 {\bf i}^\omega, f_m)$ takes
value either $1$ or $2$.

\item Suppose ${\bf i}^1, \dots, {\bf i}^8 \in \mathcal{P}(n, k)$.  If
  $|\cup_{\omega = 1}^8 {\bf i}^\omega| > 8k - 4d$, then
\[
\mathbb{E}_0 \left[  
\prod_{\omega = 1}^8 h\left({\bf R}^{(p^\omega q^\omega)}_{{\bf i}^\omega}\right) 
\right] = 0
\]
for all $1 \leq p^\omega \not= q^\omega \leq m$, $\omega = 1, \dots,
8$. If $|\cup_{\omega = 1}^8 {\bf i}^\omega| = 8k - 4d$, then
$\mathbb{E}_0 [ \prod_{\omega = 1}^8 h({\bf R}^{(p^\omega
  q^\omega)}_{{\bf i}^\omega}) ] $ is nonzero only if $|{\bf i}^\omega
\cap (\cup_{\omega' \not= \omega} {\bf i}^{\omega'})| = d$ for all
$\omega = 1, \dots, 8$, and in this case the multiplicity function
$f_m$ of the multiset $(\cup_{\omega = 1}^8 {\bf i}^\omega, f_m)$
takes value either $1$ or $2$.
\end{enumerate}
\end{lemma}

\begin{proof}
  We consider the first claim $(i)$.  Since
  ${\bf i}^1, \dots, {\bf i}^4$ are tuples in $\mathcal{P}(n, k)$, the
  multiplicity function $f_m$ of the multiset
  $(\cup_{\omega = 1}^4 {\bf i}^\omega, f_m)$ is such that
  $\sum_{i \in \cup_{\omega = 1}^4 {\bf i}^\omega} f_m(i) = 4k$. If
  $|\cup_{w = 1}^4 {\bf i}^\omega| > 4k - 2d$, the cardinality of the
  set $\{i \in \cup_{w = 1}^8 {\bf i}^\omega : f_m(i) = 1\}$ must be
  greater than $4k - 4d$, in which case there exists an $\omega'$ so
  that
  $c:= |{\bf i}^{\omega'} \cap (\cup_{\omega \not = \omega'} {\bf
    i}^\omega)| < d$.
  By symmetry, we may assume $w' = 1$ without loss of generality.

  Let
  ${\bf j} = (j_{1}, \dots, j_c) = {\bf i}^1 \cap (\cup_{w \not = 1}
  {\bf i}^\omega)$ as sets. Then, conditional on
  ${\bf X}^{(p^1 q^1)}_{\bf j}$, we have that
  $h({\bf R}_{{\bf i}^1}^{(p^1 q^1)}) $ is independent of all other
  factors $h({\bf R}_{{\bf i}^\omega}^{(p^\omega q^\omega)})$ for
  $\omega = 2, \dots, 4$.  Since $h$ has order of degeneracy $d$ under
  $H_0$, by \eqref{gen_prop},
  $\mathbb{E}_0[h({\bf R}_{{\bf i}^1}^{(p^1 q^1)}) |{\bf X}^{(p^1
    q^1)}_{\bf j} ] = 0$.  Therefore, by the aforementioned
  conditional independence,
\[
\mathbb{E}_0 \left[  
\prod_{\omega = 1}^4 h\left({\bf R}^{(p^\omega q^\omega)}_{{\bf i}^\omega}\right) 
\middle|{\bf X}^{(p^1 q^1)}_{\bf j} \right] \equiv 0
\]
as a function of ${\bf X}^{(p^1 q^1)}_{\bf j}$.  This in turn implies
that $\mathbb{E}_0 [  
\prod_{\omega = 1}^4 h({\bf R}^{(p^\omega q^\omega)}_{{\bf i}^\omega}) 
] = 0$.

The necessary condition for
$\mathbb{E}_0 [ \prod_{\omega = 1}^4 h({\bf R}^{(p^\omega
  q^\omega)}_{{\bf i}^\omega}) ] $
to be nonzero when $|\cup_{\omega = 1}^4 {\bf i}^\omega| = 4k - 2d$
can be argued similarly, and we omit the details.

The proof of $(ii)$ is analogous to that of $(i)$.  Again, we omit the details.
\end{proof}

The following three lemmas will be used to prove
\lemref{ell2-convergence}. Recall the notational shorthand
$\bar{U}^{(pq)}_h := \big(U_h^{(pq)}\big)^2 - \mu_h$ for
$1 \leq p < q \leq m$, defined in the proof of \thmref{main}.

\begin{lemma} \label{lem:4prod_lem1} Suppose
  $1 \leq p , q, l, u\leq m$ are four distinct indices, and $h$ is a
  kernel of order of degeneracy $d$ satisfying \assumpref{sp_assume}
  under $H_0$. Then
  \[
\mathbb{E}_0\left[
    \bar{U}_h^{(pl)}\bar{U}_h^{(ql)}\bar{U}_h^{(pu)}\bar{U}_h^{(qu)}
  \right]=O(n^{-4d -1}).
  \]
\end{lemma}
\begin{proof}

  Without loss of generality, we prove the result for
  $(p,q, l ,u ) = (1, 2, 3, 4)$. Note that for any four distinct
  indices $1 \leq p_1, p_2, p_3, p_4 \leq m$, the antiranks
  ${\bf R}^{(p_1)|(p_2)}$, ${\bf R}^{(p_2)|(p_3)}$,
  ${\bf R}^{(p_3)|(p_4)}$ are independent. Since $\bar{U}^{(13)}$,
  $\bar{U}^{(23)}$, $\bar{U}^{(14)}$, $\bar{U}^{(24)}$ are functions
  of ${\bf R}^{(1)|(3)}$, ${\bf R}^{(2)|(3)}$, ${\bf R}^{(1)|(4)}$,
  ${\bf R}^{(2)|(4)}$, respectively, on expansion,
\begin{multline*}
\mathbb{E}_0\left[ \bar{U}^{(13)}\bar{U}^{(23)}\bar{U}^{(14)}\bar{U}^{(24)} \right]
= \mathbb{E}_0\left[  \left(U_h^{(13)}\right)^2  \left(U_h^{(23)}\right)^2 \left(U_h^{(14)}\right)^2 \left(U_h^{(24)}\right)^2 \right] -
\mu_h^4 \\
=\mathbb{E}_0\left[  \left(U_h^{(13)}\right)^2  \left(U_h^{(23)}\right)^2 \left(U_h^{(14)}\right)^2 \left(U_h^{(24)}\right)^2 \right]  - {k \choose d}^8 \left(\frac{ d!\zeta_{d}}{n^{d}} \right)^4 + O\left(n^{-4d -1}\right),
\end{multline*}
where the last equality follows from \lemref{moment_lem}$(i)$.   The
proof is completed if we are able to show that
\begin{equation}
\mathbb{E}_0\left[  \left(U_h^{(13)}\right)^2  \left(U_h^{(23)}\right)^2 \left(U_h^{(14)}\right)^2 \left(U_h^{(24)}\right)^2 \right]  = {k \choose d}^8 \left(\frac{ d!\zeta_{d}}{n^{d}} \right)^4 + O\left(n^{-4d -1}\right). \label{eq:orderofrightprod}
\end{equation}

For ${\bf i}^\omega \in \mathcal{P}(n, k)$, $\omega = 1, \dots, 8$, we
define 
\begin{multline} \label{P_function}
P({\bf i}^1, \dots, {\bf i}^8)  
  = 
\left(\prod_{\omega = 
      1}^2 h\left(\mathbf{R}^{(13)}_{\mathbf{i}^w} \right) 
      \right)
\left(\prod_{\omega = 
      3}^4 h\left(\mathbf{R}^{(23)}_{\mathbf{i}^w} \right) 
      \right)
\left(\prod_{\omega = 
      5}^6 h\left(\mathbf{R}^{(14)}_{\mathbf{i}^w} \right) 
      \right)
\left(\prod_{\omega = 
      7}^8 h\left(\mathbf{R}^{(24)}_{\mathbf{i}^w} \right) 
      \right) . 
 \end{multline}
  Then on expansion,
\begin{equation} 
\mathbb{E}_0\left[  \left(U_h^{(13)}\right)^2  \left(U_h^{(23)}\right)^2 \left(U_h^{(14)}\right)^2 \left(U_h^{(24)}\right)^2 \right] ={n \choose k}^{-8} \sum_{ \substack{\mathbf{i}^\omega \in\mathcal{P}(n, k) \\  1 \leq \omega  \leq 8}}
\mathbb{E}_0\left[P({\bf i}^1, \dots, {\bf i}^8) \right] \label{sum_of_Ps}
. 
\end{equation}
Each summand $\mathbb{E}_0[P({\bf i}^1, \dots,
{\bf i}^8)]$ on the right hand side of \eqref{sum_of_Ps} depends on the multiset $(\cup_{w = 1}^8 {\bf
  i}^w, f_m)$.
 If $|\cup_{w = 1}^8 {\bf i}^\omega| > 8k - 4d$, by \lemref{tech_lem}$(ii)$, $\mathbb{E}_0[P({\bf i}^1, \dots,
{\bf i}^8)] = 0$.

If $|\cup_{w = 1}^8 {\bf i}^\omega| = 8k - 4d$, 
by \lemref{tech_lem}$(ii)$, for
$\mathbb{E}_0[P({\bf i}^1, \dots, {\bf i}^8)]$ to be non-zero it is
necessary that
$|{\bf i}^{\omega'} \cap (\cup_{\omega \not = \omega'} {\bf
  i}^\omega)| = d$
for all $\omega' = 1, \dots, 8$, in which case $f_m$ takes the value
$1$ or $2$.  Suppose this is true. Under $H_0$, conditioning on
${\bf X}^{(1)}_{{\bf i}^1\cap ({\bf i}^2 \cup {\bf i}^5 \cup {\bf
    i}^6)}$
and
${\bf X}^{(3)}_{{\bf i}^1\cap ({\bf i}^2 \cup {\bf i}^3 \cup {\bf
    i}^4)}$,
$h({\bf R}_{{\bf i}^1}^{(1, 3)}) $ is independent of all other
multiplicative factors on the right hand side of
\eqref{P_function}. If ${\bf i}^1$ intersects with the set
$\cup_{\omega = 3}^8{\bf i}^{\omega} \setminus {\bf i}^2$, at least
one of ${\bf i}^1\cap ({\bf i}^2 \cup {\bf i}^5 \cup {\bf i}^6)$ and
${\bf i}^1\cap ({\bf i}^2 \cup {\bf i}^3 \cup {\bf i}^4)$ has
cardinality less than $d$ given that $f_m \leq 2$, and then
\assumpref{sp_assume} yields that
\[
\mathbb{E}_0\left[ h\left({\bf R}_{{\bf i}^1}^{(13)}\right) \middle|
{\bf X}^{(1)}_{{\bf i}^1\cap ({\bf i}^2 \cup {\bf i}^5 \cup {\bf i}^6)}, {\bf X}^{(2)}_{{\bf i}^1\cap ({\bf i}^2 \cup {\bf i}^3 \cup {\bf i}^4)}\right] = 0.
\]
Hence, $\mathbb{E}_0[P({\bf i}^1, \dots, {\bf i}^8)] = 0$ by the
aforementioned conditional independence. Similarly,
${\bf i}^3, {\bf i}^5, {\bf i}^7$ can only intersect with
${\bf i}^4, {\bf i}^6, {\bf i}^8$, respectively, to ensure that
$\mathbb{E}_0[P({\bf i}^1, \dots, {\bf i}^8)]$ does not equal
zero. When this is the case, we have that
$|{\bf i}^\omega \cap {\bf i}^{w+1}| = d$ for $w = 1, 3, 5, 7$, and then
the four sets
${\bf i}^1 \cap {\bf i}^2, {\bf i}^3 \cap {\bf i}^4, {\bf i}^5 \cap
{\bf i}^6, {\bf i}^7 \cap {\bf i}^8$ are disjoint and
$\mathbb{E}_0[P({\bf i}^1, \dots, {\bf i}^8)] = (\zeta_d^h)^4$.

As a result, when $|\cup_{w = 1}^8 {\bf i}^\omega| = 8k - 4d$, $\mathbb{E}_0[P({\bf i}^1, \dots,  {\bf i}^8)]$ is only nonzero with value $(\zeta_d^h)^4$ for
\begin{multline*}
 {n \choose 8k -4d}  {8k - 4d \choose 2k - d, 2k - d, 2k - d, 2k - d} {2k - d \choose d}^4 {2k - 2d\choose k -d}^4 = \\
\frac{n!}{(n - 8k + 4d)!( (k -d)!)^8 (d!)^4}
\end{multline*}
choices of $({\bf i}^1, \dots, {\bf i}^8)$.  This count is obtained as
follows.  First, pick $8k - 4d$ indices from the set $\{1, \dots, n\}$,
and note that there are ${8k - 4d \choose 2k - d, 2k - d, 2k - d, 2k -
d }$ ways of partitioning the $8k - 4d$ indices into the four sets
${\bf i}^1 \cap {\bf i}^2, {\bf i}^3 \cap {\bf i}^4, {\bf i}^5 \cap
{\bf i}^6, {\bf i}^7 \cap {\bf i}^8$. For each $w \in \{1, 3, 5, 7\}$,
there are ${2k - d\choose d}$ choices for the $d$ shared common indices in ${\bf
  i}^w\cap {\bf i}^{w+1}$, and there are ${2k - 2d\choose k -d}$ ways
of distributing the remaining $2k -2d$ indices to ${\bf i}^\omega$ and ${\bf
  i}^{w+1}$. Since the count of the summands $\mathbb{E}_0[P({\bf i}^1, \dots, {\bf
  i}^8)]$ with $|\cup_{w = 1}^8 {\bf i}^\omega| < 8k - 4d$ is of the order
$O(n^{8k - 4d -1})$, we find from \eqref{sum_of_Ps} that
\begin{align*}
&\mathbb{E}_0\left[  \left(U_h^{(13)}\right)^2
  \left(U_h^{(23)}\right)^2 \left(U_h^{(14)}\right)^2
  \left(U_h^{(24)}\right)^2 \right] 
\\
&= {n \choose k}^{-8} \left(\frac{(\zeta_d^h)^4n!}{(n - 8k + 4d)!( (k -d)!)^8 (d!)^4}  + O\left(n^{8k - 4d -1}\right)\right)\\
&= {k \choose d}^8\frac{(d! \zeta^h_d)^4}{  n^{4d} } + O\left(n^{-4d -1}\right).
\end{align*}
This concludes the proof of~(\ref{eq:orderofrightprod}).
\end{proof}

\begin{lemma} \label{lem:4prod_lem2}
Suppose $1 \leq p , q, l, u\leq m$ are four distinct indices, and $h$
is a kernel of order of degeneracy $d$ satisfying
\assumpref{sp_assume} under $H_0$. Then 
\[
\mathbb{E}_0\left[ W_h^{(pl)}W_h^{(ql)}W_h^{(pu)}W_h^{(qu)} \right] = O(n^{-4d -1}).
\]
\end{lemma}

\begin{proof}
Again, without loss of generality, we prove the result for $(p,q, l ,u
) = (1, 2, 3, 4)$. Given ${\bf i}^\omega \in \mathcal{P}(n, 2k)$,
$\omega = 1, \dots, 4$, we define 
\begin{align} \label{Q_function}
Q({\bf i}^1, \dots, {\bf i}^4)  
  &=h^W\left({\bf R}^{(13)}_{{\bf i}^1}\right)h^W\left({\bf R}^{(2
    3)}_{{\bf i}^2}\right)h^W\left({\bf R}^{(14)}_{{\bf
    i}^3}\right)h^W\left({\bf R}^{(24)}_{{\bf i}^4}\right)\\ \nonumber
&= 
{2k \choose k}^{-4} \sum_{ \substack{\tilde{{\bf i}}^\omega \subset {\bf i}^\omega\\ |\tilde{{\bf i}}^\omega| = k }} h^{(13)}_{{\bf i}^1, \tilde{{\bf i}}^1}  \cdot h^{(23)}_{{\bf i}^2, \tilde{{\bf i}}^2} \cdot h^{(14)}_{{\bf i}^3, \tilde{{\bf i}}^3} \cdot h^{(24)}_{{\bf i}^4, \tilde{{\bf i}}^4},
 \end{align}
where 
$
h^{(pq)}_{{\bf i}^w, \tilde{{\bf i}^\omega}}:= h\left({\bf R}^{(pq)}_{\tilde{ \bf i}^\omega}\right) h\left({\bf R}^{(pq)}_{{\bf i}^\omega \setminus\tilde{\bf i}^\omega}\right).
$
 By the definition from \eqref{W_stat_unscaled}, on expansion,
\begin{align} 
\mathbb{E}_0 \left[W_h^{(13)}W_h^{(23)}W_h^{(14)}W_h^{(24)}\right]
  \notag 
&= \frac{1}{{n \choose 2k}^{4}} \sum_{ \substack{\mathbf{i}^w \in\mathcal{P}(n, 2k), \\  1 \leq w  \leq 4}} \mathbb{E}_0 \left[Q\left({\bf i}^1, {\bf i}^2, {\bf i}^3, {\bf i}^4\right) \right] \notag\\
 &= \frac{1}{\left({n \choose 2k} {2k \choose k}\right)^{4}}\sum_{ \substack{\mathbf{i}^w \in\mathcal{P}(n, 2k), \\  \omega = 1, \dots, 4}}
\sum_{ \substack{\tilde{{\bf i}}^\omega \subset {\bf i}^\omega\\ |\tilde{{\bf i}}^\omega| = k }}\mathbb{E}_0 \left[  h^{(13)}_{{\bf i}^1, \tilde{{\bf i}}^1}  \cdot h^{(23)}_{{\bf i}^2, \tilde{{\bf i}}^2} \cdot h^{(14)}_{{\bf i}^3, \tilde{{\bf i}}^3} \cdot h^{(24)}_{{\bf i}^4, \tilde{{\bf i}}^4} \right].
\label{com_inner_summand}
\end{align}
It now suffices to show that 
\begin{equation}\label{ep_weird}
\mathbb{E}_0 \left[  h^{(13)}_{{\bf i}^1, \tilde{{\bf i}}^1}  \cdot h^{(23)}_{{\bf i}^2, \tilde{{\bf i}}^2} \cdot h^{(14)}_{{\bf i}^3, \tilde{{\bf i}}^3} \cdot h^{(24)}_{{\bf i}^4, \tilde{{\bf i}}^4} \right] = 0
\end{equation}
whenever $|\cup_{\omega = 1}^4{\bf i}^\omega| \geq 8k - 4d$, because
then the right hand side of \eqref{com_inner_summand} is of the order
${n \choose 2k}^{-4} {n \choose 8k - 4d - 1} = O(n^{- 4d - 1})$.

The value of a term 
\begin{multline} \label{eight_prod_form} h^{(13)}_{{\bf i}^1,
    \tilde{{\bf i}}^1} \cdot h^{(23)}_{{\bf i}^2, \tilde{{\bf i}}^2}
  \cdot h^{(14)}_{{\bf i}^3, \tilde{{\bf i}}^3} \cdot h^{(2
    4)}_{{\bf i}^4, \tilde{{\bf i}}^4} = h\left({\bf R}^{(1
      3)}_{\tilde{\bf i}^1}\right)h\left({\bf R}^{(13)}_{{\bf i}^1
      \setminus \tilde{\bf i}^1}\right) \cdots\cdots h\left({\bf
      R}^{(24)}_{\tilde{\bf i}^4}\right)h\left({\bf R}^{(2
      4)}_{{\bf i}^4 \setminus \tilde{\bf i}^4}\right) ,\end{multline}
depends on the multi set $(\cup_{\omega = 1}^4 {\bf i}^\omega, f_m )$,
where
$f_m : \cup_{\omega = 1}^4 {\bf i}^\omega \longrightarrow \mathbb{N}$
is the multiplicity function with $f_m(i)$ equal to the number of occurrences of
$i$ among the eight tuples
\begin{equation} \label{eight_tuples}
\tilde{\bf i}^1, {\bf i}^1 \setminus \tilde{\bf i}^1, \dots, 
\tilde{\bf i}^4, {\bf i}^4 \setminus \tilde{\bf i}^4
\;\in\; \mathcal{P}(n, k),
\end{equation}
and $\sum_{i \in \cup_{\omega = 1}^4 {\bf   i}^\omega} f_m(i) = 8k$. If
$|\cup_{\omega = 1}^4 {\bf   i}^\omega| = |\cup_{\omega = 1}^4
(\tilde{\bf   i}^\omega) \cup ({\bf i}^\omega\setminus \tilde{\bf
  i}^\omega)| > 8k - 4d$, by \lemref{tech_lem}$(ii)$, $\mathbb{E}_0 [
h^{(13)}_{{\bf i}^1, \tilde{{\bf i}}^1}  \cdot h^{(23)}_{{\bf
    i}^2, \tilde{{\bf i}}^2} \cdot h^{(14)}_{{\bf i}^3, \tilde{{\bf
      i}}^3} \cdot h^{(24)}_{{\bf i}^4, \tilde{{\bf i}}^4} ] = 0$.
We are left with the case $|\cup_{\omega = 1}^4 {\bf   i}^\omega| =  8k -
4d$.

If $|\cup_{\omega = 1}^4{\bf i}^\omega| = 8k - 4d$, then
\lemref{tech_lem}$(ii)$ yields that for $\mathbb{E}_0 [  h^{(13)}_{{\bf i}^1,
  \tilde{{\bf i}}^1}  \cdot h^{(23)}_{{\bf i}^2, \tilde{{\bf i}}^2}
\cdot h^{(14)}_{{\bf i}^3, \tilde{{\bf i}}^3} \cdot h^{(24)}_{{\bf
    i}^4, \tilde{{\bf i}}^4} ] $ to be non-zero, it is necessary (but
not sufficient, as seen below) that each of the eight tuples in
\eqref{eight_tuples} intersects with the union of the other seven at
exactly $d$ elements, with  $f_m(i) \leq 2$ for all $i \in \cup_{\omega = 1}^4{\bf i}^\omega$. In particular, since $\tilde{\bf i}^1$ is disjoint from ${\bf i}^1\setminus \tilde{\bf i}^1$, it is the case that 
\begin{equation} \label{intersect_rest_6}
|\tilde{\bf i}^1 \cap (\cup_{\omega= 2}^4 {\bf i}^\omega)| = d.
\end{equation}
When conditioning on ${\bf X}^{(3)}_{\tilde{\bf i}^1 \cap {\bf i}^2}$
and ${\bf X}^{(1)}_{\tilde{\bf i}^1 \cap {\bf i}^3}$, it is seen that
$h({\bf R}^{(13)}_{\tilde{\bf i}^1})$ is independent of the other
multiplicative factors on the right hand side of
\eqref{eight_prod_form}. Note that since $f_m$ is always less than or
equal to $2$, by \eqref{intersect_rest_6} one of
$\tilde{\bf i}^1 \cap {\bf i}^2$ and $\tilde{\bf i}^1 \cap {\bf i}^3$ must
have cardinality less than $d$. Hence, by \assumpref{sp_assume} we
have that
\[
\mathbb{E}_0\left[ h({\bf R}^{(13)}_{\tilde{\bf i}^1})\middle|{\bf X}^{(3)}_{\tilde{\bf i}^1 \cap {\bf i}^2}, {\bf X}^{(1)}_{\tilde{\bf i}^1 \cap {\bf i}^3}\right] = 0,
\]
and the aforementioned conditional independence yields
the claim from~(\ref{ep_weird}).
 \end{proof}

\begin{lemma} \label{lem:4prod_lem3} Suppose
  $1 \leq p , q, l, u\leq m$ are four distinct indices, and $h$ is a
  kernel of order of degeneracy $d$ satisfying \assumpref{sp_assume}
  under $H_0$. Then
  \[
\mathbb{E}_0\left[
    U_h^{(pl)}U_h^{(ql)}U_h^{(pu)}U_h^{(qu)}
  \right]=O(n^{-2d -1}).
  \]
\end{lemma}
\begin{proof}
  The proof uses similar counting techniques to that of
  \lemsref{4prod_lem1} and \lemssref{4prod_lem2} and is only
  simpler. We only sketch the argument.  Without loss of generality,
  let $(p, q, l, u) = (1, 2, 3, 4)$. On expansion, by defining
  $B({\bf i}^1, \dots, {\bf i}^4) := h\left({\bf R}^{(13)}_{{\bf
        i}^1}\right)h\left({\bf R}^{(2 3)}_{{\bf
        i}^2}\right)h\left({\bf R}^{(14)}_{{\bf
        i}^3}\right)h\left({\bf R}^{(24)}_{{\bf i}^4}\right)$,
\begin{equation} 
\mathbb{E}_0\left[  U_h^{(13)} U_h^{(23)} U_h^{(14)}U_h^{(24)} \right] ={n \choose k}^{-4} \sum_{ \substack{\mathbf{i}^\omega \in\mathcal{P}(n, k) \\  1 \leq \omega  \leq 4}}
\mathbb{E}_0\left[B({\bf i}^1, \dots, {\bf i}^4) \right].  \label{sum_of_Ps2}
\end{equation}
By \lemref{tech_lem}$(i)$,
$\mathbb{E}_0\left[B({\bf i}^1, \dots, {\bf i}^4) \right] = 0 $ if
$|\cup_{\omega = 1}^4 {\bf i}^\omega| > 4k - 2d$. When
$|\cup_{\omega = 1}^4 {\bf i}^\omega| = 4k - 2d$, one can also show
$\mathbb{E}_0\left[P({\bf i}^1, \dots, {\bf i}^4) \right] = 0$ by
using \lemref{tech_lem}$(i)$ and the property of the kernel given by
\assumpref{sp_assume}. Hence, there are at most
$O(n^{4k - 2d - 1})$ summands on the right hand side of
\eqref{sum_of_Ps2} and we conclude that
$\mathbb{E}_0\left[ U_h^{(13)} U_h^{(23)} U_h^{(14)}U_h^{(24)} \right]
={n \choose k}^{-4} O(n^{4k - 2d - 1}) = O(n^{-2d - 1})$.
\end{proof}

\section{Proofs for \secref{rank-cor}}
\label{sec:proof-rank-cor}

\begin{proof}[Proof of \lemref{general_properties}]
  Claim $(i)$ holds because the independence of
  $\mathbf{X}^{(1)}, \dots, \mathbf{X}^{(m)}$ implies that the rank
  vectors $\mathbf{R}^{(1)}, \dots, \mathbf{R}^{(m)}$ are i.i.d.  For
  assertion $(ii)$, note that, by the permutation symmetry of $g$ in
  its $n$ arguments, $g^{(pq)}$ is a function of the antirank of
  $\mathbf{X}^{(q)}$ in relation to $\mathbf{X}^{(p)}$
  \citep[p.~63]{SidakAndHajek}.  These antiranks, which we denote by
  $\mathbf{R}^{(q)|(p)}$, are uniformly distributed on
  $\mathfrak{S}_n$ for any fixed choice of $\mathbf{X}^{(p)}$, which
  yields the independence of $g^{(pq)}$ and $\mathbf{X}^{(p)}$.
  Similarly, $g^{(pq)}$ is independent $\mathbf{X}^{(q)}$.  (Of
  course, $\mathbf{X}^{(p)}$ and $\mathbf{X}^{(q)}$ together determine
  $g^{(pq)}$.)  Claim $(iii)$ holds since the independence of
  $\mathbf{X}^{(1)}, \dots, \mathbf{X}^{(m)}$ implies that the $m-1$
  vectors of antiranks $\mathbf{R}^{(l)|(p)}$ for $p\not=l$ are
  mutually independent.  Finally, the pairwise independence stated in
  $(iv)$ is implied by the independence of
  $\mathbf{X}^{(1)}, \dots, \mathbf{X}^{(m)}$ and $(iii)$.
\end{proof}

\section{Proofs for \secref{main}} \label{sec:PfSec3}

\begin{proof}[Proof of \lemref{moment_lem}]
  It remains to prove claim $(iii)$ about the fourth moment of
  $U_h^{(pq)}$ when the kernel $h$ has its order of degeneracy $d$
  equal to $1$ or $2$ under $H_0$.  Without loss of generality, we can
  assume $(p, q) = (1, 2)$.  The fourth moment can be written as
  \begin{align}
    \mathbb{E}_0\left[ \left(U^{(12)}_h\right)^4\right]  
    &=   {n \choose k}^{-4} \sum_{ \mathbf{i}^1, \mathbf{i}^2,
      \mathbf{i}^3, \mathbf{i}^4 \in\mathcal{P}(n, k)} \mathbb{E}_0
      \left[ \prod_{\omega = 
      1}^4h\left(\mathbf{R}^{(1 2)}_{\mathbf{i}^\omega, 1},
      \dots,\mathbf{R}^{(1 2)}_{\mathbf{i}^\omega, k} \right)
      \right] \label{fourth_moment}. 
  \end{align}
 The value of each summand $\mathbb{E}_0
      \big[ \prod_{\omega = 
      1}^4h(\mathbf{R}^{(1 2)}_{\mathbf{i}^\omega} ) \big]$ in
 \eqref{fourth_moment} depends on the 
  multiset $(\cup_{\omega = 1}^4 \mathbf{i}^\omega, f_m)$ with 
\begin{equation}\label{multicardin}
\sum_{i \in \cup_{\omega = 1}^4 {\bf i}^\omega}   f_m(i) = 4k;
\end{equation}
we use the multiset notation introduced in the first paragraph of
Appendix~\ref{sec:techproof}.

By \lemref{tech_lem}$(i)$, we have
$\mathbb{E}_0 \big[ \prod_{\omega = 1}^4h(\mathbf{R}^{(1
  2)}_{\mathbf{i}^\omega} ) \big] = 0$
if $\left|\cup_{\omega = 1}^4 \mathbf{i}^\omega \right| > 4k - 2d$. If
$\left|\cup_{\omega = 1}^4 \mathbf{i}^\omega \right| < 4k - 2d$, there
are at most ${n \choose 4k - 2d -1}$ choices for the set
$\cup_{\omega = 1}^4 \mathbf{i}^\omega$.  Since $h$ is bounded, it
thus holds that
\[
{n \choose k}^{-4} \sum_{ 
\substack{\mathbf{i}^1, \mathbf{i}^2,
      \mathbf{i}^3, \mathbf{i}^4 \in\mathcal{P}(n, k)\\
|\cup_{\omega = 1}^4 \mathbf{i}^\omega |< 4k - 2d}} \mathbb{E}_0
      \left[ \prod_{\omega = 
      1}^4h\left(\mathbf{R}^{(1 2)}_{\mathbf{i}^\omega}\right) \right] = O(n^{-2d- 1}).
\]
Therefore, to complete the proof, it suffices to show that 
\begin{multline}\label{just_right_order}
{n \choose k}^{-4}  \sum_{ 
\substack{\mathbf{i}^1, \mathbf{i}^2,
      \mathbf{i}^3, \mathbf{i}^4 \in\mathcal{P}(n, k)\\
|\cup_{\omega = 1}^4 \mathbf{i}^\omega |=  4k - 2d}}  \mathbb{E}_0
      \left[ \prod_{\omega = 
      1}^4h\left(\mathbf{R}^{(1 2)}_{\mathbf{i}^\omega}\right) \right] 
 = 
  \begin{cases}
    \frac{3k^{4}(\zeta^h_1)^2}{n^2} + O(n^{-3}) &\text{if} \quad d = 1,\\
      {k \choose 2}^4\frac{12}{n^{4}}\left((\zeta_2^h)^2+ 
 4\eta^h\right)  + O(n^{-5})&\text{if} \quad d = 2.
  \end{cases}
\end{multline}

By \lemref{tech_lem}$(i)$, when
$|\cup_{\omega = 1}^4 \mathbf{i}^\omega |= 4k - 2d$, a summand
$\mathbb{E}_0 \big[ \prod_{\omega = 1}^4h(\mathbf{R}^{(1
  2)}_{\mathbf{i}^\omega}) \big]$
on the left hand side of \eqref{just_right_order} is non-zero only if
\begin{equation} \label{nec_con_nonzero}
|{\bf i}^{\omega} \cap (\cup_{\omega' \not = \omega} {\bf i}^{\omega'})| = d\text{ for all } \omega = 1, \dots, 4.
\end{equation}
For both $d = 1$ and $d  = 2$, \eqref{nec_con_nonzero} is true when the set
  $\{1,2,3,4\}$ can be partitioned into two disjoint sets $\Omega_1$ and
  $\Omega_2$ such that 
\begin{equation} \label{gen_partition}
|\Omega_1|=|\Omega_2|=2 \quad \text{and} \quad 
 |\cap_{\omega \in \Omega_1} \mathbf{i}^\omega|=|\cap_{\omega\in \Omega_2} \mathbf{i}^\omega|=d,
\end{equation}
  in which case
  $(\cup_{\omega\in \Omega_1} \mathbf{i}^\omega)\cap(\cup_{\omega\in \Omega_2}
  \mathbf{i}^\omega)=\emptyset$ and, by independence,
  \begin{equation}
    \label{fourth-moment-zeta1}
    \mathbb{E}_0 \left[ \prod_{\omega = 1}^4\left(h\left(\mathbf{R}^{(1
            2)}_{\mathbf{i}^\omega} \right)  \right)\right]  
    = \\
    \prod_{j=1}^2\;\mathbb{E}_0 \left[ \prod_{\omega\in
        \Omega_j}\left(h\left(\mathbf{R}^{(12)}_{\mathbf{i}^{\omega}} \right) 
      \right)\right] 
    \;=\; (\zeta_d^h)^2. 
  \end{equation}
  Next, we count how many summands on the left hand side of
  \eqref{just_right_order} have their indices
  ${\bf i}^1, \dots,{\bf i}^4$ satisfying the constellation in
  \eqref{gen_partition}. There are $\binom{n}{4k - 2d}$ choices for
  the set $\cup_{\omega = 1}^4 \mathbf{i}^\omega$.  Then there are
  $\frac{1}{2}{4k - 2d \choose 2k-d}$ partitions of
  $\cup_{\omega = 1}^4 \mathbf{i}^\omega$ into two subsets of equal
  cardinality.  Each of these subsets with cardinality $2k- d$ is to
  be split into two subsets that have $d$ elements in common.  We have
  ${2k - d} \choose d$ choices for this common element, and there are
  $\frac{1}{2} {2k - 2d \choose k -d}$ ways of partitioning the
  remaining elements to form the two subsets.  In the above counting
  process, no ordering is taken into account.  Hence, the number
  of summands in \eqref{fourth_moment} whose indices
  ${\bf i}^1, \dots, {\bf i}^4$ satisfy
  \eqref{gen_partition} is
  \begin{equation}
  4!{n\choose 4k - 2d}\frac{1}{2}{4k - 2d \choose 2k-d}
   \left[{2k - d\choose d}\frac{1}{2}{2k - 2d \choose k -d}
  \right]^2 =  \frac{3n!}{(n - 4k
    +2d)!\left[ d! \left((k-d)!\right)^2 \right]^2}.\label{fourth-moment-count}
  \end{equation}

  When $d = 1$, for any four tuples
  ${\bf i}^1, \dots, {\bf i}^4 \in \mathcal{P}(n, k)$ with
  $|\cup_{\omega = 1}^4 \mathbf{i}^\omega |= 4k - 2d = 4k - 2$,
  \eqref{nec_con_nonzero} is only satisfied when they can be described
  by the constellation in \eqref{gen_partition}. Since
 \begin{align} \label{easylimit}
   {n \choose
      k}^{-4} \frac{3n!}{(n - 4k
    +2d)!\left[ d! \left((k-d)!\right)^2 \right]^2} \;=\;
   {k \choose d}^4 \frac{ 3(d!)^2 }{n^{2d}}+O\left(n^{-2d -1}\right),
 \end{align}
 by \eqref{fourth-moment-zeta1} and \eqref{fourth-moment-count}, we
 have proved the equality in \eqref{just_right_order} for $d = 1$.

 When $d = 2$, in addition to \eqref{gen_partition}, there is another
 constellation for ${\bf i}^1, \dots, {\bf i}^1 \in \mathcal{P}(n, k)$
 that satisfies the condition in \eqref{nec_con_nonzero} subject to
 $|\cup_{\omega = 1}^4 \mathbf{i}^\omega |= 4k - 2d = 4k - 4$. If, up
 to relabeling of superscripts $\{1, \dots, 4\}$ for
 ${\bf i}^1, \dots, {\bf i}^4$, the multiset
 $(\cup_{\omega =1}^4 {\bf i}^\omega, f_m)$ is such that
\begin{align} \label{multiset1}
&|{\bf i}^1  \cap {\bf i}^2| =  |{\bf i}^2  \cap {\bf i}^3| = |{\bf i}^3  \cap {\bf i}^4| = |{\bf i}^4  \cap {\bf i}^1| =1 \quad \text{and}\\
 \label{multiset2}
&f_m(i) = 
  \begin{cases}
    2 \quad \text{if }  i \text{ belongs to any one of } {\bf i}^1  \cap {\bf i}^2, {\bf i}^2  \cap {\bf i}^3, {\bf i}^3  \cap {\bf i}^4 \text{ or } {\bf i}^4  \cap {\bf i}^1,\\
    1 \quad \text{otherwise},
  \end{cases}
\end{align} 
then \eqref{nec_con_nonzero} is satisfied with 
\begin{equation} \label{eta}
\mathbb{E}_0 \left[ \prod_{\omega = 1}^4\left(h\left(\mathbf{R}^{(1
            2)}_{\mathbf{i}^\omega} \right)  \right)\right]  
    = \eta^h.
\end{equation}
We will conclude the proof of \eqref{just_right_order} for $d = 2$ by showing there are 
\begin{equation}\label{countforlockedtuples}
3 \cdot 4! \cdot {n \choose 4k - 4} {4k - 4 \choose 4} {4k - 8 \choose k -2, k -2, k-2, k-2}
= \frac{3n!}{(n - 4k + 4)! ((k -2)!)^4}
\end{equation}
choices of ${\bf i}^1, \dots, {\bf i}^4$ that satisfy \eqref{multiset1} and \eqref{multiset2}, possibly after relabeling of their superscripts. If so, since $ {n \choose 4}^{-4}\frac{3n!}{(n - 4k + 4)! ((k -2)!)^4} = {k \choose 2}^4 \frac{48}{n^4} + O(n^{-5})$, combining \eqref{easylimit} with the summand values \eqref{fourth-moment-zeta1} and \eqref{eta}, we have shown that for $d = 2$, the left hand side of \eqref{just_right_order} equals
\[
 {k \choose 2}^4 \frac{ 3 {(2!)}^2 }{n^{4}}  (\zeta_d^h)^2 + {k \choose 2}^4 \frac{48}{n^4} \eta^h + O(n^{-5}) = {k \choose 2}^4\frac{12}{n^{4}}\left\{(\zeta_2^h)^2+ 
 4\eta^h\right\}  + O(n^{-5}).
\]

It remains to show the count in \eqref{countforlockedtuples}. First,
we count how many such constellations there are \emph{without} any
relabeling of superscripts. Given each of the ${n \choose 4k - 4}$
choice for the set $\cup_{\omega = 1}^4 \mathbf{i}^\omega$, there are
$4! {4k - 4 \choose 4}$ ways of picking the disjoint singleton sets
$({\bf i}^1 \cap {\bf i}^2)$, $({\bf i}^2 \cap {\bf i}^3)$,
$({\bf i}^3 \cap {\bf i}^4)$ and $({\bf i}^4 \cap {\bf i}^1)$. Now
there are ${4k - 8 \choose k-2, k-2, k-2, k-2 }$ ways to partition the
remaining $4k - 8$ elements of the set
$\cup_{\omega = 1}^4 \mathbf{i}^\omega$ into the four sets
${\bf i}^1 \setminus ({\bf i}^2 \cup {\bf i}^4) $,
${\bf i}^2 \setminus ({\bf i}^1 \cup {\bf i}^3) $ ,
${\bf i}^3 \setminus ({\bf i}^2 \cup {\bf i}^4) $ and
${\bf i}^4 \setminus ({\bf i}^1 \cup {\bf i}^3) $. Hence, there are
\[
4 \cdot {n \choose 4k - 4} {4k - 4 \choose 4} {4k - 8 \choose k -2, k -2, k-2, k-2}
\]
choices of ${\bf i}^1, \dots, {\bf i}^4$ that satisfy
\eqref{multiset1} and \eqref{multiset2} without having to relabel
their superscripts. To obtain the factor of $3$ in
\eqref{countforlockedtuples}, we note that the constellation of
${\bf i}^1, \dots, {\bf i}^4$ described by \eqref{multiset1} and
\eqref{multiset2} is such that ${\bf i}^1$ intersects with ${\bf i}^2$
and ${\bf i}^4$. Alternatively, ${\bf i}^1$ can intersect with
${\bf i}^3$ and ${\bf i}^4$, or ${\bf i}^2$ and ${\bf i}^3$, to give a
constellation satisfying \eqref{multiset1} and \eqref{multiset2} after
relabeling of index superscripts.
\end{proof}

\begin{proof}[Proof of \lemref{degen_relate}]
As in the proof of \lemref{moment_lem}, without loss of generality, we
assume $(p, q) = (1,2)$. For any given ${\bf i}, {\bf j} \in
\mathcal{P}(n, 2k)$, 
\begin{multline}\label{hw_expand}
\mathbb{E}_0\left[h^W\left({\bf R}^{(12)}_{\bf i}\right)h^W\left({\bf R}^{(12)}_{\bf j}\right)\right]\\
= {2k \choose k}^{-2} \sum_{\substack{{\bf i}^1 \subset {\bf i} \\ |{\bf i}^1| = k  }}  \sum_{\substack{{\bf j}^1 \subset {\bf j} \\ |{\bf j}^1| = k }} \mathbb{E}_0\left[ h\left({\bf R}^{(12)}_{{\bf i}^1}\right) h\left({\bf R}^{(12)}_{{\bf i}\setminus{{\bf i}^1}}\right) h\left({\bf R}^{(12)}_{{\bf j}^1}\right) h\left({\bf R}^{(12)}_{ {\bf j} \setminus {\bf j}^1}\right)\right].
\end{multline}
Since ${\bf i}^1$, ${\bf i}\setminus {\bf i}^1$, ${\bf j}^1$ and
${\bf j}\setminus {\bf j}^1$ are tuples in $\mathcal{P}(n, k)$, if
$|{\bf i} \cap {\bf j}| <2d$, or equivalently
$|{\bf i} \cup {\bf j}| > 4k - 2d$, by \lemref{tech_lem}$(i)$, all
summands on the right hand side of \eqref{hw_expand} equal zero, and thus
$\mathbb{E}_0\left[h^W\left({\bf R}^{(pq)}_{\bf i}\right)h^W\left({\bf
      R}^{(pq)}_{\bf j}\right)\right] = 0$.

Suppose $|{\bf i} \cap {\bf j}| = 2d$. If ${\bf i}^1, {\bf j}^1 \in \mathcal{P}(n, k)$ are such that ${\bf i}^1 \subset {\bf i}$ and ${\bf j}^1 \subset {\bf j}$, we define ${\bf i}^2  = {\bf i} \setminus {\bf i}^1$ and ${\bf j}^2  = {\bf j} \setminus {\bf j}^1$ to simplify notation. If 
\begin{equation} \label{i1cond1}
|{\bf i}^1 \cap {\bf j}^1 | = d \quad  \text{and} \quad |{\bf i}^2 \cap {\bf j}^2 | = d,
\end{equation}
then the necessary condition in \lemref{tech_lem}$(i)$ is satisfied. Since ${\bf i}^1 \cup {\bf j}^1$ and ${\bf i}^2 \cup {\bf j}^2$ are disjoint, independence gives
\begin{equation} \label{zetasq}
\mathbb{E}_0\left[ h\left({\bf R}^{(12)}_{{\bf i}^1}\right) h\left({\bf R}^{(12)}_{{\bf i}^2}\right) h\left({\bf R}^{(12)}_{{\bf j}^1}\right) h\left({\bf R}^{(12)}_{ {\bf j}^2}\right)\right] = (\zeta_d^h)^2.
\end{equation}
Similarly, if 
\begin{equation} \label{i1cond2}
|{\bf i}^1 \cap {\bf j}^2 | = d \quad  \text{and} \quad |{\bf i}^2 \cap {\bf j}^1 | = d,
\end{equation}
then \eqref{zetasq} holds too.

Now we give the count for how many combinations of ${\bf i}^1$ and
${\bf j}^1$ satisfy \eqref{i1cond1}.  Since
$|{\bf i }\cap {\bf j}| = 2d$, there are ${2d \choose d}$ choices for
the set ${\bf i}^1 \cap {\bf j}^1$, which determines
${\bf i}^2 \cap {\bf j}^2$.  For each such choice, there are then
${2k - 2d \choose k -d}$ choices for each of
${\bf i}^1 \setminus ({\bf i}^1 \cap {\bf j}^1)$ and
${\bf j}^2 \setminus ({\bf i}^2 \cap {\bf j}^2)$, which
determine ${\bf i}^2$ and ${\bf j}^2$. Hence, there are
${2d \choose d}{2k - 2d \choose k -d}^2$ choices of
$({\bf i}^1, {\bf j}^1)$ satisfying \eqref{i1cond1}.  
Analogously, there are also
${2d \choose d}{2k - 2d \choose k -d}^2$ choices of
$({\bf i}^1, {\bf j}^1)$ satisfying \eqref{i1cond2}.  In total, there are
\begin{equation} \label{count1}
2{2d \choose d}{2k - 2d \choose k -d}^2
\end{equation}
summands in \eqref{hw_expand} with the value $(\zeta_d^h)^2$.

If $d = 1$, then no constellations for ${\bf i}^1$ and ${\bf i}^2$
other than the ones given by \eqref{i1cond1} and \eqref{i1cond2} yield
a non-zero value for
$\mathbb{E}_0[ h({\bf R}^{(12)}_{{\bf i}^1}) h({\bf R}^{(12)}_{{\bf
    i}^2}) h({\bf R}^{(12)}_{{\bf j}^1}) h({\bf R}^{(12)}_{ {\bf
    j}^2})]$.
Therefore, we deduce from \eqref{hw_expand} that, for $d = 1$,
 \[
\zeta_{2d}^{h^W} = 2{2d \choose d}{2k - 2d \choose k -d}^2 {2k  \choose k }^{-2} (\zeta_1^h)^2 = 4 {2k - 2 \choose k -1}^2 {2k\choose k}^{-2} (\zeta_1^h)^2.
\]

It remains to prove the formula for $\zeta_{2d}^{h^W}$ when $d =
2$.
In this case, besides \eqref{i1cond1} and \eqref{i1cond2}, there is
one other constellation for ${\bf i}^1$, ${\bf i}^2$, ${\bf j}^1$,
${\bf j}^2$ so that the necessary condition in \lemref{tech_lem}$(i)$
is satisfied. If the multiset
$({\bf i}^1\cup{\bf j}^1\cup {\bf i}^2\cup{\bf j}^2, f_m)$ is such
that
\begin{align} 
&|{\bf i}^1  \cap {\bf j}^1| =
  |{\bf j}^1  \cap {\bf i}^2| =
 |{\bf i}^2  \cap {\bf j}^2| = 
|{\bf j}^2  \cap {\bf i}^1| 
=1 \quad \text{and}
\\
&f_m(i) = 
  \begin{cases}
    2 \quad \text{if }  i \text{ belongs to any one of } 
{\bf i}^1  \cap {\bf j}^1, 
{\bf j}^1  \cap {\bf i}^2, 
{\bf i}^2  \cap {\bf j}^2
\text{ or }
{\bf j}^2  \cap {\bf i}^1,\\
    1 \quad \text{otherwise},
  \end{cases}
\end{align} 
then
\begin{equation} 
\mathbb{E}_0\left[ h\left({\bf R}^{(12)}_{{\bf i}^1}\right) h\left({\bf R}^{(12)}_{{\bf i}^2}\right) h\left({\bf R}^{(12)}_{{\bf j}^1}\right) h\left({\bf R}^{(12)}_{ {\bf j}^2}\right)\right] = \eta^h.
\end{equation}
Now we count: For a fixed pair $({\bf i}, {\bf j})$ such that
$|{\bf i} \cap {\bf j}| = 4$, there are $4!$ choices for the
singletons ${\bf i}^1 \cap {\bf j}^1$, ${\bf j}^1 \cap {\bf i}^2$ ,
${\bf i}^2 \cap {\bf j}^2$ and ${\bf j}^2 \cap {\bf i}^1$. Given each
such choice for these singletons, there are ${2k -4 \choose k - 2}$
choices for each one of ${\bf i}^1$ and ${\bf j}^1$, hence there are
\[
4! {2k -4 \choose k - 2}^2
\]
summands on the right hand side of \eqref{hw_expand} with the value
$\eta^h$. Combining with the count \eqref{count1} for summands with
the value $(\zeta_d^h)^2$, we conclude that if $d=2$ then
\begin{align*}
\zeta^{h^W}_{2d} &= {2k \choose k}^{-2} \left\{2{2d \choose d}{2k - 2d \choose k -d}^2 (\zeta_d^h)^2 + 4! {2k -4 \choose k - 2}^2 \eta^h \right\}\\
&= 12 {2k -4 \choose k - 2}^2 {2k \choose k}^{-2} \left[ (\zeta_2^h)^2
  + 2 \eta^h\right].
\qedhere
\end{align*}
\end{proof}

\section{Proofs for \secref{asympt-null-distr}} \label{sec:PfSec4}

Here, we prove \lemsref{ell2-convergence} and \lemssref{lyapunov} that
were used in the proof of \thmref{main}.

\begin{lemma}
  \label{lem:ell2-convergence}
  The martingale differences from~(\ref{eq:mtg-diffs}) satisfy
  the $L^2$ convergences
  \begin{align}
    \label{L2cond1}
    \mathbb{E}_0\left[\left(  \frac{n^2}{m^2}\sum_{l =
          2}^m\mathbb{E}_0\left[(D^S_{nl})^2\middle|\mathcal{F}_{n, l
            -1}\right] -     k^4  ( \zeta^h_1)^2\right)^2\right] &\longrightarrow 0 \;, \\
    \label{L2cond2}
    \mathbb{E}_0\left[\left( \frac{n^2}{m^2} \sum_{l =
          2}^m\mathbb{E}_0\left[(D^T_{nl})^2\middle|\mathcal{F}_{n, l
            -1}\right] -     k^4  (\zeta^h_1)^2\right)^2
    \right] &\longrightarrow 0\;,\\
\label{L2cond5}
    \mathbb{E}_0\left[\left(  \frac{n}{  m^2} \sum_{l = 2}^m \mathbb{E}_0\left[(D^Z_{nl})^2\middle| \mathcal{F}_{n, l
        -1}\right]  - \frac{k^2 \zeta_1^h}{2}\right)^2
    \right]    &\longrightarrow 0 \quad   
  \end{align} 
when $d =1 $, and the $L^2$ convergences
  \begin{align}
    \label{L2cond3}
    \mathbb{E}_0\left[\left(   \frac{n^4}{m^2} \sum_{l =
          2}^m \mathbb{E}_0\left[(D^S_{nl})^2\middle|\mathcal{F}_{n, l
            -1}\right] -     4 {k \choose 2}^4 \left\{ (\zeta^h_2)^2 + 6 \eta^h  \right\}\right)^2\right] &\longrightarrow 0 \;, \\
    \label{L2cond4}
    \mathbb{E}_0\left[\left(  \frac{n^4}{m^2} \sum_{l =
          2}^m \mathbb{E}_0\left[(D^T_{nl})^2\middle|\mathcal{F}_{n, l
            -1}\right] -     4 {k \choose 2}^4 \left\{
    (\zeta^h_2)^2 + 2 \eta^h  \right\}\right)^2   \right]         &\longrightarrow 0 \;,\\
\label{L2cond6}
\mathbb{E}_0\left[\left( \frac{n^2}{m^2}\sum_{l = 2}^m \mathbb{E}_0\left[(D^Z_{nl})^2\middle| \mathcal{F}_{n, l
        -1}\right]  -    {k \choose 2}^2 \zeta^h_2 \right)^2   \right]  &\longrightarrow 0
  \end{align} 
when $d = 2$.
\end{lemma}

\begin{proof}
When $d =1$, for the $L^2$ convergences in \eqref{L2cond1}, \eqref{L2cond2} and \eqref{L2cond5}, it is sufficient to show that, as $m, n \longrightarrow \infty$,

\begin{equation}
\label{expconv1}
\begin{split}
 &\frac{n^2}{m^2}\sum_{l = 2}^m\mathbb{E}_0[(D^S_{nl})^2], \ \ \frac{n^2}{m^2} \sum_{l =2}^m\mathbb{E}_0[(D^T_{nl})^2]  \longrightarrow  k^4 (\zeta_1^h)^2, \\
&\frac{n}{  m^2} \sum_{l = 2}^m \mathbb{E}_0\left[(D^Z_{nl})^2\right]  \longrightarrow  \frac{k^2 \zeta^h_1}{2},
\qquad and
\end{split}
\end{equation}

\begin{equation}
\begin{aligned}\label{varconv1}
  &\text{Var}_0\left[ \frac{n^2}{m^2} \sum_{l =
          2}^m\mathbb{E}_0[(D^S_{nl})^2|\mathcal{F}_{n, l
    -1}]\right] \longrightarrow 0, \\
&\text{Var}_0\left[ \frac{n^2}{m^2}\sum_{l =
          2}^m\mathbb{E}_0[(D^T_{nl})^2|\mathcal{F}_{n, l
    -1}]\right] \longrightarrow 0, \\
& \text{Var}_0\left[ \frac{n}{m^2}\sum_{l =
          2}^m\mathbb{E}_0[(D^Z_{nl})^2|\mathcal{F}_{n, l
    -1}]\right] \longrightarrow 0.\\
\end{aligned}
\end{equation}
When $d = 2$, for the $L^2$ convergences in \eqref{L2cond3} and \eqref{L2cond4}, it suffices to show that, as $m, n \longrightarrow \infty$, 
  \begin{align}
    \label{expconv2a}
\begin{split}
 &\sum_{l =
          2}^m \frac{n^4}{m^2}\mathbb{E}_0[(D^S_{nl})^2] 
    \; \longrightarrow \;   4 {k \choose 2}^4 \left\{ (\zeta^h_2)^2 + 6 \eta^h  \right\},\\
 &\sum_{l =
          2}^m \frac{n^4}{m^2}\mathbb{E}_0[(D^T_{nl})^2] 
    \; \longrightarrow \; 4 {k \choose 2}^4 \left\{ (\zeta^h_2)^2 + 2 \eta^h  \right\}, \\
&\sum_{l = 2}^m \frac{n^2}{m^2}\mathbb{E}_0\left[(D^Z_{nl})^2\right]  \longrightarrow  
 {k \choose 2}^2 \zeta^h_2 ,
\qquad and
\end{split}
\end{align}

\begin{equation}
\begin{aligned} \label{varconv2}
  &\text{Var}_0\left[\frac{n^4}{m^2}\sum_{l =
          2}^m \mathbb{E}_0 [(D^S_{nl})^2|\mathcal{F}_{n, l
    -1}]\right]  \longrightarrow 0,\\
&\text{Var}_0\left[\frac{n^4}{m^2}\sum_{l =
          2}^m \mathbb{E}_0[(D^T_{nl})^2|\mathcal{F}_{n, l
    -1}]\right]  \longrightarrow 0,\\
& \text{Var}_0\left[ \frac{n^2}{m^2}\sum_{l =
          2}^m\mathbb{E}_0[(D^Z_{nl})^2|\mathcal{F}_{n, l
    -1}]\right] \longrightarrow 0.\\
\end{aligned}
\end{equation}
%
We will first show the convergences of expectations in
\eqref{expconv1} and \eqref{expconv2a}. Suppose $d
= 1$ or $2$ is the order of degeneracy of $h$ under $H_0$.   By \lemref{general_properties}(i)
  and (iii), the terms $\bar U^{(pl)}_h$ that are summed to form
  $D_{nl}^S$ are i.i.d.~such that
  \begin{equation*} 
  \frac{n^{2d }}{m^2}  \mathbb{E}_0[(D_{nl}^S)^2] \;=\;   
     \frac{n^{2d }}{m^2}\sum_{p=1}^{l -1}
    \mathrm{Var}_0\left[\bar{U}^{(pl)}_h\right] 
    \;=\; \frac{n^{2d }}{m^2} (l-1)\mathrm{Var}_0\left[\bar{U}^{(12)}_h\right].
  \end{equation*}
 It follows that
\begin{equation}  \label{exp_diff_sq_S}
  \frac{n^{2d  }}{m^2}\sum_{l = 2}^m \mathbb{E}_0[(D_{nl}^S)^2] =  \frac{(m-1)n^{2d  }}{2 m }\mathrm{Var}_0\left[\bar{U}^{(12)}_h\right]
  .
  \end{equation}
Similarly, by \lemref{general_properties}(i)
  and (iii), we have that
\begin{align}  \label{exp_diff_sq_T}
 \frac{n^{2d }}{m^2}\sum_{l = 2}^m \mathbb{E}_0[(D_{nl}^T)^2] &= \frac{(m-1)n^{2d }}{2m }\mathrm{Var}_0\left[W^{(12)}_h\right]
\quad \text{and}\\
\label{exp_diff_U}
\frac{n^d}{m^2} \sum_{l = 2}^m \mathbb{E}_0\left[(D^Z_{nl})^2\right]&= \frac{(m-1)n^{d }}{2m }\mathrm{Var}_0\left[U^{(12)}_h\right].
  \end{align}
By \lemref{moment_lem}$(i)$ and $(iii)$,
   \begin{align}
\mathrm{Var}_0\left[\bar{U}^{(12)}_h\right] &=  
\mathbb{E}_0\left[\left(
    U^{(12)}_h  \right)^4 \right]- \mu_h^2 \notag \\
&=
  \begin{cases}
    \frac{2k^{4}(\zeta^h_1)^2}{n^2} + O\left(n^{-3}\right) &\text{if} \quad d = 1,\\
      \frac{8}{n^{4}}  {k \choose 2}^4\{(\zeta_2^h)^2+ 
6\eta^h \}+
 O\left(n^{-5}\right) &\text{if} \quad d = 2.
  \end{cases}
 \label{VarUbar}
   \end{align}
Since $W^{(12)}_h$ is a rank-based U-statistic with the induced kernel function $h^W$ of degree $2k$, via \lemref{degen_relate}, \lemref{moment_lem}$(i)$ applies to give 
\begin{align}
 \mathrm{Var}_0\left[W^{(12)}_h\right] &= \mathbb{E}_0\left[\left(
    W^{(12)}_h  \right)^2 \right] = {2k \choose 2d}^2 \frac{(2d)!}{n^{2d}} \zeta_{2d}^{h^W} + O(n^{-2d- 1}) \notag\\
&= 
  \begin{cases}
     \frac{2 k^4 (\zeta_1^h)^2}{n^2}+ O\left(n^{-3}\right) &\text{if} \quad d = 1,\\
 \frac{8}{n^4} {k \choose 2}^4\{(\zeta^h_2)^2 + 2 \eta^h\} + O(n^{-5}) &\text{if} \quad d = 2.
  \end{cases}
                                                                         \label{VarW}
\end{align}
Moreover, \lemref{moment_lem}$(i)$ yields that
\begin{equation} \label{VarU}
\text{Var}_0 \left[ \left(U^{(12)}_h\right)^2\right] = \mathbb{E}_0 \left[ \left(U^{(12)}_h\right)^2\right]= 
  \begin{cases}
     \frac{k^2 \zeta_1^h}{n}+ O\left(n^{-2}\right) &\text{if} \quad d = 1,\\
 \frac{2\zeta_2^h}{n^2} {k \choose 2}^2 + O(n^{-3}) &\text{if} \quad d = 2.
  \end{cases}
\end{equation}
Plugging \eqref{VarUbar}, \eqref{VarW} and \eqref{VarU} into \eqref{exp_diff_sq_S}
, \eqref{exp_diff_sq_T} and \eqref{exp_diff_U} for $d = 1$ and $d =2$, respectively, and
taking the limit, we obtain the convergences in \eqref{expconv1} and \eqref{expconv2a}.

Next, we show that the variances in \eqref{varconv1} and
\eqref{varconv2} converges to zero. For $d\in\{1,2\}$, write
  \begin{align*}
   &\frac{n^{2d}}{m^2} \sum_{l = 2}^m\mathbb{E}_0\left[(D^S)^2_{nl}|\mathcal{F}_{n, l -1}\right]\\
    &= \frac{n^{2d}}{m^2}  \left \{\sum_{l=2}^m\sum_{p=1}^{l -1}
      \mathbb{E}_0\left[\left(\bar U^{(pl)}_h\right)^2 \middle|
      \mathcal{F}_{n, l-1}\right] +
      2\sum_{l=3}^m\sum_{1\le p<q< l}\mathbb{E}_0\left[
      \bar U^{(pl)}_h\bar U^{(ql)}_h \middle| \mathcal{F}_{n,
      l-1}\right] \right\}, 
  \end{align*}
  and notice that the first sum on the right-hand side is a constant
  because, by  \lemref{general_properties}(ii),
  \[
  \mathbb{E}_0\left[\left(\bar{U}^{(pl)}_h\right)^2 \middle|
    \mathcal{F}_{n, l-1}\right] =
  \mathbb{E}_0\left[\left(\bar{U}^{(pl)}_h\right)^2 \middle|
    \mathbf{X}^{(p)}\right] = \mathbb{E}_0 \left[
    \left(\bar{U}^{(pl)}_h\right)^2\right]. 
  \] 
  We observe that in order to show
  $\text{Var}_0\left[ \frac{n^{2d}}{m^2} \sum_{l =
      2}^m\mathbb{E}_0[(D^S_{nl})^2|\mathcal{F}_{n, l -1}]\right]
  \longrightarrow 0$, it suffices to show
  \begin{equation}
    \label{eq:lemB1-var-to-show}
    \frac{n^{4d }}{m^4}\mathrm{Var}_0\left[\sum_{l=3}^m\sum_{1\le p<q<l}\mathbb{E}_0\left[
      \bar{U}^{(pl)}_h\bar{U}^{(ql)}_h \middle| \mathcal{F}_{n,
      l-1}\right] \right]\;\longrightarrow\; 0.
  \end{equation}
By exactly analogous arguments, it suffices to show 
  \begin{align}
    \label{eq:lemB2-var-to-show}
    \frac{n^{4d }}{m^4}\mathrm{Var}_0\left[\sum_{l=3}^m\sum_{1\le p<q<l}\mathbb{E}_0\left[
      W^{(pl)}_h W^{(ql)}_h \middle| \mathcal{F}_{n,
      l-1}\right] \right]\;&\longrightarrow\; 0 \quad \text{ and } \\
    \label{eq:lemB3-var-to-show}
    \frac{n^{2d }}{m^4}\mathrm{Var}_0\left[\sum_{l=3}^m\sum_{1\le p<q<l}\mathbb{E}_0\left[
      U^{(pl)}_h U^{(ql)}_h \middle| \mathcal{F}_{n,
      l-1}\right] \right]\;&\longrightarrow\; 0 
  \end{align}
in order to prove $\text{Var}_0\left[ \frac{n^{2d}}{m^2}\sum_{l =
          2}^m\mathbb{E}_0[(D^S_{nl})^2|\mathcal{F}_{n, l
    -1}]\right], \ \text{Var}_0\left[ \frac{n^{d}}{m^2}\sum_{l =
          2}^m\mathbb{E}_0[(D^Z_{nl})^2|\mathcal{F}_{n, l
    -1}]\right] \longrightarrow 0$.
  
  We first prove~(\ref{eq:lemB1-var-to-show}). For $p < q <l$, consider 
  \[
  C^{(pq)} := \mathbb{E}_0\left[ \bar{U}^{(pl)}_h\bar{U}^{(ql)}_h
    \middle|\ \mathcal{F}_{n, l-1}\right] \; =\; \mathbb{E}_0\left[
    \bar{U}^{(pl)}_h\bar{U}^{(ql)}_h \middle| \mathbf{X}^{(p)},
    \mathbf{X}^{(q)}\right],
  \]
  which is a function of $\mathbf{X}^{(p)}$ and $\mathbf{X}^{(q)}$
  alone.  Since
\[
\bar{U}^{(pl)}_h\bar{U}^{(ql)}_h = f({\bf R}^{(p l)}_1, \dots, {\bf R}^{(p l)}_k) f({\bf R}^{(q l)}_1, \dots, {\bf R}^{(q l)}_k)
\]
for a function $f : (\mathbb{R}^2)^k \longrightarrow \mathbb{R}$ that
is permutation symmetric in its $k$ arguments, and since the rank
vectors ${\bf R}^{(p)}$, ${\bf R}^{(q)}$, ${\bf R}^{(l)}$ are
independent and uniformly distributed on $\mathfrak{S}_n$ under $H_0$,
the conditional expectation $C^{(pq)}$ is in fact a function of the
tuple $({\bf R}_1^{(pq)}, \dots, {\bf R}_n^{(pq)})$ that is symmetric
in its $n$ arguments.  Therefore, Lemma~\ref{lem:general_properties}
applies to the collection of $C^{(pq)}$, $1\le p\not= q\le m$.  The
variance in~(\ref{eq:lemB1-var-to-show}) is thus
  \begin{align*}
    \label{eq:lemB1-var-sum-of-vars}
    \mathrm{Var}_0\left[\sum_{l=3}^m\sum_{1\le p<q<l} C^{(pq)} \right] &=
   \sum_{1 \leq p < q \leq m-1} (m - q)^2 \mathrm{Var}_0\left[C^{(pq)}\right]\\
&=  \frac{1}{12} m(m-2)(m-1)^2 \mathrm{Var}_0 \left[ C^{(12)} \right].
  \end{align*}
Now under the asymptotic regime $m, n \longrightarrow \infty$,  (\ref{eq:lemB1-var-to-show}) holds if $\mathrm{Var}_0 \left[ C^{(12)} \right]$ is of order $O(n^{-4d -1})$. 
  
   Suppose $2 < l <u \leq m$.  Then, by definition, 
\[C^{(12)} =
\mathbb{E}_0\left[
    \bar{U}^{(1l)}_h\bar{U}^{(2l)}_h \middle| \mathbf{X}^{(1)},
    \mathbf{X}^{(2)}\right] 
=
\mathbb{E}_0\left[
    \bar{U}^{(1u)}_h\bar{U}^{(2u)}_h \middle| \mathbf{X}^{(1)},
    \mathbf{X}^{(2)}\right].
\]
 It follows that 
\begin{align}
 \mathbb{E}_0\left[ \bar{U}^{(1l)}_h\bar{U}^{(2l)}_h \bar{U}^{(1u)}_h\bar{U}^{(2u)}_h\right] &=  \mathbb{E}_0\left[\mathbb{E}_0\left[ \bar{U}^{(1l)}_h\bar{U}^{(2l)}_h \bar{U}^{(1u)}_h\bar{U}^{(2u)}_h\middle|\mathbf{X}^{(1)},
    \mathbf{X}^{(2)} \right] \right] \notag\\
&= \mathbb{E}_0\left[\mathbb{E}_0\left[
    \bar{U}^{(1l)}_h\bar{U}^{(2l)}_h \middle| \mathbf{X}^{(1)},
    \mathbf{X}^{(2)}\right] \mathbb{E}_0\left[
    \bar{U}^{(1u)}_h\bar{U}^{(2u)}_h \middle| \mathbf{X}^{(1)},
    \mathbf{X}^{(2)}\right]\right]  \label{eq:indep_magic}\\
&= \mathbb{E}_0 \left[\left( C^{(12)} \right)^2 \right] \notag,
\end{align}
where ~(\ref{eq:indep_magic}) follows from independence of ${\bf X}^{(l)}$ and ${\bf X}^{(u)}$. Applying \lemref{4prod_lem1}, we deduce that $\mathbb{E}_0[(C^{(12)})^2]$ is of order $O(n^{-4d -1})$.  This concludes the proof as an application of \lemref{general_properties}$(iii)$ shows that $C^{(12)}$ has mean zero, and thus $\mathrm{Var}_0[C^{(12)}]=\mathbb{E}_0[(C^{(12)})^2]$.

The proof of \eqref{eq:lemB2-var-to-show} and $\eqref{eq:lemB3-var-to-show} $ proceeds line by line as the proof of \eqref{eq:lemB1-var-to-show}, where for all $1 \leq p \not = q \leq m$ we replace $\bar{U}^{(pq)}_h$ by $W^{(pq)}_h$ or $U^{(pq)}_h$, define $C^{(pq)}$ alternatively as 
\[
C^{(pq)} := \mathbb{E}_0\left[ W^{(pl)}_hW^{(ql)}_h
    \middle|\ \mathcal{F}_{n, l-1}\right] \quad \text{or} \quad  C^{(pq)} := \mathbb{E}_0\left[ U^{(pl)}_hU^{(ql)}_h
    \middle|\ \mathcal{F}_{n, l-1}\right]
,\]
and apply \lemref{4prod_lem2} or  \lemref{4prod_lem3}.   We
omit the details. 
\end{proof}

\begin{lemma}
  \label{lem:lyapunov}
  For $d = 1$ or $2$, the martingale differences from~(\ref{eq:mtg-diffs}) satisfy
  the Lyapunov conditions 
  \begin{equation}
    \label{cond2'}
   \frac{n^{4d }}{m^4} \sum_{l = 2}^m \mathbb{E}_0\left[(D^S_{nl})^4 \middle| \mathcal{F}_{n,
        l- 1}\right] 
, \quad 
   \frac{n^{4d }}{m^4}\sum_{l = 2}^m \mathbb{E}_0\left[(D^T_{nl})^4 \middle| \mathcal{F}_{n,
        l- 1}\right] \underset{p}{\longrightarrow } 0 \quad\text{ and}\end{equation}
\begin{equation}
 \frac{n^{2d }}{m^4}\sum_{l = 2}^m \mathbb{E}_0\left[(D^Z_{nl})^4 \middle| \mathcal{F}_{n,
        l- 1}\right] \underset{p}{\longrightarrow } 0
  \end{equation}
as $m, n \longrightarrow \infty$.
\end{lemma}

  \begin{proof}
    Since
    $\sum_{l = 2}^m \mathbb{E}_0[(D^S_{nl})^4 | \mathcal{F}_{n, l-
        1}] $, $\sum_{l = 2}^m \mathbb{E}_0[(D^T_{nl})^4 | \mathcal{F}_{n, l-
        1}] $ and $\sum_{l = 2}^m \mathbb{E}_0\left[(D^Z_{nl})^4 \middle| \mathcal{F}_{n,
        l- 1}\right]$
    are nonnegative random variables, it suffices to show that all three
    expectations converge to zero, that is,
\[
\frac{n^{4d }}{m^4}  \sum_{l = 2}^m \mathbb{E}_0\left[(D^S_{nl})^4 \right], \quad  \frac{n^{4d }}{m^4}\sum_{l = 2}^m \mathbb{E}_0\left[(D^T_{nl})^4 \right], \quad \frac{n^{2d }}{m^4}\sum_{l = 2}^m \mathbb{E}_0\left[(D^Z_{nl})^4 \middle| \mathcal{F}_{n,
        l- 1}\right] \longrightarrow 0.
\]

%

    We first show it for $\frac{n^{4d }}{m^4}  \sum_{l = 2}^m \mathbb{E}_0\left[(D^S_{nl})^4 \right]$. By
    \lemref{general_properties}$(i)$ and $(iii)$, $D_{nl}^S$ is a sum
    of $l-1$ centered i.i.d.~random variables.  On expansion, we
    have that
    \begin{align*}
    \mathbb{E}_0\left[\left(D^S_{nl}\right)^4 \right] 
      &= \sum_{p=1}^{l-1}\mathbb{E}_0\left[\left(
        \bar{U}^{(pl)}_h\right)^4\right] +
        6\sum_{1\le p<q<l} \mathbb{E}_0\left[\left(
        \bar{U}^{(pl)}_h\right)^2\right]\mathbb{E}_0\left[\left(
        \bar{U}^{(ql)}_h\right)^2\right]\\
      &= (l-1) \mathbb{E}_0\left[\left(
        \bar{U}^{(12)}_h\right)^4\right] + 6\binom{l-1}{2}\left(\mathrm{Var}_0\left[
        \bar{U}^{(12)}_h\right]\right)^2 .
    \end{align*}
    It follows that
    \begin{equation}
      \label{lemB2-expansion1}
    \frac{n^{4d }}{m^4} \sum_{l = 2}^m\mathbb{E}_0\left[\left(D^S_{nl}\right)^4 \right]  \;=\;
     \frac{n^{4d }}{m^4}\left\{ \binom{m}{2}\mathbb{E}_0\left[\left(
        \bar{U}^{(12)}_h\right)^4\right] + 6 \binom{m}{3}\left(\mathrm{Var}_0\left[
        \bar{U}^{(12)}_h\right]\right)^2 \right\} . 
    \end{equation}
    Now recall from \eqref{VarUbar} that the variance
    of $\bar{U}^{(12)}_h$ is of order $O(n^{-2d})$.  Furthermore,
    \begin{align*}
      \mathbb{E}_0\left[\left(
          \bar{U}^{(12)}_h\right)^4\right] &=
      \mathbb{E}_0\left[\left(\big(U^{(12)}_h\big)^2- \mu_h
        \right)^4\right] \\ 
      &=  \mathbb{E}_0\left[\left(U^{(12)}_h\right)^8 - 4 \mu_h
        \left(U^{(12)}_h\right)^6 + 6 \mu_h^2\left(U^{(12)}_h\right)^4 -
        4 \mu_h^3 \left(U^{(12)}_h\right)^2 + \mu_h^4\right]  
    \end{align*}
    is of order $O(n^{-4d})$ by \lemref{moment_lem}$(ii)$.  Substituting these into \eqref{lemB2-expansion1} we conclude that \[\frac{n^{4d }}{m^4}\sum_{l = 2}^m \mathbb{E}_0\left[(D^S_{nl})^4 \right] = O(m^{-1})\longrightarrow  0 \quad \text{as} \quad m,n\longrightarrow \infty. \]

The proof for $\frac{n^{4d }}{m^4}  \sum_{l = 2}^m \mathbb{E}_0\left[(D^T_{nl})^4 \right]$ and $\frac{n^{2d }}{m^4}  \sum_{l = 2}^m \mathbb{E}_0\left[(D^Z_{nl})^4 \right]$ is similar. On expansion, we have 
\begin{align}
     \frac{n^{4d }}{m^4} \sum_{l = 2}^m\mathbb{E}_0\left[\left(D^T_{nl}\right)^4 \right]  \;&=\;
   \frac{n^{4d }}{m^4}\left\{{m \choose 2} \mathbb{E}_0\left[\left(
       W^{(12)}_h\right)^4\right] + 6{m \choose 3}\left(\mathbb{E}_0\left[
        \left(W^{(12)}_h\right)^2\right]\right)^2 \right\}  \label{lemB2-expansion2},\\
     \frac{n^{2d }}{m^4} \sum_{l = 2}^m\mathbb{E}_0\left[\left(D^Z_{nl}\right)^4 \right]  \;&=\;
   \frac{n^{2d }}{m^4}\left\{{m \choose 2} \mathbb{E}_0\left[\left(
       U^{(12)}_h\right)^4\right] + 6{m \choose 3}\left(\mathbb{E}_0\left[
        \left(U^{(12)}_h\right)^2\right]\right)^2 \right\} \label{lemB2-expansion3}
\end{align}
by \lemref{general_properties}$(i)$ and $(iii)$. By \lemsref{moment_lem}$(ii)$ and \lemssref{degen_relate}, since $h^W$ has order of degeneracy $2d$, $\mathbb{E}_0\big[(
       W^{(12)}_h)^4\big] $ and $\mathbb{E}_0\big[(
       W^{(12)}_h)^2\big]$ are of order $O(n^{-4d})$ and $O(n^{-2d})$ respectively. Another application of  \lemref{moment_lem}$(ii)$
gives that $\mathbb{E}_0\big[(
       U^{(12)}_h)^4\big]  = O(n^{-2d})$ and $\mathbb{E}_0\big[(
       U^{(12)}_h)^2\big] = O(n^{-d})$. On substituting these into \eqref{lemB2-expansion2} and \eqref{lemB2-expansion3} we get that
both  $\frac{n^{4d }}{m^4} \sum_{l = 2}^m\mathbb{E}_0\left[\left(D^T_{nl}\right)^4 \right]$ and  $\frac{n^{2d }}{m^4} \sum_{l = 2}^m\mathbb{E}_0\left[\left(D^Z_{nl}\right)^4 \right]$ are of order $O(m^{-1})$ and converge to $0$  as $m, n \longrightarrow \infty$.
 \end{proof}

\begin{proof}[Proof of  Corollary~\ref{spearnull}]
  It suffices to show that $\frac{n}{m}(S_\rho - S_{\hat{\rho}}) = o_p(1)$, in
  which case the corollary is implied by the fact that
  $\frac{n}{m}S_{\hat{\rho}} \longrightarrow N(0, 1)$ as given in
  \thmref{main} and the value of $\zeta_1^{h_{\hat{\rho}_s}}$ in \tabref{degen_tab}. By the decomposition in \eqref{rho_decomp}, the
  statistic $S_\rho$ from \eqref{rho_test_stat} may be written as
  \begin{align*}
    S_{\rho}&=  \sum_{1\leq p < q \leq m} \left( \frac{n - 2}{n+1 }  \hat{\rho}^{(pq)}  + \frac{3}{n+1} \tau^{(pq)}\right)^2 - {m \choose 2}\mu_{\rho^2}.
  \end{align*}
  Expanding the square in the summands on the right-hand side, we
  obtain that
  \begin{multline} \label{expan_S_rho}
    S_{\rho}=\left(\frac{n - 2}{n+1}\right)^2 S_{\hat{\rho}} +
    \frac{9}{(n+1)^2} S_{\tau}
    +\frac{6 (n -2)}{(n+1)^2} \sum_{1 \leq p < q \leq
      m}\hat{\rho}^{(pq)} \tau^{(pq)} \\
    + {m \choose 2} \left[\left(
        \frac{n -2}{n+1}\right)^2   \mu_{\hat{\rho}^2}
      +\frac{9}{(n+1)^2} \mu_{\tau^2} - \mu_{\rho^2} \right]; 
  \end{multline}
  recall the definition of $S_{\tau}$ and $S_{\hat{\rho}}$.  Note that since
  $S_\rho$, $S_\tau$ and $S_{\hat{\rho}}$ have mean zero, it holds
  that
  \[
  \mu_{\hat{\rho} \tau} := \mathbb{E}_0\left[ \hat{\rho}^{(pq)} \tau^{(pq)}\right]
  \;=\; \frac{(n+1)^2}{6(n-2)} \left[\mu_{\rho^2} - \left( \frac{n
        -2}{n+1}\right)^2   \mu_{\hat{\rho}^2} -\frac{9}{(n+1)^2}
    \mu_{\tau^2}\right], 
  \]
and hence \eqref{expan_S_rho} can be rewritten as 
\[
   S_{\rho}=\left(\frac{n - 2}{n+1}\right)^2 S_{\hat{\rho}} +
    \frac{9}{(n+1)^2} S_{\tau}
    +  \frac{6 (n -2)}{(n+1)^2} \left[  \sum_{1 \leq p < q \leq
      m}\hat{\rho}^{(pq)} \tau^{(pq)}  - {m \choose 2} \mu_{\hat{\rho}
      \tau} \right] .
\]
  Since $\frac{9n}{m(n+1)^2}S_\tau = o_p(1)$ by \thmref{main}, in order to prove the assertion that $\frac{n}{m}(S_\rho-S_{\hat{\rho}})=o_p(1)$,
  it thus suffices to show that
  \[ 
  \frac{6n (n -2)}{m(n+1)^2} \left[  \sum_{1 \leq p < q \leq
      m}\hat{\rho}^{(pq)} \tau^{(pq)}  - {m \choose 2} \mu_{\hat{\rho}
      \tau} \right] \underset{p}{\longrightarrow} 0. 
  \] 
  We show this by proving convergence to zero in $L^2$, for which
  we need to argue that
  \begin{equation}
    \label{cross_rho_hat_tau}
    \frac{36 n^2(n -2)^2}{m^2(n+1)^4} \mathbb{E}_0 \left[\left\{  \sum_{1 \leq
          p < q \leq m}\hat{\rho}^{(pq)} \tau^{(pq)}  - {m \choose 2}
        \mu_{\hat{\rho} \tau} \right\}^2 \right] \;\longrightarrow\; 0.
  \end{equation}

  Note that \lemref{general_properties} applies to the collection of
  statistics $\hat{\rho}^{(pq)} \tau^{(pq)}$.  By
  \lemref{general_properties}$(i)$ and $(iv)$, the term
  in~(\ref{cross_rho_hat_tau}) equals
  \begin{equation}  
    \frac{18 n^2(n -2)^2(m-1)}{(n+1)^4m} \left\{  \mathbb{E}_0
      \left[ \left( \hat{\rho}^{(12)} \tau^{(12)}\right)^2\right] -
      \mu_{\hat{\rho}
        \tau}^2\right\} \label{cross_l2_norm}. 
  \end{equation}
  Since $\frac{18 n^2(n -2)^2(m-1)}{(n+1)^4m} = O(1)$ as $m, n \longrightarrow \infty$, for the convergence from~(\ref{cross_rho_hat_tau})
  it remains to show that
  \begin{equation*}
    \mathrm{Var}_0\left[ \hat{\rho}^{(12)} \tau^{(12)} \right] =
    \mathbb{E}_0 \left[ \left( \hat{\rho}^{(12)}
        \tau^{(12)}\right)^2\right] - \mu_{\hat{\rho} \tau}^2 
    \;\longrightarrow\; 0.
  \end{equation*} 
  However, using the inequality $2xy\le (x^2+y^2)$,
  we see that 
  \[
  0 \leq \mathrm{Var}_0\left[ \hat{\rho}^{(12)} \tau^{(12)} \right]
  \leq \mathbb{E}_0 \left[ \left( \hat{\rho}^{(12)}
      \tau^{(12)}\right)^2\right] \le
  \frac{1}{2} \mathbb{E}_0 \left[ \left( \hat{\rho}^{(12)}\right)^4\right] +\frac{1}{2} \mathbb{E}_0 \left[ \left( 
      \tau^{(12)}\right)^4\right],
  \]
  which is of order $O(n^{-2})$ by \lemref{moment_lem}$(ii)$.
\end{proof}

\section{Proofs for \secref{power}}
\label{sec:proofs-sec5}

Unlike in other sections, here all the rank-based U-statistics will be
treated as functions of the original data
${\bf X}^{(1)}, \dots, {\bf X}^{(m)}$ in our presentation.

\begin{proof}[Proof of \thmref{beatMax}]
  In this proof, all operators $\mathbb{E}[\cdot]$,
  $\mathrm{Cov}[\cdot]$, $\mathrm{Var}(\cdot)$, $P(\cdot)$ are with
  respect to a general distribution in $\mathcal{D}_m$.

$(i)$: Let ${\bf U}_\tau$ be the ${m \choose 2}$-vector $(
U^{(pq)}_{h_\tau})_{1 \leq p< q \leq m}$. Then ${\bf U}_\tau$ is a
U-statistic taking values in $\mathbb{R}^{m \choose 2}$, with the ${m
  \choose 2}$-dimensional vector-valued kernel 
\[
{\bf h}_\tau ({\bf X}_i, {\bf X}_j)= \left(h_\tau({\bf
    X}^{(pq)}_i,{\bf X}^{(pq)}_j)\right)_{1 \leq p < q \leq m}
\]
of degree $k = 2$.  Here, $i \not = j$ index any pair of
samples.  Note that
${S}_\tau = \|{\bf U}_\tau\|_2^2 - {m \choose 2} \mu_{h_\tau}$, and
based on \thmref{main} we have
under the regime $\frac{m}{n} \longrightarrow \gamma$ that the test $\phi_\alpha( S_\tau)$ rejects $H_0$ when
$\|{\bf U}_\tau\|_2 \geq \sqrt{{m \choose 2} \mu_{\tau} +
  \frac{4m}{9n} z_{1 - \alpha}} = O(\sqrt{n})$.  
Recall that $\mu_\tau = \frac{2(2n +5)}{9 n(n-1)}$ and that the value
of $\zeta^{h_\tau}_1$ is given in \tabref{degen_tab}. By the triangle
inequality
\begin{equation*} \label{someTri}
\|{\bf U}_\tau\|_2 \geq \|  \Theta_\tau \|_2 - \|{\bf U}_\tau - \Theta_\tau\|_2,
\end{equation*}
it suffices to show that as $n \longrightarrow \infty$, uniformly over $\mathcal{D}_m$,
\[
P(\|{\bf U}_\tau - \Theta_\tau\|_2 \geq C\sqrt{n}) \leq 1 - \beta
\]
for some constant $C >0$ that only depends on $\beta$ and $\gamma$.
For any pair $i \not = j$, let
$
{\bf h}_{\tau, 1} ({\bf X}_i) = \mathbb{E}[{\bf h}_\tau ({\bf X}_i, {\bf X}_j)| {\bf X}_i]
$ and define the canonical functions \citep[p.8]{UstatBanach}
\begin{align}
{\bf g}_1 ({\bf X}_i)&:=   {\bf h}_{\tau, 1} ({\bf X}_i)- \Theta , \\{\bf g}_2 ({\bf X}_i , {\bf X}_j) &:=  {\bf h}_\tau ({\bf X}_i, {\bf X}_j)- {\bf h}_{\tau, 1} ({\bf X}_i) - {\bf h}_{\tau, 1} ({\bf X}_j)+ \Theta.
\end{align}
Since the Kendall kernel $h_\tau$ is bounded, $\|{\bf g}_1\|_2^2$ and
$\|{\bf g}_2\|_2^2$ are both less than ${m \choose 2} M$ for a certain
constant $M >0$ that does not depend on $n$ and $m$. Suppose
$d \in \{1, 2\}$ is the order of degeneracy for the kernel
${\bf h}_\tau$. By \citet[Corollary 8.1.7]{UstatBanach}, we have that
for any $t >0$,
\[
P(\|{\bf U}_\tau - \Theta_\tau\|_2 > t) \;\leq\; C_1\exp\left\{- C_2
  n\left(\frac{t^2 }{\lambda^2}\right)^{1/d}\right\}, 
\]
where $C_1, C_2 > 0$ are universal constants and
$\lambda^2 = M {m \choose 2} \sum_{c = 0}^{2 - d} n^{-c} = M {m
  \choose 2} \frac{1 - n^{d - 3}}{1 - n^{-1}}$.
Using the fact that
$\frac{1 - n^{d - 3}}{1 - n^{-1}} \leq \frac{1}{1 - n^{-1}}$ and
letting $t = C\sqrt{n}$ for some $C > 0$, we get
\begin{equation} \label{Banach_final} P(\|{\bf U}_\tau -
  \Theta_\tau\|_2 > C \sqrt{n}) \;\leq\; C_1\exp\left\{- C_2
    \left(\frac{2n(n -1)C^2 }{M m (m -1)}\right)\right\},
\end{equation}
for large enough $n$ as $\frac{m}{n} \longrightarrow \gamma$. The
proof for $(i)$ is completed by picking $C$ large so that the right
hand side of \eqref{Banach_final} is less than $1 - \beta$ as
$\frac{m}{n} \longrightarrow \gamma \in (0, \infty)$.

$(ii)$: Recall that $\mathbb{E}[T_\tau] = \|\Theta\|_2^2$, and the test $\phi(T_\tau)$ rejects $H_0$ when $T_\tau \geq \frac{4m}{9n}z_{1- \alpha}$. In what follows we let $\|\Theta\|_2 = C \sqrt{n}$ for an arbitrary fixed constant $C > 0$. By Chebyshev's inequality, for large enough $n$ under the regime $\frac{m}{n} \longrightarrow \gamma$,
\begin{multline}
1 - \mathbb{E}[ \phi(T_\tau)] = P\left( T_\tau - \|\Theta_\tau\|_2^2 \leq \frac{4m}{9n} z_{1 - \alpha} - \|\Theta_\tau\|_2^2\right) \\
\leq P\left(\left|T_\tau - \|\Theta_\tau\|_2^2 \right| \geq \left|\frac{4m}{9n} z_{1 - \alpha} - \|\Theta_\tau\|_2^2\right| \right) \leq \frac{\mathrm{Var}(T_\tau)}{   (\frac{4m}{9n}  z_{1 - \alpha} - \|\Theta_\tau\|_2^2 )^2} \label{Cheby},
\end{multline} 
where the first inequality is true when $C$ is taken large enough. We will finish the proof by showing that as $\frac{m}{n} \longrightarrow \gamma$, the rightmost term of \eqref{Cheby} is less than $1 - \beta$ when C is chosen large enough. To that end we will study the variance of the statistic $T_\tau$. Note that 
\[
T_\tau = \frac{1}{{n \choose 4}}  \sum_{1 \leq i < j < k < l \leq n} h^T_\tau({\bf X}_i, {\bf X}_j, {\bf X}_k, {\bf X}_l )
\]
is a U-statistic with the kernel of degree 4
\[
h^T_\tau({\bf X}_i, {\bf X}_j, {\bf X}_k, {\bf X}_l ) := 
\sum_{1 \leq p < q \leq m} h_\tau^W  ({\bf X}_i ^{(pq)}, {\bf X}_j ^{(pq)}, {\bf X}_k ^{(pq)}, {\bf X}_l ^{(pq)}),
\]
where $h^W_\tau$ is the function $h^W$ defined in \eqref{hw_kern} when
$h$ is the Kendall kernel $h_\tau$.  Here it is important to note that
the kernel $h^T_\tau$ also depends on the number of variables $m$ since it
is a sum of ${m \choose 2}$ terms.  By Lemma 5.2.1A in
\cite{Serfling1980}, the variance of $T_\tau$ satisfies
\begin{equation} \label{Var_Ttau}
\mathrm{Var}(T_\tau) := {n \choose 4}^{-1} \sum_{c = 1}^4 {4 \choose c } {n - 4 \choose 4 - c} \zeta_c^{\tau} 
\leq \frac{16\zeta^{h^T_\tau}_1}{n} + \frac{\widetilde{C}}{n^2}\left(\zeta^{h^T_\tau}_2 + \zeta^{h^T_\tau}_3+ \zeta^{h^T_\tau}_4\right)
\end{equation}
for a constant $\widetilde{C} > 0 $ that does not depend on $C$;  recall definition \eqref{cov_quan_0} for the kernel $h = h^T_\tau$.
\begin{claim*}
$\zeta_1^{h^T_\tau} \leq C^2 n m (m-1)$
\end{claim*}

\begin{proof}[Proof of the claim]
For seven distinct sample indices $i_1, \dots, i_7 \in \{1, \dots, n\}$,
 \begin{align*}
 \zeta_1^{h^T_\tau} &=  \mathbb{E} [h^T_\tau({\bf X}_{i_1}, \dots, {\bf X}_{i_4}) h^T_\tau({\bf X}_{i_4}, \dots, {\bf X}_{i_7})] - \|\Theta_\tau\|_2^4\\
 &= \sum_{ \substack{1 \leq p < q \leq m \\ 1 \leq p' < q' \leq m}}  \mathbb{E}[h^W_\tau ({\bf X}^{(pq)}_{i_1},\dots, {\bf X}^{(pq)}_{i_4}) h^W_\tau ({\bf X}^{(p'q')}_{i_4},\dots, {\bf X}^{(p'q')}_{i_7})] - \|\Theta_\tau\|_2^4\\
 &= \sum_{ \substack{1 \leq p < q \leq m \\ 1 \leq p' < q' \leq m}}\theta_\tau^{(pq)}\theta_\tau^{(p'q')}\zeta_1^{h_\tau} ,
 \end{align*}
 where the last equality is true by the definition of $h_\tau^W$ and
 independence. Since $|h_\tau| \leq 1 $, it is true that
 $ \zeta_1^{h_\tau} = |\zeta_1^{h_\tau}| \leq 2 $. This in turns
 implies that $\zeta_1^{h^T_\tau}$ is less than the quadratic form
 $2 \Theta_\tau^T {\bf J}_{m \choose 2} \Theta_\tau$, where
 ${\bf J}_{m \choose 2}$ is the ${m \choose 2}$-by-${m \choose 2}$
 semi-positive definite matrix with all $1$'s. Since the largest
 eigenvalue of ${\bf J}_{m \choose 2}$ is ${m \choose 2}$, given that
 $\|\Theta_\tau\|_2 = C \sqrt{n}$,
 \[2 \Theta_\tau^T {\bf J}_{m \choose 2} \Theta_\tau \leq C^2 n m
 (m-1),\] and the claim is proved.
\end{proof}

Returning to the other quantities in~(\ref{Var_Ttau}), since
$|h^W_\tau| \leq 1$, it is easy to show that each of
$\zeta_2^{h^T_\tau}$, $\zeta_3^{h^T_\tau}$ and $\zeta_4^{h^T_\tau}$ is
bounded by $2 {m \choose 2}^2$. Hence, under the regime $\frac{m}{n}
\longrightarrow \gamma$, together with the claim above,
\eqref{Var_Ttau} gives that for all large $n$,
\begin{equation} \label{varlessthan}
\mathrm{Var}(T_\tau) \leq   m^2( 16 C^2 + 3 \gamma^2 \widetilde{C}).
\end{equation}
Recalling that $\|\Theta_\tau\|^2 = C\sqrt{n}$, and applying
\eqref{varlessthan} to \eqref{Cheby}, we get that
\begin{equation} \label{Cheby2}
1 - \mathbb{E}[\phi(T_\tau)] \leq 
\frac{m^2( 16 C^2 + 3 \gamma^2 \widetilde{C})}{C^4 n^2 - C^2 \frac{8}{9}m z_{1 - \alpha} + \frac{16m^2}{81n^2}  z_{1- \alpha}^2}
\end{equation}
for all large $n$. Since $C$ is arbitrary, by choosing it large enough the right hand side of \eqref{Cheby2} can be made less than $1 - \beta$ as $\frac{m}{n} \longrightarrow \gamma$. 
\end{proof}

The following lemma is needed for the proof of \thmsref{conjectureThm} and \thmssref{HLnotWork}.

\begin{lemma} \label{lem:Taylor}
Let $I= [0, 1- \epsilon] \subset \mathbb{R}$ for some small fixed $\epsilon> 0$.  For fixed positive integers $c_1, \dots, c_b$ such that $\sum_{i = 1}^b c_i = c$, suppose ${\bf X} = (X^{(1)}, \dots, X^{(c)}) \sim N(0, \Sigma)$ is a $c$-variate normal random vector with an invertible block diagonal covariance matrix
\[
\Sigma = \Sigma(\rho) = \begin{bmatrix}
  B_1(\rho) &  \\

  & \ddots \\
   &  &B_b  (\rho)
 \end{bmatrix},
\]
where each $B_i(\rho)$ is a $c_i$-by-$c_i$ matrix with $1$'s on the diagonal and all off-diagonal entries equal to some $\rho \in I$. If $H: \mathbb{R}^c \longrightarrow \mathbb{R}$ is a bounded function such that $\mathbb{E}[H({\bf X})] = 0$ when $\rho = 0$, then there exists a constant $C = C(H, \epsilon )> 0$ such that $|\mathbb{E}[H({\bf X})]| \leq C \rho$ for all $\rho \in I$.
\end{lemma}

\begin{proof}
  For all $\rho \in I$, the matrix $\Sigma(\rho)$ is invertible and
  the precision matrix $\Sigma^{-1}(\rho)$ is a smooth function of
  $\rho$.  Hence, the set of distributions $N(0, \Sigma(\rho))$ forms
  a curved exponential family.  By standard results on exponential
  families \citep[Theorem 5.8]{TPE}, the expectation
  $\mathbb{E}[H({\bf X})]$ is a continuous function of $\rho$ that is
  differentiable on $(0, 1 - \epsilon)$. The lemma is thus implied by
  the mean value theorem and the compactness of $[0, 1 - \epsilon]$.
\end{proof}

\begin{proof}[Proof of \thmref{conjectureThm}]
  The value of $T_\tau$ depends only on the rank vectors
  ${\bf R}^{(1)}, \dots, {\bf R}^{(m)}$.  Without loss of generality,
  we may thus assume that each $X^{(p)}$ is centered with unit
  variance, i.e., $(X^{(1)}, \dots, X^{(m)}) \sim N(0, R)$, where
  $R = (\rho^{(pq)})$ is a correlation matrix, with $1$'s on
  the diagonal.

  It suffices to prove the result under the restriction that $\theta$
  can only take values in a closed interval $[0, 1 - \epsilon]$, for
  some fixed small $\epsilon > 0$.  In other words, in the statement
  of the theorem, replace the set of distributions
  $\mathcal{N}_m(\|\Theta_\tau\|_2 \geq \tilde{C}; \theta^{(pq)}_\tau
  = \theta)$ under the infimum by the subset
  \begin{equation} \label{subsetDist}
    \{N \in \mathcal{N}_m(\|\Theta_\tau\|_2 \geq \tilde{C};
    \theta^{(pq)}_\tau = \theta) : \theta \in  [0, 1 - \epsilon]\}. 
  \end{equation}
  To see that this restriction can be made, note that
  $\theta > 1 - \epsilon$ implies that
  $\|\Theta_\tau\|_2 > \sqrt{m \choose 2} (1 - \epsilon) = O(m)$.
  Since $O(m) > O(\sqrt{n})$ asymptotically under the regime
  $\frac{m}{n} \longrightarrow \gamma$, by \thmref{beatMax}$(ii)$,
  nothing is lost by ignoring the normal distributions in
  $\mathcal{N}_m(\|\Theta_\tau\|_2 \geq \tilde{C}; \theta^{(pq)}_\tau
  = \theta)$
  with $\theta > 1 - \epsilon$. In addition, for all $p \not= q$, by
  we have the classical result \citep[p.823]{kruskal1958},
\[
\rho^{(pq)} = \rho  = \sin\left(\frac{\pi \theta}{2}\right)
\] 
when $\theta_\tau^{(pq)} = \theta$. As a consequence, for the covariance matrix $R$ to be positive definite it must be that $ \theta >-\frac{2}{\pi} \arcsin [ \frac{1}{m -1}]$ \citep[Theorem 7.2.5]{HornAndJohnson}. Hence, as $n$ and $m$ grow, it can be seen that $\|\Theta_\tau\|_2 < 1/\sqrt{2}$ when $\theta$ lies in the interval $(-\frac{2}{\pi} \arcsin [ \frac{1}{m-1}], 0)$. As such, by taking the constant $\tilde{C}$ to be larger than $1/\sqrt{2}$ when necessary, it suffices to consider  the subset of distributions \eqref{subsetDist} under the infimum.

In what follows, the operators $\mathbb{E}[\cdot], \text{Var}[\cdot]$
and $\text{Cov}[\cdot]$ are all with respect to an $m$-variate normal
distribution for $(X^{(1)}, \dots, X^{(m)})$ in
\eqref{subsetDist}. Recall from \eqref{Var_Ttau}
that
\[
\mathrm{Var}(T_\tau) := {n \choose 4}^{-1} \sum_{c = 1}^4 {4 \choose c } {n - 4 \choose 4 - c} \zeta_c^{h^T_\tau}. 
\]
Our proof now begins with the Chebyshev inequality from \eqref{Cheby}:
\begin{align} \label{Cheby2}
1 - \mathbb{E}[ \phi_\alpha(T_\tau)] &\leq \frac{\text{Var}(T_\tau)}{ ( \frac{4m}{9n} z_{1 - \alpha} - \|\Theta_\tau\|_2^2 )^2 } 
\;\leq
\frac{\zeta_c^{h^T_\tau} B \sum_{c = 1}^4 n^{-c}}{
 ( \frac{4m}{9n}  z_{1 - \alpha})^2 - \frac{8m}{9n} z_{1 - \alpha}\|\Theta_\tau\|_2^2 + \|\Theta_\tau\|_2^4
},
\end{align}
where the last inequality is true since
${n \choose 4}^{-1}{4 \choose c } {n - 4 \choose 4 - c} \leq B n^{-c}$
for a constant $B > 0$.  To finish the proof, it suffices to show that
for each $c = 1, \dots, 4$, a constant
$\tilde{C}_c(\alpha, \beta, \gamma)> 0$ exists such that for large
enough $n$ (depending on $\tilde{C}_c$),
\begin{equation} \label{c=}
 \frac{B\zeta_c^{h^T_\tau} n^{-c} }{ ( \frac{4m}{9n}z_{1 - \alpha})^2 - \frac{8m}{9n} z_{1 - \alpha}\|\Theta_\tau\|_2^2 + \|\Theta_\tau\|_2^4} < \frac{1 - \beta}{4}
\end{equation}
whenever $\|\Theta_\tau\|_2 > \tilde{C}_c$.  We may then take $\tilde{C} = \max_{c= 1, \dots, 4} \tilde{C}_c$. 

For notational convenience, we define
\[
 f_{{\bf i}, {\bf j}} := \sum_{ \substack{1 \leq p < q \leq m \\ 1 \leq p' < q' \leq m}}  \mathbb{E}[h^W_\tau ({\bf X}^{(pq)}_{i_1},\dots, {\bf X}^{(pq)}_{i_4}) h^W_\tau ({\bf X}^{(p'q')}_{j_1},\dots, {\bf X}^{(p'q')}_{j_4})] \geq 0,
\]
for any tuples ${\bf i} = (i_1, \dots, i_4)$, ${\bf j} = (j_1, \dots,
j_4) \in \mathcal{P}(n, 4)$ such that $|{\bf i}\cap {\bf j}| = c$.  Then
\begin{equation} \label{zeta_c_simplified}
\zeta_c^{h^T_\tau}  = f_{{\bf i}, {\bf j}} - \|\Theta_\tau\|_2^4.
\end{equation}
Since the ratio 
\[ \frac{B \|\Theta_\tau\|^4_2}{( \frac{4m}{9n}  z_{1 - \alpha})^2 -
  \frac{8m}{9n}  z_{1 - \alpha}\|\Theta_\tau\|_2^2 +
   \|\Theta_\tau\|_2^4}
\]
is bounded for all values of $\|\Theta_\tau\|_2$, we have for each $c =
1,\dots, 4$ that
\[
\frac{B \|\Theta_\tau\|^4_2 n^{-c}}{( \frac{4m}{9n}  z_{1 -
    \alpha})^2 - \frac{8m}{9n}  z_{1 - \alpha}\|\Theta_\tau\|_2^2
  + \|\Theta_\tau\|_2^4} \longrightarrow 0
\]
as $\frac{m}{n} \longrightarrow \gamma$.  Upon substituting \eqref{zeta_c_simplified}
into \eqref{c=}, we see that the proof is finished if the
below claim is shown to be true.
\end{proof}

\begin{claim*}
Under $\theta_\tau^{(pq)} = \theta$, there exists for each $c = 1, \dots, 4$, a constant $\tilde{C}_c(\alpha, \beta, \gamma)> 0$ such that  for large enough $n$ (depending on $\tilde{C}_c$),  
\begin{equation} \label{c=again}
 \frac{B f_{{\bf i}, {\bf j}} n^{-c} }{ ( \frac{4m}{9n}  z_{1 - \alpha})^2 - \frac{8m}{9n} z_{1 - \alpha}\|\Theta_\tau\|_2^2 + \|\Theta_\tau\|_2^4} < \frac{1 - \beta}{5}.
\end{equation}
whenever $\|\Theta_\tau\|_2 = \theta\sqrt{{m \choose 2} }> \tilde{C}_c$.
\end{claim*}

\begin{proof}[Proof of the claim when $c = 1$] Using independence, we
  find that for any four distinct indices $1 \leq i, j, k \leq n$,
\begin{multline}\label{fc=1}
f_{\bf i, j} 
= \underbrace{ \sum_{\substack{ 1 \leq p < q \leq m \\ 1 \leq p' < q' \leq m \\ |\{p, q\}\cap \{p', q'\}| \geq 1} }  \theta^2 \mathbb{E}[h_\tau({\bf X}^{(pq)}_i, {\bf X}^{(pq)}_j)h_\tau({\bf X}^{(p'q')}_i, {\bf X}^{(p'q')}_k)]}_{(1)}  + \\
\underbrace{\sum_{ \substack{ 1 \leq p < q \leq m \\ 1 \leq p' < q' \leq m \\ |\{p, q\}\cap \{p', q'\}| = 0}}\theta^2 \mathbb{E}[h_\tau({\bf X}^{(pq)}_i , {\bf X}^{(pq)}_j)  
h_\tau({\bf X}^{(p'q')}_i , {\bf X}^{(p'q')}_k)]}_{(2)}. 
\end{multline}
Since $|h_\tau| \leq 1$, the term $(1)$ is bounded in absolute value
by  $[{m \choose 2}^2 - {m \choose 2}{m -2 \choose 2}] \theta^2 =
O(m)\|\Theta_\tau\|^2_2$. 
To bound $(2)$, note that when $|\{p, q\} \cap \{p', q'\}| = 0$, the expectation term 
\begin{equation} \label{exp_term_all_not_equal}
\mathbb{E}[h_\tau({\bf X}^{(pq)}_i , {\bf X}^{(pq)}_j)  
h_\tau({\bf X}^{(p'q')}_i , {\bf X}^{(p'q')}_k)]
\end{equation}
equals $0$ when $\theta = 0$ due to the independence of
$\{{\bf X}_i^{(pq)}, {\bf X}_j^{(pq)}\}$ and
$\{{\bf X}_i^{(p'q')}, {\bf X}_k^{(p'q')}\}$.  Moreover, for
$\theta\not=0$, the pairs
$\{{\bf X}_i^{(pq)}, {\bf X}_j^{(pq)}, {\bf X}_i^{(p'q')}, {\bf
  X}_k^{(p'q')}\}$
jointly follow a $8$-variate normal distribution with block diagonal
covariance matrix, where each block has $1$'s on the diagonal and all
its off-diagonal entries equal to $\rho = \sin(\pi \theta/2)$. By
\lemref{Taylor}, the expectation \eqref{exp_term_all_not_equal} is
bounded in absolute value, up to a multiplying constant, by $\theta$,
and hence $(2)$ bounded by $O(m^4) \theta^3 = O(m)\|\Theta_\tau\|_2^3$
in absolute value.

Using  the above bounds for $(1)$ and $(2)$ we get that the left hand side of \eqref{c=again} is less than 
\[
\frac{\frac{O(m)}{n} (\|\Theta_\tau\|^2_2 +    \|\Theta_\tau\|_2^3)}{ ( \frac{4m}{9n}  z_{1 - \alpha})^2 - \frac{8m}{9n}  z_{1 - \alpha}\|\Theta_\tau\|_2^2 + \|\Theta_\tau\|_2^4}.
\]
Under the regime $\frac{m}{n} \longrightarrow \gamma$, we see that the
expression in the above display can be made less than $\frac{1 -
  \beta}{5}$ when $\|\Theta_\tau\|_2$ and $n$ are large enough. 
\end{proof}
\begin{proof}[Proof of the claim when $c = 2$]
Again, using independence, we find that 
\begin{multline} \label{multline0}
9 f_{\bf i, j} =
\underbrace{\sum  4 \left(\mathbb{E}\left[h_\tau({\bf X}^{(pq)}_i, {\bf X}^{(pq)}_j) h_\tau({\bf X}^{(p'q')}_i, {\bf X}^{(p'q')}_k) \right]\right)^2 }_{(1)}+ \\
\underbrace{\sum \theta^2 \mathbb{E}[h_\tau({\bf X}_i^{(pq)}, {\bf X}_j^{(pq)}) h_\tau({\bf X}_i^{(p'q')}, {\bf X}_j^{(p'q')})]}_{(2)} +\\
\underbrace{\sum2 \theta \mathbb{E}[h_\tau({\bf X}_i^{(pq)}, {\bf X}_k^{(pq)}) h_\tau({\bf X}_i^{(p'q')}, {\bf X}_j^{(p'q')}) h_\tau({\bf X}_j^{(pq)}, {\bf X}_l^{(pq)})]}_{(3)} + \\
\underbrace{\sum2 \theta \mathbb{E}[h_\tau({\bf X}_i^{(p'q')}, {\bf X}_k^{(p'q')}) h_\tau({\bf X}_i^{(pq)}, {\bf X}_j^{(pq)}) h_\tau({\bf X}_j^{(p'q')}, {\bf X}_l^{(p'q')})]}_{(4)},
\end{multline}
where each summation is over all pairs $1 \leq p <q \leq m$ and
$1 \leq p' <q' \leq m$, and $i, j, k, l$ are any $4$ distinct indices
in $\{1, \dots, n\}$.  We now derive bounds for the absolute values of
the terms $(1), (2), (3), (4)$.

\ \

{\em Term $(1)$:}  We claim that $|(1)| \leq O(m^2) (1 + \|\Theta_\tau\|^2_2)$.
 To show this, observe that $(1)$ equals
\begin{multline} \label{multline1}
\sum_{ 1\leq p < q \leq m} 4
(\mathbb{E}[h_\tau({\bf X}_i^{(pq)}, {\bf X}_j^{(pq)}) h_\tau({\bf X}_i^{(pq)}, {\bf X}_k^{(pq)})])^2 + \\
\sum_{ \substack{|\{p, q \} \cap \{p', q'\}| = 0\\ 1 \leq p < q \leq m\\ 1 \leq p' < q' \leq m }}4 (\mathbb{E}[h_\tau({\bf X}_i^{(pq)}, {\bf X}_j^{(pq)}) h_\tau({\bf X}_i^{(p'q')}, {\bf X}_k^{(p'q')})])^2.
\end{multline}
Since $|h_\tau|
\leq
1$, the first sum in \eqref{multline1} is bounded by a term of order
$O(m^2)$.  Considering the second sum, an expectation
\begin{equation} \label{expectation2}
\mathbb{E}[h_\tau({\bf X}_i^{(pq)}, {\bf X}_j^{(pq)}) h_\tau({\bf X}_i^{(p'q')}, {\bf X}_k^{(p'q')})]
\end{equation}
with $\{p,
q\} \not = \{p' , q'\}$ equals $0$ when $\theta =
0$ by independence.  Moreover, ${\bf X}^{(pq)}_i$, ${\bf
  X}^{(pq)}_j$, ${\bf X}^{(p'q')}_i$, and ${\bf
  X}^{(p'q')}_k$ jointly follow an
$8$-variate
normal distribution with block diagonal covariance matrix as in
\lemref{Taylor}. By that lemma and the fact that $\rho
= \sin(\pi
\theta/2)$, we obtain that \eqref{expectation2} is bounded in absolute
value by $\theta$
times a constant, hence the second sum in \eqref{multline1} is bounded
in absolute value by a term equal to $O(m^2)
\|\Theta_\tau\|_2^2$.  Gathering the bounds for the two sums in
\eqref{multline1} gives the claimed bound for the absolute value of
term $(1)$.

\medskip

{\em Term $(2)$:} We claim that
$|(2)| \leq O(m^2)\|\Theta_\tau\|_2^2$.  Indeed, since
$|h_\tau| \leq 1$, it is easy show that $(2)$ is bounded in absolute
value by
${m \choose 2}^2 \theta^2 = {m \choose 2}\|\Theta_\tau\|_2^2 =
O(m^2)\|\Theta_\tau\|_2^2 $.

\medskip

{\em Terms $(3)$ and $(4)$:} We claim that
$|(3)|, |(4)| \leq O(m^2) (\|\Theta_\tau\|_2 + \|\Theta_\tau\|_2^2)$.
We give details for the proof of bound for $|(3)|$.  The bound for
$(4)$ is analogous.  We write $(3)$ as
\begin{multline} \label{multline2}
\sum_{ \substack{|\{p, q \} \cap \{p', q'\}| \geq 1\\ 1 \leq p < q \leq m\\ 1 \leq p' < q' \leq m }} 2 \theta \mathbb{E}[h_\tau({\bf X}_i^{(pq)}, {\bf X}_k^{(pq)}) h_\tau({\bf X}_i^{(p'q')}, {\bf X}_j^{(p'q')}) h_\tau({\bf X}_j^{(pq)}, {\bf X}_l^{(pq)})] + \\
\sum_{\substack{|\{p, q \} \cap \{p', q'\}| = 0\\ 1 \leq p < q \leq m\\ 1 \leq p' < q' \leq m }} 2 \theta\mathbb{E}[h_\tau({\bf X}_i^{(pq)}, {\bf X}_k^{(pq)}) h_\tau({\bf X}_i^{(p'q')}, {\bf X}_j^{(p'q')}) h_\tau({\bf X}_j^{(pq)}, {\bf X}_l^{(pq)})],
\end{multline}
where the first sum is bounded by
$2 \theta {m \choose 2}[ {m \choose 2}- {m -2 \choose 2}]
=O(m^2)\|\Theta_\tau\|_2$ because $|h_\tau| \leq 1$.  The expectation
\[
\mathbb{E}[h_\tau({\bf X}_i^{(pq)}, {\bf X}_k^{(pq)}) h_\tau({\bf X}_i^{(p'q')}, {\bf X}_j^{(p'q')}) h_\tau({\bf X}_j^{(pq)}, {\bf X}_l^{(pq)})]
\]
equals $0$ when $\{p, q \} \cap \{p', q'\}| = 0$, and \lemref{Taylor}
can be invoked to show the second sum in \eqref{multline2} is bounded
in absolute value by $O(m^2) \|\Theta_\tau\|^2_2$.

\medskip

Having established the bounds for the terms $(1) - (4)$ in
\eqref{multline0}, we find that when $c = 2$ the left hand side of
\eqref{c=again} is less than
\[
 \frac{O(m^2)n^{-2} (1 + \|\Theta_\tau\|_2 + \|\Theta_\tau\|_2^2)}{ ( \frac{4m}{9n}  z_{1 - \alpha})^2 - \frac{8m}{9n} z_{1 - \alpha}\|\Theta_\tau\|_2^2 + \|\Theta_\tau\|_2^4},
\]
which, under $\frac{m}{n} \longrightarrow \gamma$, can be made to be less than $\frac{1 - \beta}{5}$ when $\|\Theta_\tau\|_2$ and $n$ are large enough.
\end{proof}
\begin{proof}[Proof of the claim when $c \ge 3$] For
  $c=3$ or $c=4$, we may proceed similarly, using again the
  boundedness of $h_\tau$ and 
  \lemref{Taylor}.  We note that if $c=3$, then
  $|f_{{\bf i , j}}| \leq O(m^3) (1 + \|\Theta_\tau\|_2)$ and omit 
  further details.
\end{proof}

\begin{proof}[Proof of \thmref{HLnotWork}]
 By \cite{HanLiu2014}, under $H_0$, 
\begin{equation} \label{HLconv}
\sup_{t \in \mathbb{R}}\left|P_0\left( n(4 \zeta_1^{h_\tau})^{-1} (S^{\max}_\tau)^2   - 4 \log m +  \log \log m \leq t\right) - \exp( - \exp(-t/2) /\sqrt{8 \pi})\right| \longrightarrow 0
\end{equation}
as $\frac{m}{n} \longrightarrow \gamma$ (recall that $\frac{m}{n} \longrightarrow \gamma$ is a special case of the regime $\log m = o(n^{1/3})$ considered in \cite{HanLiu2014}), where $\zeta_1^{h_\tau} = 1/9$ as given in \tabref{degen_tab} and $\exp( - \exp(-t/2) /\sqrt{8 \pi})$ is the distribution function of a Gumbel-distributed random variable $G$. Defining $H(t)$ be the distribution function of the transformed random variable $\sqrt{4 \zeta_1^{h_\tau}(G + 4 \log m - \log \log m)}$, \eqref{HLconv} is equivalent to 
\begin{equation} \label{HLconv2}
\sup_{t \in \mathbb{R}} \left|  P_0 (\sqrt{n}S^{\max}_\tau \leq t) - H(t)   \right| \longrightarrow 0
\end{equation}
  Hence, the critical value of the test $\phi_\alpha(S_\tau^{\max})$ is calibrated by the $(1 - \alpha)$-quantile of the distribution function $H(\cdot)$, and $H_0$ is rejected if $\sqrt{n} S^{\max}_\tau$ exceeds this value. 

Suppose, towards a contradiction, that the constant $C$ indicated by the theorem exists, and let ${\bf X}$ be $m$-variate normal with $\theta_\tau^{(pq)}  = \theta_\tau= \frac{\sqrt{2}C}{\gamma' n}$ for some $\gamma' < \gamma$ and all $1\leq p < q \leq  m$,  such that $\|\Theta_\tau\|_2 \geq C$ as $m, n \longrightarrow \infty$. By $\rho^{(pq)} = \rho=\sin\big(\frac{\pi}{2}\theta_\tau^{(pq)}\big)$ \citep{ellipKendall}, the distribution of ${\bf X}$ belongs to the set of equicorrelation alternatives $\mathcal{N}_m^{\text{equi}}(\|\Theta_\tau \|_2 \geq
  C)$, and we will use $P(\cdot)$ to denote the probability operator under this distribution. To finish the proof it suffices to show that 
\begin{equation} \label{conv_2_extreme}
\sup_{t \in \mathbb{R}}\left|P\left(\sqrt{n}\max_{ 1\leq p < q \leq m} \left| U_{h_\tau}^{(pq)}- \sqrt{2}C(\gamma' n)^{-1}\right|\leq t\right) - H(t)\right| \longrightarrow 0
\end{equation}
as $\frac{m}{n} \longrightarrow \infty$. Since 
\[
\left|\sqrt{n}\max_{ 1\leq p < q \leq m} \left| U_{h_\tau}^{(pq)}- \sqrt{2}C(\gamma' n)^{-1}\right| -
\sqrt{n}S^{\max}_\tau\right| \leq \frac{\sqrt{2}C}{\gamma' \sqrt{n}}
\] 
by reverse triangular inequality, if \eqref{conv_2_extreme} is true,
by uniform continuity  of $H(\cdot)$ (as $H(\cdot)$ is a continuous
distribution function)  and the fact that $\frac{\sqrt{2}C}{\gamma'
  \sqrt{n}} \longrightarrow 0$ , elementary arguments show that
\[
\sup_{t \in \mathbb{R}}\left|P(\sqrt{n}|S^{\max}_\tau| \leq t) - H(t)\right| \longrightarrow 0
\]
as $\frac{m}{n} \longrightarrow \gamma$. As such,  the asymptotic power of our test $\phi_\alpha(S_\tau^{\max})$, under this alternative, also equals $\alpha$, leading to the desired contradiction since $\beta > \alpha$.

It remains to  show \eqref{conv_2_extreme}. Since the value of
$S_\tau^{\max}$ depends only on the rank vectors ${\bf R}^{(1)},
\dots, {\bf R}^{(m)}$, we can assume without loss of generality that the components of ${\bf X}$ have variance $1$. Let $\Gamma^0$ and $\Gamma$ be two ${m \choose 2} \times {m \choose 2}$ matrices, whose components are indexed by $((p, q), (p', q'))$ for $1 \leq p< q \leq m$, $1 \leq p'< q' \leq m$, and defined as
\begin{align*}
\Gamma_{(p, q), (p', q')}^0 &:= \mathbb{E}_0[h_\tau({\bf X}_i^{(pq)}, {\bf X}_j^{(pq)}) h_\tau({\bf X}_i^{(p'q')}, {\bf X}_k^{(p'q')})] 
 \quad \text{and} \quad \\
\Gamma_{(p, q), (p', q')} &:= \mathbb{E}[h_\tau({\bf X}_i^{(pq)}, {\bf X}_j^{(pq)}) h_\tau({\bf X}_i^{(p'q')}, {\bf X}_k^{(p'q')})] - \theta_\tau^2 \\
&=  \mathbb{E}[h_\tau({\bf X}_i^{(pq)}, {\bf X}_j^{(pq)}) h_\tau({\bf X}_i^{(p'q')}, {\bf X}_k^{(p'q')})] - \frac{2C^2}{(\gamma'n)^2}. 
\end{align*}
Here, again, $\mathbb{E}[\cdot]$ is the expectation operator under the
alternative distribution of ${\bf X}$. We note that $\Gamma^0$ and
$\Gamma$ are in fact the covariance matrices of the H\'{a}jek
projection of the vector-valued U-statistic ${\bf U}_\tau$ defined in
the proof of \thmref{beatMax}.  Applying Theorem 2.1 from
\citet{chen2016sup}, we obtain that 
\begin{align}
&\sup_{t \in \mathbb{R}}\left|P_0(\sqrt{n}|S^{\max}_\tau| \leq t) - F_0(t)\right| \longrightarrow 0
\quad \text{and} \label{chen1}\\
\quad
&\sup_{t \in \mathbb{R}}\left|P(\sqrt{n}\max_{ 1\leq p < q \leq m} | U_{h_\tau}^{(pq)}- \sqrt{2}C(\gamma' n)^{-1}| \leq t) - F(t)\right| \longrightarrow 0
\label{chen2}
\end{align}
as $\frac{m}{n} \longrightarrow \gamma$, where $F_0(\cdot)$ and $F(\cdot)$ are, respectively, the cumulative distribution functions of $\|Z_0\|_\infty$ and $\|Z\|_\infty$ for multivariate normal random vector $Z_0 \sim \mathcal{N}_{m \choose 2}(0, \Gamma^0)$ and $Z \sim \mathcal{N}_{m \choose 2}(0, \Gamma)$. Now by \lemref{Taylor}, for each pair $((p, q), (p', q'))$, 
\[
|\Gamma_{(p, q), (p', q')} - \Gamma^0_{(p, q), (p', q')}| \lesssim  \rho +  \theta_\tau^2 = O(n^{-1}),
\]
and hence by a comparison lemma in \citet[Lemma 3.1]{MR3161448},
\begin{equation}\label{GaussComparison}
\sup_{t \in \mathbb{R}}|F_0(t)-F(t) |
\lesssim  O(n^{-1/3})\left(1\vee \log ( O(m n))\right)^{2/3}, 
\end{equation}
where the right hand side of \eqref{GaussComparison} converges to $0$ as $\frac{m}{n} \longrightarrow \gamma$. Collecting \eqref{HLconv2}, \eqref{chen1}, \eqref{chen2} and \eqref{GaussComparison} leads to \eqref{conv_2_extreme}. 
\end{proof}

\bibliographystyle{ba}
\bibliography{Kendall_bib}

\begin{table}[ht]
\centering
\caption[Table caption text]{Simulated size of tests when $X^{(1)}, \dots, X^{(m)}$ are i.i.d.~$t_{3,2}$ data. For each combination of $(m, n)$ and each test, the sizes are computed from $5000$ independently generated datasets.}
\begin{tabular}{c|c|*{8}{c}}
  \hline
Statistics &  $ n \backslash m$  & 4 & 8 & 16 & 32 & 64 & 128 & 256 & 512 \\ 
  \hline
  $S_r$&\multirow{7}{*}{16} & 0.060 & 0.065 & 0.062 & 0.067 & 0.071 & 0.063 & 0.071 &0.072\\ 
  $S_\tau$ &  & 0.069 & 0.079 & 0.080 & 0.090 & 0.094 & 0.093 & 0.086& 0.089 \\ 
  $T_\tau$  &  & 0.088 & 0.096 & 0.102 & 0.113 & 0.120 & 0.110 & 0.113 &0.114\\ 
  $S_{\rho_s}$ &  & 0.046 & 0.050 & 0.052 & 0.057 & 0.059 & 0.053 & 0.053 &0.055\\ 
$ T_{\hat{\rho}_s}$ &  & 0.079 & 0.093 & 0.099 & 0.107 & 0.111 & 0.107 & 0.104 &0.109\\ 
 $S_{t^*}$ &  & 0.079 & 0.098 & 0.115 & 0.112 & 0.123 & 0.122 & 0.111 &0.121\\ 
 $Z_{t^*}$  &  & 0.079 & 0.092 & 0.098 & 0.098 & 0.111 & 0.104 & 0.096 & 0.099 \\ 
\hline
   $S_r$&\multirow{7}{*}{32}  & 0.066 & 0.078 & 0.076 & 0.081 & 0.076 & 0.089 & 0.079& 0.086\\ 
   $S_\tau$ &  & 0.059 & 0.069 & 0.067 & 0.077 & 0.073 & 0.071 & 0.070& 0.077\\ 
$T_\tau$&  & 0.064 & 0.078 & 0.075 & 0.087 & 0.081 & 0.082 & 0.080& 0.086 \\ 
  $S_{\rho_s}$ &  & 0.047 & 0.054 & 0.052 & 0.061 & 0.056 & 0.053 & 0.056 &0.058\\ 
$ T_{\hat{\rho}_s}$&  & 0.062 & 0.075 & 0.072 & 0.082 & 0.080 & 0.079 & 0.072&0.083 \\ 
 $S_{t^*}$ &  & 0.056 & 0.081 & 0.085 & 0.090 & 0.088 & 0.078 & 0.087&0.085 \\ 
  $Z_{t^*}$&  & 0.062 & 0.069 & 0.067 & 0.081 & 0.077 & 0.077 & 0.079&0.078 \\ 
\hline
$S_r$ & \multirow{7}{*}{64} & 0.073 & 0.083 & 0.095 & 0.095 & 0.102 & 0.097 & 0.096 &0.091\\ 
  $S_\tau$  &  & 0.057 & 0.061 & 0.062 & 0.065 & 0.058 & 0.058 & 0.065& 0.059\\ 
 $T_\tau$ &  & 0.058 & 0.064 & 0.066 & 0.069 & 0.061 & 0.064 & 0.067& 0.062\\ 
 $S_{\rho_s}$ &  & 0.048 & 0.053 & 0.055 & 0.055 & 0.050 & 0.052 & 0.057 & 0.048\\ 
$ T_{\hat{\rho}_s}$ &  & 0.057 & 0.061 & 0.065 & 0.067 & 0.060 & 0.064 & 0.059&0.062 \\ 
 $S_{t^*}$ &  & 0.045 & 0.074 & 0.064 & 0.070 & 0.068 & 0.070 & 0.069&0.063 \\ 
  $Z_{t^*}$&  & 0.054 & 0.061 & 0.058 & 0.064 & 0.065 & 0.062 & 0.063 &0.064\\
\hline 
$S_r$ &  \multirow{7}{*}{128} & 0.072 & 0.089 & 0.107 & 0.112 & 0.101 & 0.109 & 0.110& 0.115 \\ 
$S_\tau$ &  & 0.047 & 0.061 & 0.053 & 0.061 & 0.052 & 0.056 & 0.053&0.055 \\ 
 $T_\tau$&  & 0.049 & 0.063 & 0.053 & 0.064 & 0.054 & 0.060 & 0.054& 0.058 \\ 
 $S_{\rho_s}$ &  & 0.043 & 0.059 & 0.049 & 0.056 & 0.048 & 0.052 & 0.048 &0.051\\ 
$ T_{\hat{\rho}_s}$ &  & 0.048 & 0.062 & 0.052 & 0.060 & 0.055 & 0.057  & 0.058&0.054\\
 $S_{t^*}$ &  & 0.041 & 0.066 & 0.070 & 0.071 & 0.060 & 0.058 & 0.052& 0.058\\ 
  $Z_{t^*}$&  & 0.050 & 0.055 & 0.058 & 0.062 & 0.053 & 0.056 & 0.055 &0.055\\ 
\hline
\end{tabular}
\label{tab:tncsize}
\end{table}


\begin{table}[ht]
\centering
\caption{Simulated power of tests when
 data are generated from the multivariate normal (MVN),  multivariate t  (MVT) and multivariate power exponential (MVPE) distributions  with three different values for the dependency signal $\|\Theta_\tau\|^2_2$. 
All pairwise (population) Kendall's tau correlations $\theta^{(pq)}_\tau, 1\leq p <q \leq m$ are equal to the same value $\theta$ so that $\|\Theta_\tau\|^2_2 ={m \choose 2} \theta^2$. For each combination of $(m, n)$ and each test, the  power is calculated from $500$ independently generated datasets.}
\begin{tabular}{c|c| *{3}{c}|*{3}{c}|*{3}{c}}
    \hline
 &  &   \multicolumn{3}{c}{ $\|\Theta_\tau\|^2_2=0.1$}   &\multicolumn{3}{|c|}{$\|\Theta_\tau\|^2_2=0.3$}&\multicolumn{3}{c}{ $\|\Theta_\tau\|^2_2=0.7$} \\ 
\hline
Statistic & $ n \backslash m$ &64 & 128 & 256 & 64 & 128 & 256 & 64 & 128 & 256 \\  
\hline
\hline
\multicolumn{11}{c}{MVN}\\
  \hline
\hline
  $S_\tau$ & \multirow{5}{*}{64} & 0.094 & 0.054 & 0.070 & 0.182 & 0.108 & 0.092 & 0.424 & 0.218 & 0.114 \\ 
   $T_\tau$  &  & 0.100 & 0.068 & 0.078 & 0.194 & 0.110 & 0.090 & 0.426 & 0.228 & 0.134 \\ 
   $S_\tau^{\max}$ &  & 0.046 & 0.046 & 0.020 & 0.040 & 0.058 & 0.046 & 0.056 & 0.054 & 0.058 \\ 
  $S_r$ &  & 0.070 & 0.058 & 0.070 & 0.178 & 0.114 & 0.080 & 0.448 & 0.222 & 0.110 \\ 
  \text{Cai \& Ma}  &  & 0.076 & 0.076 & 0.060 & 0.190 & 0.116 & 0.086 & 0.456 & 0.278 & 0.130 \\ 
\hline
  $S_\tau$ &\multirow{6}{*}{128}  & 0.130 & 0.086 & 0.056 & 0.342 & 0.164 & 0.080 & 0.794 & 0.444 & 0.176 \\ 
   $T_\tau$  &  & 0.132 & 0.088 & 0.058 & 0.352 & 0.174 & 0.084 & 0.806 & 0.446 & 0.186 \\ 
   $S_\tau^{\max}$ &  & 0.062 & 0.064 & 0.052 & 0.046 & 0.058 & 0.060 & 0.094 & 0.058 & 0.060 \\ 
  $S_r$ &  & 0.142 & 0.072 & 0.066 & 0.378 & 0.172 & 0.084 & 0.832 & 0.514 & 0.198 \\ 
 \text{LRT}&  &0.094& -- & -- & 0.204 & -- & -- &0.396 & -- & -- \\  \text{Cai \& Ma}  &  & 0.134 & 0.064 & 0.068 & 0.386 & 0.172 & 0.096 & 0.834 & 0.520 & 0.204 \\ 
\hline
  $S_\tau$ & \multirow{6}{*}{256} & 0.256 & 0.108 & 0.096 & 0.780 & 0.358 & 0.198 & 0.992 & 0.838 & 0.476 \\ 
   $T_\tau$  &  & 0.262 & 0.114 & 0.094 & 0.782 & 0.364 & 0.200 & 0.992 & 0.830 & 0.470 \\ 
   $S_\tau^{\max}$ &  & 0.048 & 0.050 & 0.046 & 0.064 & 0.056 & 0.058 & 0.124 & 0.082 & 0.052 \\ 
  $S_r$ &  & 0.282 & 0.126 & 0.094 & 0.816 & 0.420 & 0.224 & 1.000 & 0.880 & 0.502 \\ 
 \text{LRT}&  &0.166 &0.086  & -- & 0.450 &0.152  & -- &0.876 & 0.370 & -- \\
  \text{Cai \& Ma}  &  & 0.282 & 0.124 & 0.110 & 0.812 & 0.422 & 0.234 & 1.000 & 0.882 & 0.494 \\ 
\hline
\hline
\multicolumn{11}{c}{MVT}\\
  \hline
\hline
  $S_\tau$ &\multirow{3}{*}{64} & 0.506 & 0.866 & 0.998 & 0.628 & 0.896 & 0.998 & 0.802 & 0.926 & 0.998 \\ 
   $T_\tau$  &  & 0.130 & 0.080 & 0.078 & 0.232 & 0.128 & 0.096 & 0.488 & 0.234 & 0.114 \\ 
   $S_\tau^{\max}$ &  & 0.080 & 0.066 & 0.060 & 0.086 & 0.074 & 0.060 & 0.110 & 0.074 & 0.068 \\ 
  \hline
  $S_\tau$ & \multirow{3}{*}{128}  & 0.554 & 0.912 & 0.998 & 0.806 & 0.948 & 1.000 & 0.962 & 0.990 & 1.000 \\ 
   $T_\tau$  &  & 0.130 & 0.102 & 0.094 & 0.384 & 0.210 & 0.114 & 0.796 & 0.494 & 0.244 \\ 
   $S_\tau^{\max}$ &  & 0.064 & 0.060 & 0.054 & 0.080 & 0.064 & 0.066 & 0.114 & 0.074 & 0.076 \\ 
 \hline
  $S_\tau$ & \multirow{3}{*}{256}  & 0.694 & 0.924 & 1.000 & 0.972 & 0.992 & 1.000 & 1.000 & 1.000 & 1.000 \\ 
   $T_\tau$  &  & 0.268 & 0.130 & 0.084 & 0.740 & 0.348 & 0.188 & 0.998 & 0.832 & 0.456 \\ 
   $S_\tau^{\max}$ &  & 0.076 & 0.062 & 0.072 & 0.110 & 0.066 & 0.076 & 0.186 & 0.102 & 0.078 \\ 
\hline
\hline
\multicolumn{11}{c}{MVPE}\\
  \hline
\hline
  $S_\tau$ &  \multirow{3}{*}{64} & 0.052 & 0.042 & 0.022 & 0.128 & 0.056 & 0.044 & 0.358 & 0.122 & 0.060 \\ 
   $T_\tau$  &  & 0.114 & 0.076 & 0.076 & 0.222 & 0.110 & 0.082 & 0.462 & 0.216 & 0.134 \\ 
   $S_\tau^{\max}$ &  & 0.056 & 0.050 & 0.032 & 0.046 & 0.050 & 0.034 & 0.062 & 0.054 & 0.036 \\ 
  \hline
  $S_\tau$ & \multirow{3}{*}{128}  & 0.074 & 0.038 & 0.028 & 0.274 & 0.094 & 0.036 & 0.744 & 0.314 & 0.112 \\ 
   $T_\tau$  &  & 0.128 & 0.084 & 0.056 & 0.398 & 0.174 & 0.096 & 0.836 & 0.454 & 0.214 \\ 
   $S_\tau^{\max}$ &  & 0.038 & 0.054 & 0.050 & 0.050 & 0.056 & 0.044 & 0.084 & 0.060 & 0.046 \\ 
  \hline
  $S_\tau$ & \multirow{3}{*}{256}  & 0.134 & 0.066 & 0.050 & 0.638 & 0.256 & 0.102 & 0.992 & 0.794 & 0.306 \\ 
   $T_\tau$  &  & 0.232 & 0.152 & 0.100 & 0.768 & 0.370 & 0.184 & 0.998 & 0.862 & 0.450 \\ 
   $S_\tau^{\max}$ &  & 0.052 & 0.036 & 0.060 & 0.074 & 0.040 & 0.060 & 0.120 & 0.064 & 0.062 \\ 
\hline
   \hline
\end{tabular}
\label{tab:empEvidence}
\end{table}

\begin{table}[ht]
\centering
\caption[Table caption text]{Simulated power of tests when
 data are generated from multivariate normal (MVN),  multivariate t  (MVT) and multivariate power exponential (MVPE) distributions  with three different values for the dependency signal $\|\Theta_\tau\|^2_2$. 
For each distribution, the scatter
matrix $\Sigma = (\sigma_{ij})$ is a pentadiagonal matrix with $1$'s
on the diagonal and equal values for the non-zero entries $\sigma_{ij}$,
$1 \leq |i - j| \leq 2$. 
 For each combination of $(m, n)$ and each test, the power is calculated from $500$ independently generated datasets.}
\begin{tabular}{c|c| *{3}{c}|*{3}{c}|*{3}{c}}
  \hline
 &  &   \multicolumn{3}{c}{ $\|\Theta_\tau\|^2_2=0.1$}   &\multicolumn{3}{|c|}{$\|\Theta_\tau\|^2_2=0.3$}&\multicolumn{3}{c}{ $\|\Theta_\tau\|^2_2=0.7$} \\ 
\hline
Statistic & $ n \backslash m$ &64 & 128 & 256 & 64 & 128 & 256 & 64 & 128 & 256 \\  
\hline
\hline
\multicolumn{11}{c}{MVN}\\
  \hline
\hline

  $S_\tau$ & \multirow{5}{*}{64} & 0.096 & 0.070 & 0.062 & 0.170 & 0.108 & 0.084 & 0.462 & 0.210 & 0.120 \\ 
  $T_\tau$ &  & 0.100 & 0.080 & 0.068 & 0.176 & 0.118 & 0.090 & 0.462 & 0.224 & 0.122 \\ 
  $S_\tau^{\max}$ &  & 0.056 & 0.062 & 0.048 & 0.064 & 0.036 & 0.048 & 0.090 & 0.050 & 0.056 \\ 
 $S_r$ &  & 0.068 & 0.070 & 0.062 & 0.162 & 0.092 & 0.078 & 0.478 & 0.206 & 0.116 \\ 
  \text{Cai \& Ma} &  & 0.086 & 0.074 & 0.060 & 0.176 & 0.104 & 0.086 & 0.500 & 0.228 & 0.132 \\ 

\hline

   $S_\tau$& \multirow{6}{*}{128}  & 0.140 & 0.078 & 0.072 & 0.390 & 0.176 & 0.104 & 0.862 & 0.434 & 0.190 \\ 
   $T_\tau$&  & 0.138 & 0.084 & 0.070 & 0.398 & 0.186 & 0.100 & 0.870 & 0.438 & 0.186 \\ 
  $S_\tau^{\max}$ &  & 0.066 & 0.038 & 0.042 & 0.078 & 0.048 & 0.044 & 0.156 & 0.092 & 0.036 \\ 
 $S_r$ &  & 0.126 & 0.070 & 0.064 & 0.428 & 0.178 & 0.092 & 0.914 & 0.518 & 0.180 \\ 
 \text{LRT}&  & 0.126& -- & -- & 0.300 & -- & -- & 0.776 & -- & -- \\
 \text{Cai \& Ma}&  & 0.118 & 0.062 & 0.066 & 0.414 & 0.180 & 0.094 & 0.928 & 0.506 & 0.198 \\ 
\hline

$S_\tau$ & \multirow{6}{*}{256}  & 0.246 & 0.120 & 0.078 & 0.818 & 0.394 & 0.168 & 1.000 & 0.906 & 0.476 \\ 
  $T_\tau$ &  & 0.246 & 0.120 & 0.082 & 0.808 & 0.402 & 0.164 & 1.000 & 0.908 & 0.474 \\ 
  $S_\tau^{\max}$&  & 0.086 & 0.038 & 0.064 & 0.136 & 0.080 & 0.064 & 0.618 & 0.136 & 0.066 \\ 
   $S_r$&  & 0.268 & 0.120 & 0.090 & 0.864 & 0.430 & 0.172 & 1.000 & 0.952 & 0.510 \\ 
 \text{LRT}&  & 0.220 & 0.092 & -- &  0.780 & 0.314  & -- & 1.000&  0.84 & -- \\ 
\text{Cai \& Ma}&  & 0.258 & 0.120 & 0.094 & 0.864 & 0.420 & 0.196 & 1.000 & 0.954 & 0.522 \\ 
\hline
\hline
\multicolumn{11}{c}{MVT}\\
  \hline
\hline
$S_\tau$ & \multirow{3}{*}{64} & 0.484 & 0.870 & 0.998 & 0.634 & 0.900 & 0.998 & 0.832 & 0.946 & 0.998 \\ 
  $T_\tau$ &  & 0.116 & 0.080 & 0.072 & 0.214 & 0.128 & 0.082 & 0.440 & 0.218 & 0.112 \\ 
  $S_\tau^{\max}$&  & 0.078 & 0.070 & 0.060 & 0.092 & 0.068 & 0.066 & 0.132 & 0.086 & 0.076 \\ 
 \hline
 $S_\tau$ &  \multirow{3}{*}{128} & 0.560 & 0.912 & 0.998 & 0.830 & 0.950 & 0.998 & 0.988 & 0.992 & 1.000 \\ 
  $T_\tau$ &  & 0.124 & 0.102 & 0.086 & 0.370 & 0.180 & 0.130 & 0.884 & 0.482 & 0.242 \\ 
  $S_\tau^{\max}$ &  & 0.068 & 0.062 & 0.070 & 0.102 & 0.064 & 0.076 & 0.238 & 0.102 & 0.078 \\ 
\hline
 $S_\tau$   & \multirow{3}{*}{256} & 0.712 & 0.932 & 1.000 & 0.978 & 0.988 & 1.000 & 1.000 & 1.000 & 1.000 \\ 
  $T_\tau$ &  & 0.256 & 0.134 & 0.076 & 0.804 & 0.344 & 0.170 & 1.000 & 0.892 & 0.480 \\ 
 $S_\tau^{\max}$&  & 0.094 & 0.066 & 0.096 & 0.220 & 0.092 & 0.090 & 0.638 & 0.226 & 0.132 \\ 

\hline
\hline
\multicolumn{11}{c}{MVPE}\\
  \hline
\hline
   $S_\tau$ &\multirow{3}{*}{64}  & 0.054 & 0.038 & 0.026 & 0.120 & 0.062 & 0.030 & 0.324 & 0.110 & 0.048 \\ 
  $T_\tau$ &  & 0.120 & 0.074 & 0.072 & 0.204 & 0.108 & 0.094 & 0.462 & 0.212 & 0.128 \\ 
 $S_\tau^{\max}$ &  & 0.056 & 0.046 & 0.034 & 0.078 & 0.046 & 0.038 & 0.094 & 0.048 & 0.038 \\ 
 \hline
  $S_\tau$ & \multirow{3}{*}{128} & 0.060 & 0.036 & 0.028 & 0.250 & 0.082 & 0.038 & 0.822 & 0.272 & 0.092 \\ 
 $T_\tau$ &  & 0.128 & 0.082 & 0.058 & 0.386 & 0.150 & 0.086 & 0.906 & 0.446 & 0.190 \\ 
$S_\tau^{\max}$ &  & 0.034 & 0.060 & 0.050 & 0.062 & 0.058 & 0.046 & 0.168 & 0.076 & 0.050 \\ 
  \hline
 $S_\tau$&\multirow{3}{*}{256} & 0.122 & 0.058 & 0.026 & 0.716 & 0.226 & 0.082 & 1.000 & 0.828 & 0.268 \\ 
 $T_\tau$ &  & 0.226 & 0.126 & 0.072 & 0.842 & 0.374 & 0.144 & 1.000 & 0.910 & 0.452 \\ 
$S_\tau^{\max}$ &  & 0.058 & 0.030 & 0.056 & 0.146 & 0.044 & 0.066 & 0.578 & 0.106 & 0.076 \\ 

   \hline
\end{tabular}
\label{tab:empConj}
\end{table}

\begin{table}[ht]
\centering
\caption[Table caption text]{Simulated power when contaminating 5\% of
 data generated from $N_m \left(0,   \Sigma_{\text{band}2} \right)$, where $\Sigma_{\text{band}2}=(\sigma_{ij})$ has
diagonal entries $\sigma_{ii}=1$ and off-diagonal entry $\sigma_{ij}=0.1$ if
$1\le |i - j| \leq 2$ and $\sigma_{ij}=0$ if $|i-j| \ge 3$. For each combination of $(m, n)$ and each test, the power is calculated from $500$ independently generated datasets.}
\begin{tabular}{c|c|*{7}{c}}
  \hline
Statistic & $ n \backslash m$ & 4 & 8 & 16 & 32 & 64 & 128 \\ 
  \hline
  $S_r$& \multirow{7}{*}{16} & 0.058 & 0.058 & 0.038 & 0.072 & 0.086 & 0.092 \\ 
   $S_\tau$ &  & 0.074 & 0.090 & 0.094 & 0.096 & 0.116 & 0.120 \\ 
  $T_\tau$  &  & 0.094 & 0.108 & 0.122 & 0.108 & 0.144 & 0.146 \\ 
  $S_{\rho_s}$ &  & 0.034 & 0.068 & 0.056 & 0.070 & 0.076 & 0.074 \\ 
   $T_{\hat{\rho}_s}$ &  & 0.088 & 0.096 & 0.118 & 0.116 & 0.136 & 0.152 \\ 
  $S_{t^*}$ &  & 0.078 & 0.114 & 0.114 & 0.130 & 0.150 & 0.162 \\ 
   $Z_{t^*}$ &  & 0.100 & 0.112 & 0.118 & 0.096 & 0.112 & 0.138 \\ \hline
  $S_r$ & \multirow{7}{*}{32}& 0.072 & 0.100 & 0.078 & 0.110 & 0.106 & 0.104 \\ 
   $S_\tau$ &  & 0.086 & 0.112 & 0.114 & 0.130 & 0.136 & 0.126 \\ 
  $T_\tau$  &  & 0.090 & 0.130 & 0.128 & 0.132 & 0.150 & 0.138 \\ 
  $S_{\rho_s}$ &  & 0.072 & 0.098 & 0.086 & 0.110 & 0.106 & 0.096 \\ 
   $T_{\hat{\rho}_s}$ &  & 0.084 & 0.126 & 0.114 & 0.138 & 0.136 & 0.128 \\ 
  $S_{t^*}$ &  & 0.068 & 0.114 & 0.130 & 0.122 & 0.148 & 0.112 \\ 
   $Z_{t^*}$ &  & 0.088 & 0.120 & 0.130 & 0.118 & 0.146 & 0.116 \\ \hline
  $S_r$ & \multirow{7}{*}{64} & 0.110 & 0.156 & 0.128 & 0.158 & 0.172 & 0.182 \\ 
   $S_\tau$ &  & 0.134 & 0.164 & 0.176 & 0.216 & 0.222 & 0.204 \\ 
  $T_\tau$  &  & 0.138 & 0.176 & 0.182 & 0.220 & 0.240 & 0.202 \\ 
  $S_{\rho_s}$ &  & 0.114 & 0.166 & 0.152 & 0.190 & 0.190 & 0.192 \\ 
   $T_{\hat{\rho}_s}$ &  & 0.134 & 0.176 & 0.180 & 0.204 & 0.228 & 0.200 \\ 
  $S_{t^*}$ &  & 0.110 & 0.168 & 0.148 & 0.184 & 0.184 & 0.168 \\ 
   $Z_{t^*}$ &  & 0.130 & 0.170 & 0.174 & 0.192 & 0.184 & 0.190 \\ \hline
  $S_r$ &\multirow{7}{*}{128} & 0.224 & 0.290 & 0.332 & 0.342 & 0.384 & 0.414 \\ 
   $S_\tau$ &  & 0.306 & 0.390 & 0.408 & 0.436 & 0.454 & 0.484 \\ 
  $T_\tau$  &  & 0.308 & 0.392 & 0.418 & 0.440 & 0.462 & 0.484 \\ 
  $S_{\rho_s}$ &  & 0.296 & 0.376 & 0.392 & 0.418 & 0.444 & 0.470 \\ 
   $T_{\hat{\rho}_s}$ &  & 0.302 & 0.398 & 0.414 & 0.434 & 0.452 & 0.424 \\ 
  $S_{t^*}$ &  & 0.198 & 0.292 & 0.338 & 0.356 & 0.370 & 0.412 \\ 
   $Z_{t^*}$ &  & 0.274 & 0.336 & 0.402 & 0.388 & 0.412 & 0.414 \\ 
   \hline
\end{tabular}
\label{tab:dirtpowernew}
\end{table}

\end{document}